\documentclass[reqno,12pt,draft]{amsart} 

\usepackage{amsmath}
\usepackage{amsfonts}
\usepackage{enumerate}
\usepackage{fancybox}

\newcommand{\bx}{\mathbf{x}}

\newcommand{\1}{{\rm 1\hspace*{-0.4ex}
\rule{0.1ex}{1.52ex}\hspace*{0.2ex}}}

\theoremstyle{plain}
\newtheorem{thm}{Theorem}[section]{\bf}{\it}
\newtheorem{prop}[thm]{Proposition}{\bf}{\it}
\newtheorem{lemma}[thm]{Lemma}{\bf}{\it}
\newtheorem{lem}[thm]{Lemma}{\bf}{\it}
\newtheorem{cor}[thm]{Corollary}{\bf}{\it}
\theoremstyle{definition}
 
\newtheorem{remark}[thm]{Remark}{\it}{\rm}
{\bf}{\rm}
{\rm}{\rm}

{\bf}{\rm}
\newcommand{\supp}{{\operatorname{supp}}}

\newcommand{\Vol}{{\operatorname{Vol}}}
\newcommand{\dist}{{\operatorname{dist}}}

\newenvironment{pf}{\par\medskip\noindent\textit{Proof}:\,}{\hspace*{\fill}\qed\medskip\par\noindent}
\newenvironment{pf*}[1]{\par\medskip\noindent\textit{#1}\,:}{\hspace*{\fill}\qed\medskip\par\noindent}  

\numberwithin{equation}{section}

\newcommand{\R}{{\mathbb R}}
\newcommand{\C}{{\mathbb C}}
\newcommand{\N}{{\mathbb N}}

\title[Pointwise estimates]{Estimates on 
derivatives of Coulombic wave functions and their electron densities}

\thanks{\copyright\ 2018 by the
       authors. This article may be reproduced in its entirety for
       non-commercial purposes. SF was partially supported by a Sapere
       Aude Grant from the Independent Research Fund Denmark, Grant
       number DFF-4181-00221, and by the European Research Council,
       ERC grant agreement 202859. 
}

\author[S. Fournais and T. \O. S\o rensen]
{S{\o}ren Fournais \and
Thomas \O stergaard S\o rensen}

\address[S. Fournais]{Department of Mathematics, Aarhus University, {\,}Ny
  Mun\-ke\-gade 118, DK-8000 Aarhus C, Denmark.
           }
\email{fournais@math.au.dk}

\address[T. \O stergaard S\o rensen]
{Department Mathematisches Institut, LMU Munich, Theresienstrasse 39,
  D-80333 Munich, Germany.}
\email{sorensen@math.lmu.de}

\date{\today}

\begin{document}

\thispagestyle{empty}

\begin{abstract}
We prove {\it a priori} bounds for all derivatives of
non-relativistic Coulombic
eigenfunctions \(\psi\), involving negative powers of the distance to
the singularities of the 
many-body potential. We use these to derive bounds for all
derivatives of the corresponding one-electron densities \(\rho\),
involving negative powers of the distance from the nuclei. The results
are both natural and optimal, as seen from the ground state of Hydrogen.
\end{abstract}

\maketitle

\section{Introduction and results}
In a series of papers \cite{monster,non-iso,KS,thirdder}, the present
authors (together with M. and T. Hoff\-mann-Osten\-hof) have studied the
regularity properties of molecular Coulombic eigenfunctions \(\psi\)
and their electron densities \(\rho\) at the singularities of the
many-body Cou\-lomb potential. For a recent review, see
\cite[pp. 170--178]{Simon-on-Kato}. Some relevant previous works not
mentioned in that review are \cite{HO-alone,HO-Nadira,HO-Stremnitzer}. 

In this paper we take a different approach. Away from these singularities
(where eigenfunctions are real analytic) we prove local
\(L^p\)-estimates on all pointwise derivatives of such eigenfunctions
\(\psi\), with the optimal behaviour in the distance to the
singularities. The estimates are {\it a priori}, so, if \(\psi\)
decays exponentially---as is typically the case for atomic and
molecular eigenfunctions---we get exponential decay of these
estimates. As a corollary we get that all pointwise derivatives of
\(\psi\) belong to certain weighted Sobolev-spaces. This formulation
is inspired by the results in \cite{NistorEtAl}, which we improve and
clarify (see Remarks~\ref{rem:1}(iv) and \ref{rem:2} below for
more details).

We then apply this to obtain estimates on pointwise derivatives of the
corresponding electron density \(\rho\) away from the nuclei, with the
optimal behaviour in the distance to the positions of the nuclei.  

Both types of results are of mathematical interest in themselves, but
also of importance for numerical calculations in Quantum Chemistry.

We now formulate the problem. For simplicity of the presentation (and
only therefore), we restrict our attention to the case of atoms (i.e.,
one nucleus).
Let $H$ be the non-relativistic Schr\"o\-dinger operator of an
$N$-electron atom with nuclear charge $Z$ in the fixed nucleus  
approximation,
\begin{equation}\label{H}
  H=\sum_{j=1}^N\Big(-\Delta_j-\frac{Z}{|x_j|}\Big)
    +\sum_{1\le i<j\le N}\frac{1}{|x_i-x_j|}
   ={}-\Delta+V\,.
\end{equation}
Here the $x_j=(x_{j,1},x_{j,2},x_{j,3})\in \mathbb R^3$, $j=1,\dots,
N$, denote the positions of the electrons, and the $\Delta_j$ are the
associated Laplacians so that $\Delta=\sum_{j=1}^N\Delta_j$ is the
$3N$-dimensional Laplacian. Let ${\bf x}=(x_1,x_2,\dots, x_N)\in
\mathbb R^{3N}$ and let $\nabla=(\nabla_1,\dots, \nabla_N)$ 
denote the $3N$-dimensional gradient operator. 
By abuse of notation, we use \(|\cdot|\) for the Euclidean norm in
both \(\R^3\) and \(\R^{3N}\). 
The operator $H$ is selfadjoint with operator domain $\mathcal
D(H)=W^{2,2}(\mathbb R^{3N})$ and form domain $\mathcal
Q(H)=W^{1,2}(\mathbb R^{3N})$ \cite{Kato-s-a}. 
We are interested in the behaviour of (pointwise) derivatives 
away from the singularities of the potential \(V\) in \eqref{H}
of \(L^2\)-eigenfunctions \(\psi\) of the 
operator \(H\), 
\begin{align}\label{eigen}
  H\psi=E\psi\ , \quad \text{ with }\psi\in W^{2,2}(\mathbb R^{3N})\
  ,\ E\in \R\,. 
\end{align} 

More precisely, let \(\Sigma\) denote the set of coalescence points
(i.e., singularities of \(V\)),  
\begin{align}
  \label{eq:Sigma}
  \Sigma:=\big\{{\bf x}=(x_1,\ldots,x_N)\in\R^{3N}\,\big|\,
           \prod_{j=1}^{N}|x_j|
           \!\!\prod_{1\le i<j\le N}|x_i-x_j|=0\big\}\,.
\end{align}
Note that \(V\) is real analytic in the open set
\(\R^{3N}\setminus\Sigma\). 
Hence if, for some open \(\Omega\subset\R^{3N}\), \(\psi\) is a weak
solution to \eqref{eigen} in \(\Omega\), then \cite[Section 7.5,
pp.\ 177--180]{Hormander} \(\psi\) is real
analytic away from \(\Sigma\), that is, \(\psi\in
C^{\omega}(\Omega\setminus\Sigma)\). In particular, any eigenfunction
\(\psi\in W^{2,2}(\R^{3N})\)  
of the operator \(H\) is real analytic in \(\R^{3N}\setminus\Sigma\). 
Moreover it is known that \(\psi\in C^{0,1}_{\rm
  loc}(\R^{3N})\) \cite[Proposition~1.5]{AHP}, an improvement of
Kato's famous Cusp Condition \cite{Kato-reg}; see also
\cite{Maria-Seiler}. 

Define the distance from a point \({\bf x}\in\R^{3N}\) to a
subset \(K\subseteq\R^{3N}\)  by  
\begin{align}\label{def:dist}
  d({\bf x},K) = \inf \,\big\{ |{\bf x}-{\bf y}|\,\big|\, {\bf
    y}\in K\big\}\,.
\end{align}
Note that
\begin{align}\label{eq:dist-detail}
  d({\bf x},\Sigma) = \min\big\{\,|x_i|,
  \tfrac{1}{\sqrt{2}}|x_j-x_{k}|\,\big|\, i,j,k\in\{1,\ldots,N\},
  j\neq k\,\big\}\,.
\end{align}
More generally, for \(k\in\{1,\ldots,N\}\), let
\(\Sigma^{k}\subset\Sigma\) (the singularities of \(V\) involving
\(x_k\)) be defined by
\begin{align}
  \label{eq:Sigma-k}
  \Sigma^{k}:=\big\{{\bf x}=(x_1,\ldots,x_N)\in\R^{3N}\,\big|\,
           |x_k|
           \!\!\prod_{j=1, j\neq k}^{N}|x_k-x_j|=0\big\}\,;
\end{align}
then \(\Sigma=\cup_{k=1}^{N}\Sigma^{k}\), and we note that, for
\(k\in\{1,\ldots,N\}\),  
\begin{align}\label{eq:dist-detail-k}
  d({\bf x},\Sigma^{k}) = \min\big\{\,|x_k|,
  \tfrac{1}{\sqrt{2}}|x_k-x_{j}|\,\big|\, j\in\{1,\ldots,N\},
  j\neq k\,\big\}\ge d({\bf x},\Sigma)\,.
\end{align}
For \(Q\subseteq\{1,\ldots,N\}\), let \(\Sigma^{Q}:=\cup_{k\in
  Q}\Sigma^{k}\subseteq\Sigma\) (the singularities of \(V\) involving
some \(x_k\) with \(k\in Q\)). This way, \(d({\bf x},\Sigma^{Q})\le
d({\bf x},\Sigma^{k})\) 
for all  \(k\in Q\), and 
\begin{align}\label{eq:dist-detail-Q}\nonumber
  d({\bf x},\Sigma^{Q}) &= \min\big\{\,|x_k|,
  \tfrac{1}{\sqrt{2}}|x_k-x_{j}|\,\big|\, k\in Q, j\in\{1,\ldots,N\},
  j\neq k\,\big\}\\&\ge
  d({\bf x},\Sigma)\,.
\end{align}
For \(\alpha=(\alpha_1,\ldots,\alpha_N)\in (\N_0^{3})^{N}=\N_0^{3N}\),
let
\begin{align}\label{def:Q-alpha}\nonumber
  Q_{\alpha}&:=\big\{ k\in\{1,\ldots,N\}\,\big|
  \alpha_k\neq0\in\N_0^3\,\big\}\subseteq
  \{1,\ldots,N\}\,,
  \\\Sigma^{\alpha}&:=\Sigma^{Q_{\alpha}}\subseteq\Sigma\ ,\quad d_{\alpha}({\bf
    x},\Sigma):=d({\bf x},\Sigma^{\alpha})\,. 
\end{align}
That is, for \(\alpha\in\N_0^{3N}\) fixed, \(\Sigma^{\alpha}\) is the
set of singularities of \(V\) involving the \(x_k\)'s for which
\(\alpha_k\neq0\). 
Hence,
\begin{align}\label{eq:d-alpha}\nonumber
  d_{\alpha}({\bf x},\Sigma)&=\min\big\{\,|x_k|,
  \tfrac{1}{\sqrt{2}}|x_k-x_{j}|\,\big|\, k\in Q_{\alpha}, j\in\{1,\ldots,N\},
  j\neq k\,\big\}\,,\\
  |{\bf x}|&\ge d_{\alpha}({\bf x},\Sigma)\ge d({\bf x},\Sigma)\,.
\end{align}
Note that \(\Sigma^{\alpha}=\Sigma\) (and therefore, \(d_{\alpha}({\bf
  x},\Sigma)=d({\bf x},\Sigma)\)) for all
\(\alpha=(\alpha_1,\ldots,\alpha_N)\in\N_0^{3N}\) for which
\(\alpha_k\neq0\) for all \(k\in\{1,\ldots,N\}\). 

Let \(B_{n}(x,r)\subset\R^{n}\) denote the open ball of centre
\(x\in\R^{n}\) and radius \(r>0\). We recall that for
\(\alpha=(\alpha_1,\ldots,\alpha_N)\in (\N_0^{3})^{N}=\N_0^{3N}\), 
\(\alpha_{k}=(\alpha_{k,1},\alpha_{k,2},\alpha_{k,3})\in \N_0^{3}\),
we let \(|\alpha|=\sum_{k=1}^{N}\sum_{j=1}^{3}\alpha_{k,j}\). 

The first main result of this paper is the following. Our main
interest are the cases \(p=2\) and \(p=\infty\) (for the proof, see
Section~\ref{sec:proof main} below).
\begin{thm}\label{thm:main-one}
Let \(H\) be the non-relativistic Hamiltionian given by \eqref{H}.
Let the singular set
\(\Sigma\subset \R^{3N}\) be defined by 
\eqref{eq:Sigma}, 
and let the distance \(d({\bf x},\Sigma)\) from \({\bf x}\in\R^{3N}\)
to \(\Sigma\) be given by \eqref{eq:dist-detail}. Furthermore,
for every \(\alpha\in\N_{0}^{3N}\), \(|\alpha|\ge
1\), let the corresponding singular set \(\Sigma^{\alpha}\subset
\R^{3N}\) be defined by 
\eqref{def:Q-alpha}, 
and the distance \(d_{\alpha}({\bf x},\Sigma)\) from \({\bf
  x}\in\R^{3N}\) to \(\Sigma^{\alpha}\) be 
given by \eqref{eq:d-alpha}. 

Then:
\begin{itemize}
\item[(i)]
For all \(p\in(1,\infty]\), all \(\alpha\in\N_{0}^{3N}\), \(|\alpha|\ge
1\), all \(0<r<R<1\), and all \(E\in\C\), there
exists a constant \(C=C(p,\alpha,r,R,E)\) (depending also on \(N, Z\))
such that, for all 
\(\psi \in W_{{\rm loc}}^{2,2}(\R^{3N})\) satisfying  
\begin{align}
  \label{eq:eigen-bis}
  H\psi=E\psi\,,
\end{align}
and for all \({\bf x}\in \R^{3N}\setminus \Sigma^{\alpha}\), the
following inequality holds:
\begin{align}\label{eq:resultPSI-Q-alpha}
  \|\partial^{\alpha}&\psi\|_{L^{p}(B_{3N}({\bf x},r\lambda_{\alpha}({\bf x})))}
  \\&\le C \,
  \lambda_{\alpha}({\bf x})^{1-|\alpha|} \big(\|\psi\|_{L^{p}(B_{3N}({\bf
      x},R\lambda_{\alpha}({\bf x})))}
  +\|\nabla\psi\|_{L^{p}(B_{3N}({\bf
      x},R\lambda_{\alpha}({\bf x})))}\big)\,,
  \nonumber
\end{align} 
where \(\lambda_{\alpha}({\bf x}):=\min\big\{1, d_{\alpha}({\bf
  x},\Sigma)\big\}\). 
\item[(ii)]
For all \(p\in(1,\infty]\), all \(\alpha\in\N_{0}^{3N}\), \(|\alpha|\ge
1\), all \(0<r<R<1\), and all \(E\in\C\), there
exists a constant \(C=C(p,\alpha,r,R,E)\) (depending also on \(N, Z\))
such that, for all 
\(\psi \in W_{{\rm loc}}^{2,2}(\R^{3N})\) satisfying  \eqref{eq:eigen-bis},
and for all \({\bf x}\in \R^{3N}\setminus \Sigma\), the
following inequality holds:
\begin{align}\label{eq:resultPSIa}
  \|\partial^{\alpha}&\psi\|_{L^{p}(B_{3N}({\bf x},r\lambda({\bf x})))}
  \\&\le C \,
  \lambda({\bf x})^{1-|\alpha|} \big(\|\psi\|_{L^{p}(B_{3N}({\bf
      x},R\lambda({\bf x})))}
  +\|\nabla\psi\|_{L^{p}(B_{3N}({\bf
      x},R\lambda({\bf x})))}\big)\,,
  \nonumber
\end{align}
where \(\lambda({\bf x}):=\min\big\{1, d({\bf
  x},\Sigma)\big\}\). 
\end{itemize}
\end{thm}

As a corollary of the case \(p=\infty\) we have the following
pointwise estimates, one of our main motivations for the  study of
these problems (for the proof, see Section~\ref{sec:proof-cor} below):
\begin{cor}\label{cor:main}
Let the notation and assumptions be as in Theorem~\ref{thm:main-one}
above and let 
\(\alpha\in\N_{0}^{3N}\), \(|\alpha|\ge 
1\), and \(R>0\). Then:
\begin{itemize}
\item[(i)]
There exists a constant \(C=C(\alpha,R, E)\) (depending also on \(N,
Z\)) such that
for all \({\bf x}\in \R^{3N}\setminus \Sigma^{\alpha}\),
\begin{align}\label{eq:pointwise}
  |\partial^{\alpha}\psi({\bf x})|
  \le C \,
  \lambda_{\alpha}({\bf x})^{1-|\alpha|}
  \|\psi\|_{L^{\infty}(B_{3N}({\bf x},R))}\,,
\end{align}
where \(\lambda_{\alpha}({\bf x})=\min\big\{1, d_{\alpha}({\bf
  x},\Sigma)\big\}\). 
\item[(ii)]
There exists a constant \(C=C(\alpha,R, E)\) (depending also on \(N,
Z\)) such that
for all \({\bf x}\in \R^{3N}\setminus \Sigma\),
\begin{align}\label{eq:pointwise-bis}
  |\partial^{\alpha}\psi({\bf x})|
  \le C \,
  \lambda({\bf x})^{1-|\alpha|}
  \|\psi\|_{L^{\infty}(B_{3N}({\bf x},R))}\,,
\end{align}
where \(\lambda({\bf x})=\min\big\{1, d({\bf
  x},\Sigma)\big\}\). 
\end{itemize}
\end{cor}
If \(\psi\) is an eigenfunction of \(H\) (that is,
\(\psi\in W^{2,2}(\R^{3N})\) and \(\psi\)
satisfies \eqref{eq:eigen-bis} for some \(E\in\R\)), and \(\psi\) also
decays exponentially, then we have the following corollary to
Theorem~\ref{thm:main-one} (for the proof, see
Section~\ref{sec:proof-cor} below).
\begin{cor}\label{cor:exp} With the notation and assumptions as in
Theorem~\ref{thm:main-one}, assume \(\psi\) is an eigenfunction of
\(H\) (that is, \(\psi\in W^{2,2}(\R^{3N})\) and \(\psi\)
satisfies \eqref{eigen} for some \(E\in\R\)). Assume
furthermore that \(E\) and \(\psi\) are such that 
there exist constants \(C_0,c_{0}>0\) such that
\begin{align}
  \label{eq:exp-dec}
  |\psi({\bf x})|\leq C_0\, {\rm e}^{-c_{0}|{\bf x}|}\quad\text{ for
    all } {\bf x}\in\mathbb R^{3N}\,.
\end{align}

Then for all multiindices \(\alpha\in\N_{0}^{3N}\) with \(|\alpha|\ge
1\):
\begin{itemize}
\item[(i)]
There exist constants \(C_{\alpha}, c_{\alpha}>0\) such that,
for all \({\bf x}\in \R^{3N}\setminus \Sigma^{\alpha}\),
\begin{align}
  \label{eq:resultPSIb}
     \big|\partial^{\alpha}\psi({\bf x})\big|\le C_{\alpha} \,d_{\alpha}({\bf
       x},\Sigma)^{1-|\alpha|}{\rm e}^{-c_{\alpha}|{\bf x}|}\,.
\end{align}
\item[(ii)]
There exist constants \(C_{\alpha}, c_{\alpha}>0\) such that,
for all \({\bf x}\in \R^{3N}\setminus \Sigma\),
\begin{align}
  \label{eq:resultPSIc}
     \big|\partial^{\alpha}\psi({\bf x})\big|\le C_{\alpha} \,d({\bf
       x},\Sigma)^{1-|\alpha|}{\rm e}^{-c_{\alpha}|{\bf x}|}\,.
\end{align}
\item[(iii)]
\begin{align}\label{eq:Sob-a}
    d_{\alpha}(\,\cdot\,,\Sigma)^{|\alpha|-a}\partial^{\alpha}\psi \in
    L^{2}(\R^{3N}\setminus \Sigma^{\alpha})\ \text{ for all } a<\tfrac{5}{2}\,.
\end{align}
\item[(iv)]
\begin{align}\label{eq:Sob-bis}
  d(\,\cdot\,,\Sigma)^{|\alpha|-a}\partial^{\alpha}\psi \in 
    L^{2}(\R^{3N}\setminus \Sigma)\ \text{ for all } a<\tfrac{5}{2}\,.
\end{align}
\end{itemize}
\end{cor}
\begin{remark}\label{rem:L^2}
Of course, \(\Sigma^{\alpha}, \Sigma\subset\R^{3N}\) are of Lebesgue
measure zero, so \(L^{2}(\R^{3N}\setminus
\Sigma^{\alpha})=L^{2}(\R^{3N}\setminus\Sigma)=L^{2}(\R^{3N})\). However,
in \eqref{eq:Sob-a}--\eqref{eq:Sob-bis} we want to emphasize that the 
derivatives \(\partial^{\alpha}\psi\) are the {\it pointwise} (classical)
derivatives in \(\R^{3N}\setminus 
\Sigma\) and \(\R^{3N}\setminus\Sigma^{\alpha}\) (which exist), and
{\it not} weak derivatives in \(\R^{3N}\) (on which we have no 
statements). The same remark holds
for \eqref{eq:Sob-two}--\eqref{eq:Sob-three} below.
\end{remark}

Using Theorem~\ref{thm:main-one} with \(p=2\), we can prove the
following variant of Corollary~\ref{cor:exp} (iii)--(iv), which does
not assume any decay of \(\psi\) (apart from \(\psi\in
W^{2,2}(\R^{3N})\); for the proof, see Section~\ref{sec:proof-Nistor} below):
\begin{thm}\label{thm:Nistor}
Let the notation and assumptions be as in
  Theorem~\ref{thm:main-one}, and assume furthermore that
  \(\psi\in W^{2,2}(\R^{3N})\). 
Then, for all multiindices \(\alpha\in\N_{0}^{3N}\) with \(|\alpha|\ge
1\) we have
\begin{align}\label{eq:Sob-two}
    \lambda_{\alpha}^{|\alpha|-a}\partial^{\alpha}\psi \in
    L^{2}(\R^{3N}\setminus \Sigma^{\alpha})\ \text{ for all } a<\tfrac{5}{2}\,,
\end{align}
and
\begin{align}\label{eq:Sob-three}
    \lambda^{|\alpha|-a}\partial^{\alpha}\psi \in
    L^{2}(\R^{3N}\setminus \Sigma)\ \text{ for all } a<\tfrac{5}{2}\,,
\end{align}
where \(\lambda_{\alpha}({\bf x})=\min\big\{1, d_{\alpha}({\bf
  x},\Sigma)\big\}\) and
  \(\lambda({\bf x})=\min\big\{1, d({\bf
  x},\Sigma)\big\}\). 

In fact, for all  \(\alpha\in\N_{0}^{3N}\) with \(|\alpha|\ge
1\) there exists \(C_{\alpha}>0\) such that
\begin{align}\label{eq:Sob-apriori-1}
  \|\lambda_{\alpha}^{|\alpha|-a}\partial^{\alpha}\psi\|_{L^2(\R^{3N}\setminus
    \Sigma^{\alpha})}
  \le C_{\alpha}\|\psi\|_{W^{2,2}(\R^3N)}\,, \\  \label{eq:Sob-apriori-2}
  \|\lambda^{|\alpha|-a}\partial^{\alpha}\psi\|_{L^2(\R^{3N}\setminus
    \Sigma^{\alpha})}
  \le C_{\alpha}\|\psi\|_{W^{2,2}(\R^3N)}\,.
\end{align}
\end{thm}
We now give some remarks on the results stated above.
\begin{remark}\label{rem:1}
  \(\, \)
\begin{enumerate}
\item[\rm (i)] For \(\alpha\in\N_{0}^{3N}\) with \(|\alpha|=1\),
  \eqref{eq:resultPSIb}--\eqref{eq:resultPSIc} 
  was proved in \cite[Theorem~1.2; Remark~1.9]{AHP}.  
\item[\rm (ii)] 
The ground state function  
of Hydrogen (that is, of the
  operator \(-\Delta_{x}-\tfrac{1}{|x|}\), \(x\in\R^3\), \(N=1\)) is
\(\phi_{0}(x)=c_{0}{\rm e}^{-|x|/2}\). In this case
\(\Sigma=\{0\}\subset\R^{3}\) and \(d(x,\Sigma)=|x|\). This example shows that
the results in Theorem~\ref{thm:main-one}, Corollary~\ref{cor:main},
Corollary~\ref{cor:exp}, and Theorem~\ref{thm:Nistor} 
are both natural and optimal.
\item[\rm (iii)]  As will be clear from the proofs,
  Theorem~\ref{thm:main-one}, Corollary~\ref{cor:main},
  Corollary~\ref{cor:exp}, and 
  Theorem~\ref{thm:Nistor} 
  generalize to the case of mole\-cules (i.e., \(K\) nuclei with positive
  charges \(Z_1,\ldots,Z_K>0\), fixed at
  positions \(R_1,\ldots, R_K\) in \(\R^3\)) in the obvious way. In
  this case (compare with \eqref{H}),
\begin{align}\label{V-molecule}
   V({\bf x}) = \sum_{j=1}^{N}\sum_{k=1}^{K}{}\Big( -\frac{Z_k}{x_j-R_k}\Big) 
  +\sum_{1\le i<j\le N}\frac{1}{|x_i-x_j|}\,.
\end{align}
\item[\rm (iv)] The local version (i.e., without the exponential
  decay) of \eqref{eq:resultPSIc} for \(N=1\) was already known: It follows
  immediately from \cite[Theorem~1.1]{KS} (also in the case of several
  nuclei). In fact, more generally, for any \(N\) the
  local version of 
  \eqref{eq:resultPSIb} for points \({\bf x}\) in a small
  neighbourhood of so-called `two-particle coalescence points' follows
  from \cite[Theorem~1.4]{KS}.
\item[\rm (v)]
For references on the exponential decay of eigenfunctions (i.e.,
\eqref{eq:exp-dec}), see e.g.\ 
Froese and Herbst~\cite{froese-herbst} and
Simon~\cite{Si-semi, Simon_notes}. Exponential decay is known to hold for
any eigenfunction \(\psi\) associated to an
eigenvalue \(E\)
which is not a so-called `thres\-hold energy'. This includes (but is
not restricted to) any
eigenvalue below the essential spectrum in any symmetry subspace, for
instance, the fermionic ground state energy.
\end{enumerate}
\end{remark}
\begin{remark}\label{rem:2}
As seen from the example in Remark~\ref{rem:1}(ii), the
  natural critical value for \(a\) in
  \eqref{eq:Sob-a}--\eqref{eq:Sob-three} is \(5/2\), as long as
  \(|\alpha|\ge1\). Hence, the result is optimal.
However, for \(\alpha=0\), one has
  \eqref{eq:Sob-a}--\eqref{eq:Sob-three} only for all \(a<3/2\); this
  follows from \eqref{eq:apriori_rho_coulomb}
  in
  Proposition~\ref{prop:AprioriEstimatesOnRho} below. The example in
  Remark~\ref{rem:1}(ii) again shows that this is optimal.

  The statements in
  \eqref{eq:Sob-a}--\eqref{eq:Sob-three} can be re-formulated in terms of
  certain weight\-ed Sobolev-spaces: Define the following spaces (called
  `Babu\v{s}ka-Kondratiev' spaces in \cite{NistorEtAl}; recall that
  \(\lambda({\bf x})=\min\{1, d({\bf 
  x},\Sigma)\}\)):
\begin{align}\label{def:weighted-Sob}
   \mathcal{K}_{a}^{m}\big(&\R^{3N}\setminus\Sigma, \lambda\big)
     \nonumber \\&=\big\{ u:\R^{3N}\to \C \,\big|\,
     \lambda^{|\alpha|-a}\partial^{\alpha}u\in
     L^{2}(\R^{3N}\setminus \Sigma)\,, |\alpha|\le m\big\}\,. 
\end{align}
Then it follows from Theorem~\ref{thm:Nistor}, and the remark above,
that any eigenfunction \(\psi\in W^{2,2}(\R^{3N})\) of the operator
\(H\) in \eqref{H}  
belongs to \(\mathcal{K}_{a}^{m}\big(\R^{3N}\setminus \Sigma,
\lambda\big)\) for any \(m\in\N\) and any
\(a<3/2\). However, Theorem~\ref{thm:Nistor} gives much more, since the
restriction on \(a\) is only due to the case \(\alpha=0\): It also
follows that, for any \(|\alpha|=1\), also \(\partial^{\alpha}\psi\)
belongs to \(\mathcal{K}_{a}^{m}\big(\R^{3N}\setminus \Sigma, 
\lambda\big)\) for any \(m\in\N\) and any
\(a<3/2\). The example in
  Remark~\ref{rem:1}(ii) again gives optimality.

In the case of exponentially decaying eigenfunctions,
Corollary~\ref{cor:exp} gives the same statements, but with
\(\mathcal{K}_{a}^{m}\big(\R^{3N}\setminus \Sigma,  
\lambda\big)\) replaced by \\
\(\mathcal{K}_{a}^{m}\big(\R^{3N}\setminus \Sigma, 
d(\,\cdot\,,\Sigma)\big)\). It is natural that without any additional decay
assumptions on \(\psi\), one can only expect to get this type of result with a
`regularised distance function' like \(\lambda({\bf x})=\min\big\{1, d({\bf
  x},\Sigma)\big\}\). 
This vastly improves and clarifies the results proved and the conjectures
  stated in \cite{NistorEtAl} (which were also for a regularised
  distance function). (See also Remark~\ref{rem:1}(iv) above.) 

  Note that the
  results in \cite{NistorEtAl} are stated for slightly more general
  potentials \(V\) than the one in \eqref{H}: Let 
  \begin{align}\label{eq:general-V}
    W({\bf x})=\sum_{j=1}^N \frac{b_j(\tfrac{x_j}{|x_j|})}{|x_j|}
    +\sum_{1\le i<j\le N}\frac{c_{ij}(\tfrac{x_i-x_j}{|x_i-x_j|})}{|x_i-x_j|}\,,
  \end{align}
  with \(b_j,c_{ij}\in C^{\infty}(\mathbb{S}^{2})\)
  (with \(\mathbb{S}^{2}\) the unit sphere in \(\R^3\)). Then all our results
  hold with the operator \(H\) in \eqref{H} replaced with
  \({}-\Delta+W\). For 
  simplicity of the presentation, we
  have chosen to stick to the physically most relevant case of atoms
  and molecules, as in \eqref{H}.

For another approach, via a singular pseudo-differential operator
calculus, giving a parametrix for the resolvent of \(H\) in \eqref{H} in the case
of Hydrogen (\(N=1\)) \cite{Flad-1} and Helium (\(N=2\)) \cite{Flad-2} with the correct
asymptotic behaviour at {\it two}-particle coalescence points, see
\cite{Flad-3}. 
\end{remark}

An important quantity derived from any eigenfunction \(\psi\) of the operator
\(H\) in \eqref{H} is its associated one-electron density \(\rho\) defined by
\begin{align}\label{rho}
  \rho(x)\equiv\rho_{\psi}(x)=&\sum_{j=1}^N\rho_j(x)=
  \sum_{j=1}^N
  \int_{\mathbb R^{3N-3}}|\psi(x,\hat{\bf{x}}_j)|^2
  d\hat{\bf{x}}_j\,,
\end{align}
where we have introduced the notation
\begin{align}\label{notation: hat x}
  \hat{\bf{x}}_j&=(x_1,\dots,x_{j-1}, x_{j+1},\dots, x_N)
  \intertext{ and }
  d\hat{\bf{x}}_j&=
  dx_1\dots dx_{j-1} dx_{j+1}\dots dx_N\,.
  \label{notation:dx}
\end{align} 
By abuse of notation, we identify \((x_1,\dots,x_{j-1},x,
x_{j+1},\dots, x_N)\) and  \((x,\hat{\bf x}_j)\). 

The regularity properties of \(\rho\) at the origin (or,
more generally, at the positions of the nuclei, when studying a
molecule, see Remark~\ref{rem:1}(iii) above) have been studied recently in
\cite{non-iso,thirdder} (see also \cite{Bingel,AHP,Steiner}). In
\cite{analytic} it was proved that  
\(\rho\) is real analytic away from the position of the nucleus (i.e., \(\rho\in
C^{\omega}(\R^{3}\setminus\{0\})\)); for another recent 
proof of this, see \cite{Jecko} (see also \cite{rho-smooth,Taxco}).
(This result is known as the {\it Holographic Electron Density
  Theorem} (HEDT) in Quantum Chemistry; see \cite{Mezey}.)

Our main result on \(\rho\) in this paper is the following.

\begin{thm}\label{thm:main-rho}
Let \(H\) be the non-relativistic Hamiltionian given by \eqref{H}, and
assume that \(\psi \in W^{2,2}(\R^{3N})\) satisfies, for some
\(E\in\R\), 
\begin{align}
  \label{eq:eigen-bis-2}
  H\psi=E\psi\,.
\end{align}
Define the associated one-electron density \(\rho\) as in \eqref{rho}.

Then, for all multiindices \(\alpha\in\N_{0}^{3}\) with \(|\alpha|\ge
1\):

\begin{itemize}
\item[(i)] For all \(R>0\) there exists a constant \(C_{\alpha}(R)>0\) such that 
\begin{align}\label{est:pointwise-rho}
  |\partial^{\alpha}\rho(x)|\le C_{\alpha}(R)\,
  r(x)^{1-|\alpha|}\int_{B_3(x,R)}\!\!\!\!\!\!\!\!
  \rho(y)\,dy\ \text{ for all }
  x\in\R^{3}\setminus\{0\}\,, 
\end{align}
where \(r(x):=\min\{1,|x|\}\), \(x\in\R^3\). \\\noindent
In particular, \(r^{|\alpha|-a}\partial^{\alpha}\rho\in
L^{\infty}(\R^{3}\setminus\{0\})\) for all \(a\in[0,1]\), with   
\begin{align}\label{eq:K-space-rho}
  \|r^{|\alpha|-a}\partial^{\alpha}\rho\|_{L^{\infty}(\R^{3}\setminus\{0\})}\le
  C_{\alpha}\|\rho\|_{L^1(\R^{3})}=C_{\alpha}\|\psi\|_{L^2(\R^{3N})}^2\,. 
\end{align} 
\item[(ii)] Furthermore, 
for all \(p\in [0,\infty)\) and all
  \(a\in[0,\tfrac{p+3}{p})\), there exists a constant
  \(C_{\alpha}(a,p)>0\) such that 
\begin{align}\label{eq:rho-Babuschka-final}
  \|r^{|\alpha|-a}\partial^{\alpha}\rho\|_{L^{p}(\R^{3}\setminus\{0\})}
  \le C_{\alpha}(a,p)\|\rho\|_{L^1(\R^{3})}\,.
\end{align}
In particular, \(r^{|\alpha|-a}\partial^{\alpha}\rho\in
L^{p}(\R^{3}\setminus\{0\})\) for all \(p\in[1,\infty)\) and all
\(a\in[0,\tfrac{p+3}{p})\). 
\item[(iii)] Under the decay assumption
\eqref{eq:exp-dec}, \(r(x)\) can be replaced with \(|x|\) above:
 \(|\cdot|^{|\alpha|-a}\partial^{\alpha}\rho\in
L^{p}(\R^{3}\setminus\{0\})\) for all \(p\in[1,\infty)\) and all
\(a\in[0,\tfrac{p+3}{p})\), and all \(a\in[0,1]\) for \(p=\infty\).
In fact,
if we assume exponential decay of \(\psi\) (i.e., there exist constants
\(C_0,c_{0}>0\) such 
that \eqref{eq:exp-dec} holds), then for all multiindices
\(\alpha\in\N_{0}^{3}\) with \(|\alpha|\ge 
1\) there exist constants \(C_{\alpha},c_{\alpha}>0\) such that
\begin{align}\label{eq:resultRHO}
  |\partial^{\alpha}\rho(x)|
  \le C_{\alpha} |x|^{1-|\alpha|}
  {\rm e}^{-c_{\alpha}|x|}
\end{align}
for all $x \in \R^{3}\setminus\{0\}$.
\end{itemize}
\end{thm}
\begin{remark}\label{rem:3}
  \(\, \)
\begin{enumerate}
\item[\rm (i)]
Again, the example in Remark~\ref{rem:1}(ii) above (for which
\(\rho(x)=c_{0}^2{\rm e}^{-|x|}\)) shows that
the results in Theorem~\ref{thm:main-rho} are both natural and optimal.
\item[\rm (ii)]
As will be clear from the proof, also this result 
  generalizes to the case of molecules (see Remark~\ref{rem:1}(iii) above)
  in the obvious way. 
\item[\rm (iii)] The corresponding local version of
  \eqref{eq:resultRHO} near \(x=0\) for the case of the one-electron density of
  Hartree-Fock states (i.e., Slater-deter\-minants of solutions to the
  Hartree-Fock equations) follows from \cite[Corollary~1.5]{ks-hf}.
  It says that in this case 
 there exist \(\varepsilon>0\) 
and real analytic  functions 
  \(\rho_{1}, \rho_{2}:B_3(0,\varepsilon)\to \R\) 
(i.e., \(\rho_{1}, \rho_{2}\in
C^{\omega}(B_3(0,\varepsilon))\)), such that
\begin{align}
  \label{eq:density-analytic}
  \rho(x)=\rho_{1}(x)+|x|\,\rho_{2}(x)
  \ \text{ for all }\ x\in B_3(0,\varepsilon)\,.
\end{align}
See also \cite{Flad-et-al-1, Flad-et-al-2} for related work.
It would be interesting to determine whether 
the same result holds in the (present) Schr{\"o}dinger case (recall
that then \(\rho\in 
C^{\omega}(\R^{3}\setminus\{0\})\)).
Note that, as for Har\-tree-Fock, \eqref{eq:resultRHO} (near \(x=0\))
would follow from such a 
result. 
\item[\rm (iv)] 
The statements in \eqref{eq:K-space-rho}--\eqref{eq:rho-Babuschka-final}
can again be re-formulated in terms of weight\-ed Sobolev-spaces (see
also Remark~\ref{rem:2} above): 
Define the following spaces (recall that 
  \(r(x)=\min\{1, |x|\}\)):
\begin{align}\label{def:weighted-Sob-bis}
   \mathcal{K}_{a}^{m,p}\big(&\R^{3}\setminus\{0\}, r \big)
     \nonumber \\&=\big\{ f:\R^{3}\to \C \,\big|\,
     r^{|\alpha|-a}\partial^{\alpha}f\in
     L^{p}(\R^{3}\setminus\{0\})\,, |\alpha|\le m\big\}\,. 
\end{align}
Then it follows from Theorem~\ref{thm:main-rho}
that for the electron density \(\rho\) (given by \eqref{rho})
of any eigenfunction \(\psi\in W^{2,2}(\R^{3N})\) of the operator
\(H\) in \eqref{H}, and any \(|\alpha|=1\), \(\partial^{\alpha}\rho\)  
 belongs to 
\(\mathcal{K}_{a}^{m,p}\big(\R^{3}\setminus \{0\}, r\big)\) for every \(m\in\N\), for any
 \(a\in[0,1]\) if \(p=\infty\), and any \(a\in[0,\tfrac{p+3}{p})\) if \(p\in[1,\infty)\).
\item[\rm (v)] 
For precise information on the behaviour at infinity of \(\rho\) itself 
(f.ex., similar to \eqref{eq:resultRHO}, but for \(\alpha=0\)), see 
\cite{H-O-Ahlrichs-Morgan-2, H-O-2-schr-ineq,H-O-Ahlrichs-Morgan-1}.
\end{enumerate}
\end{remark}

An important ingredient in the proof of Theorem~\ref{thm:main-rho} is
an estimate on derivatives of \(\psi\) along certain singularities of
\(V\) (`parallel derivatives'; see also
\cite[Proposition~2]{rho-smooth}, \cite[Lemma~2.2]{Taxco}, 
\cite[Lemma~3.1]{analytic}). Since this estimate is interesting in itself, we
formulate it here.

First we need some additional notation. For 
\(Q\subset\{1,\ldots,N\}\), \(Q\neq\emptyset\), define the
(`centre of mass') coordinate \(x_{Q}\in\R^{3}\) by
 \begin{align}\label{def:x_Q}
   x_Q:=\frac{1}{\sqrt{|Q|}}\sum_{j\in Q}x_j\,.
 \end{align} 
We now define
\(\partial_{x_Q}^{e_s} f\), \(s=1,2,3\), for a function \(f \in C^1({\mathbb
  R}^{3N})\) and \(e_{s}\) the canonical unit vectors in \(\R^{3}\). 
For the given \(Q\) and \(s\), let \({\mathbf v}=(v_1,\ldots,v_N) \in
{\mathbb R}^{3N}\) with \(v_j = 0\) for \(j \notin Q\), and \(v_j =
e_s/\sqrt{|Q|}\) for \(j 
\in Q\).
Then we define
\begin{align}\label{def:der-x_q}
  \partial_{x_Q}^{e_s} f({\bf x}): = \nabla f({\bf x}) \cdot {\mathbf
    v}=\frac{1}{\sqrt{|Q|}}\sum_{j\in Q}\frac{\partial f}{\partial
    x_{j,s}}({\bf x})
  =\Big(\frac{1}{\sqrt{|Q|}}\sum_{j\in Q}\partial_{x_{j,s}} f\Big)({\bf
    x})\,.
\end{align}
The definition of \(\partial_{x_Q}^{\alpha}\) then follows
by iteration
for any \(\alpha=(\alpha_1,\alpha_2,\alpha_2) \in {\mathbb N}^3_{0}\):
\begin{align}\label{def:d-x-Q}
  \partial_{x_Q}^{\alpha}f
  =\Big[\prod_{s=1}^{3}
  \Big(\frac{1}{\sqrt{|Q|}}\sum_{j\in
    Q}\partial_{x_{j,s}}\Big)^{\alpha_s}\Big] f\,.
\end{align}
In particular, if \(Q=\{j,k\}\), \(j,k\in\{1,\ldots,N\}\), \(j\neq
k\), then
\begin{align}\label{der-j+k}
  \partial_{x_j+x_k}^{\alpha}f:=\partial_{x_Q}^{\alpha}f
  =\Big[\prod_{s=1}^{3}\Big(\frac{1}{\sqrt2}\big(\partial_{x_j,s}+\partial_{x_k,s}\big)^{\alpha_s}\Big)\Big]f\,. 
\end{align}
It follows that if \(Q\subseteq\{1,\ldots,N\}\), \(Q\neq\emptyset\),
and \(f({\bf x})=g(x_j-x_k)\) for some \(j,k\in Q\) and
\(g:\R^3\to\R\), then
\begin{align}\label{parallel-der-zero}
  \partial_{x_Q}^{\alpha}f=\partial_{x_j+x_k}^{\alpha}f=\partial_{x_j+x_k}^{\alpha}g=0\,.
\end{align}

One can clearly reformulate these definition in terms of Fourier
transforms (multiplication by \(\xi_{Q}^{e_s}\) for suitably defined
\(\xi_{Q}\) in Fourier space). In a previous paper \cite{rho-smooth} we used
a coordinate transformation to describe these derivatives. 

Furthermore, we define (notice that generally $\Sigma_{Q}$ is
different from the previously defined $\Sigma^{Q}$) 
\begin{align}\label{def:Sigma-parallel}
  \Sigma_{Q}:=\big\{{\bf x}=(x_1,\ldots,x_N)\in\R^{3N}\,\big|\,
  \prod_{j\in Q}|x_j|
  \!\!\prod_{j\in Q, k\not\in Q}|x_j-x_k|=0\big\}\,,
\end{align}
so that
\begin{align}\label{eq:dist-Q-parallel}
      d_{Q}({\bf x},\Sigma):= d({\bf x},\Sigma_{Q})
     = \min\big\{\,|x_j|,
    \tfrac{1}{\sqrt{2}}|x_j-x_{k}|\,\big|\, j\in Q, k\not\in Q\}\,.
\end{align}
We then have the following estimate, concerning derivatives
\(\partial_{x_{Q}}^{\alpha}\) of local solutions \(\psi\)
along/parallel to the singularity \(\Sigma_{Q}\):
\begin{prop}\label{prop:main-three}
Let \(H\) be the non-relativistic Hamiltionian given by \eqref{H}.
For any \(Q\subset\{1,\ldots,N\}\), \(Q\neq\emptyset\), 
let the singular set \(\Sigma_{Q}\subset \R^{3N}\) be defined by
\eqref{def:Sigma-parallel}, 
and let the distance \(d_{Q}({\bf x},\Sigma)\) from \({\bf
  x}\in\R^{3N}\) to \(\Sigma_{Q}\) be  
given by \eqref{eq:dist-Q-parallel}. Furthermore, for 
\(\alpha\in\N_{0}^{3}\), \(|\alpha|\ge
1\), let \(\partial_{x_Q}^{\alpha}\) be defined by \eqref{def:d-x-Q}.

Then:
For all \(p\in(1,\infty]\), all \(Q\subset\{1,\ldots,N\}\),
\(Q\neq\emptyset\), all \(\alpha\in\N_{0}^{3}\), \(|\alpha|\ge
1\), all \(0<r<R<1\), and all \(E\in\C\), there
exists a constant \(C=C(p,Q,\alpha,r,R,E)\) (depending also on \(N, Z\))
such that for all 
\(\psi \in W_{{\rm loc}}^{2,2}(\R^{3N})\) satisfying  
\begin{align}
  \label{eq:eigen-parallel}
  H\psi=E\psi\,,
\end{align}
and for all \({\bf x}\in \R^{3N}\setminus \Sigma_{Q}\), the
following inequality holds:
\begin{align}\label{eq:resultPSI-Q-alpha-parallel}
  \|\partial_{x_{Q}}^{\alpha}&\psi\|_{L^{p}(B_{3N}({\bf x},r\lambda_{Q}({\bf x})))}
  \\&\le C \,
  \lambda_{Q}({\bf x})^{1-|\alpha|} \big(\|\psi\|_{L^{p}(B_{3N}({\bf
      x},R\lambda_{Q}({\bf x})))}
  +\|\nabla\psi\|_{L^{p}(B_{3N}({\bf
      x},R\lambda_{Q}({\bf x})))}\big)\,,
  \nonumber
\end{align} 
where \(\lambda_{Q}({\bf x})=\min\big\{1, d_{Q}({\bf
  x},\Sigma)\big\}\). 
\end{prop}

\ %

\subsection{Organisation of the paper and strategy of the proofs}

\ %

The first main idea of the proofs of Theorem~\ref{thm:main-one} (in
Section~\ref{sec:proof main} below) and
Proposition~\ref{prop:main-three} (in Section~\ref{sec:proof Q} below)  
(see also Proposition~\ref{prop:our a priori} in Appendix~\ref{sec:our
  a priori} below) is an 'Ansatz',
\(\psi=e^{F}\psi_{F}\), for the solution of \(H\psi=E\psi\), for
various suitable (see below), slightly different, 
choices of \(F\) (see \eqref{def:F-alpha}, \eqref{def:F-alpha-Q}, and
\eqref{def:tilde-F} below). The function \(e^{F}\) is often 
called a `Jastrow factor' in the Chemistry literature \cite{Jastrow}.
In the mathematical study of Coulombic eigenfunctions, it was
introduced in \cite{Leray} (with \(F=\widetilde{F}\) in
\eqref{def:tilde-F} below). It was applied (with the same \(F\)) to
study unique continuation in \cite[Corollary~4.1;
(4.7)]{HO-Stremnitzer} and regularity from \cite{AHP} onwards.
Using that \(H\psi=E\psi\), the function \(\psi_{F}\) solves the equation
\begin{align}\label{eq:resulting-psi-F}
  -\Delta\psi_{F} -2\nabla F\cdot\nabla\psi_{F}+ \big(V-\Delta F -
  |\nabla F|^2-E)\psi_{F} = 0\,.
\end{align}

The second main idea is to re-scale the resulting equation
\eqref{eq:resulting-psi-F}, from a ball around a (fixed) \({\bf
  x}\in\R^{3N}\) (away from the relevant singularity of \(V\)) of the size of
the distance \(d\) from \({\bf x}\) to the relevant part of the
singular set \(\Sigma\) (i.e., \(d_{\alpha}({\bf x},\Sigma)\) or
\(d_{Q}({\bf x},\Sigma)\)), to a ball of size one around \({\bf x}\). 
The \(F\) above has been
chosen such that, by the homogeneity of the potential \(V\) (see
\eqref{H}), 
this re-scaled equation
has coefficients whose (relevant) derivatives are either zero (see
\eqref{eq:der-K-H-zero} 
and \eqref{eq:der-K-H-zero-Q}), or are uniformly
bounded on compact subsets of the unit ball (see
\eqref{eq:der-V--alpha-final} and \eqref{eq:der-V-alpha-Q}). For this
to work, one needs to work with \(\lambda=\min\{1,d\}\), not
\(d\). 

Successive differentiation of this re-scaled equation (with respect to
the relevant variable), and application 
of standard elliptic regularity theory (\(C^{1,\theta}\) and
\(W^{2,p}\); see Appendix~\ref{sec:app1} below) to the resulting
equations, produces {\it a priori} estimates (on balls
of size slightly less than one, hence the \(r\) and \(R\) in the
theorems) with constants 
{\it independent} of \({\bf x}\). This fact is the essential part of the
argument.

Scaling back these {\it a priori} estimates for \(\alpha\)
derivatives delivers   
the explicit dependence (in \(\alpha\)) on the distance  \(d\) to the
relevant part of the singular set \(\Sigma\) (or rather,
on the corresponding \(\lambda\)) of the {\it a priori}
bounds of \(\partial^{\alpha}\psi_{F}\) on balls of the size of this
distance around \({\bf x}\) (see 
\eqref{eq:resultPSI-Q-alpha}--\eqref{eq:resultPSIa} and
\eqref{eq:resultPSI-Q-alpha-parallel} above).
An extra argument/iteration is
needed to get the optimal behavior in the number of derivatives. This
is assured by an {\it a priori} estimate on first derivatives, see
Proposition~\ref{prop:our a priori} in Appendix~\ref{sec:our a
  priori}. The estimates for \(\partial^{\alpha}\psi\) follow by the
properties of \(F\).

The (short) proofs of Corollaries~\ref{cor:main} and ~\ref{cor:exp}
can be found in Section~\ref{sec:proof-cor} below.

The proof of Theorem~\ref{thm:Nistor} (in
Section~\ref{sec:proof-Nistor} below) consists in carefully
integrating up the (local) {\it 
  a priori} estimates from Theorem~\ref{thm:main-one} (for \(p=2\)),
and applying the aforementioned {\it a priori} estimate in 
Proposition~\ref{prop:our a priori} in Appendix~\ref{sec:our a priori}.

To prove Theorem~\ref{thm:main-rho} (in Section~\ref{sec:proof-rho} below), on
\(\alpha\) derivatives (with respect to \(x_1\in\R^{3}\)) of the
electron density \(\rho\), we introduce (see
\eqref{chi-I}--\eqref{partition} in Appendix~\ref{sec:partition} below) a
particular partition of unity, \(1=\sum_{I}\chi_{I}\), in the
integration variable \(\hat{\bf 
  x}_{1}\) (here, \({\bf x}=(x_1,\hat{\bf x}_{1})\in\R^{3N}\)) in the
integral defining \(\rho\) (see \eqref{rho} above). This partition has
the property that, on 
\(\supp\,\chi_{I}\), the derivative \(\partial_{x_{1}}\) can be
changed into a \(\partial_{x_{Q}}\)
for a certain \(Q\subset\{1,\ldots,N\}\) (i.e., a
'derivative parallel to a singularity \(\Sigma_{Q}\)'; see
\eqref{def:x_Q}--\eqref{eq:dist-Q-parallel} above).
Furthermore (again, on
\(\supp\,\chi_{I}\)), \(\lambda_{Q}\) (\(=\min\{1,d_{Q}\}\)) is comparable to \(r(x_{1})\)
(\({}=\min\{1,|x_{1}|\}\); see Lemma~\ref{partition:control} and
Lemma~\ref{lem:der-chi's} below). Applying Proposition~\ref{prop:main-three} 
to each \(\chi_{I}\), and summing, then leads to \eqref{est:pointwise-rho}.
\section{Proof of Theorem~\ref{thm:main-one}}
\label{sec:proof main}
We give the proof of (i) and indicate the necessary changes for the
(much simpler) case of (ii).

We first derive an associated model-equation
(\eqref{eq:P-phi-alpha-lambda} below). 

Fix \(\alpha=(\alpha_1\ldots,\alpha_N)\in\N_0^{3N}\),
\(\alpha_i\in\N_{0}^{3}\), with \(|\alpha|\ge1\), and recall that
\(\Sigma^{\alpha}=\Sigma^{Q_{\alpha}}=\cup_{k\in Q_{\alpha}}\Sigma^{k}\) (see
\eqref{eq:Sigma-k} for \(\Sigma^{k}\)), with  
\(Q_{\alpha}=\{k\in\{1,\ldots,N\}\,|\, \alpha_k\neq0\}\). 
For \({\bf
  x}^{0}=(x_1^{0},\ldots,x_N^0)\in\R^{3N}\setminus\Sigma^{\alpha}\), let
\begin{align}\label{def:lambda-alpha}
  \lambda_{\alpha}:=\min\{1,d_{\alpha}({\bf
    x}^0,\Sigma)\}=\min\{1,d({\bf x}^0,\Sigma^{\alpha})\}>0\,. 
\end{align}

Define, for \({\bf x}=(x_1,\ldots,x_N)\in\R^{3N}\) ,
\begin{align}\label{def:F-alpha}\nonumber
  F_{\alpha}({\bf x})&=\sum_{j\not\in
    Q_{\alpha}}\big(-\tfrac{Z}{2}|x_j|+\tfrac{Z}{2}\sqrt{|x_j|^2+1}\big)
   \\&\quad+\sum_{j,k\not\in Q_{\alpha},
    j<k}\big(\tfrac{1}{4}|x_j-x_k|-\tfrac{1}{4}\sqrt{|x_j-x_k|^2+1}\big)\,. 
\end{align}
Note that there exists \(C=C(N,Z)>0\) such that
\begin{align}\label{eq:bounds-F-alpha}
  |F_{\alpha}({\bf x})|\,,\, |\nabla_{\bf x} F_{\alpha}({\bf x})|\le C
  \quad \text{ for all } {\bf x}\in\R^{3N}\setminus\Sigma\,, 
\end{align}
and that \(\partial^{\beta}_{\bf x}F_{\alpha}\equiv0\) for all
\(\beta\in\N_0^{3N}\) satisfying \(0<\beta\le\alpha\). This follows
from the definition of \(Q_{\alpha}\), since for such
\(\beta=(\beta_1,\ldots,\beta_N)\in\N_{0}^{3N}\),
\(\beta_j\in\N_{0}^{3}\), we have \(\beta_j=0\)
for all \(j\not\in Q_{\alpha}\).
Let 
\begin{align}\label{def:V-alpha}
  V_{\alpha}({\bf x})=\sum_{j\in
    Q_{\alpha}}-\frac{Z}{|x_j|}
   +\sum_{j,k\in Q_{\alpha}, j<k}\frac{1}{|x_j-x_k|}
   +\sum_{j\in Q_{\alpha}, k\not\in Q_{\alpha}}\frac{1}{|x_j-x_k|}\,.
\end{align}
We have that \(V_{\alpha}\in C^{\infty}(B_{3N}({\bf
  x}^{0},\lambda_{\alpha}))\): Note that \(B_{3N}({\bf
  x}^{0},\lambda_{\alpha})\subset \R^{3N}\setminus\Sigma^{\alpha}\) by \eqref{def:lambda-alpha}.
By the definition of \(\Sigma^{\alpha}\) (see \eqref{def:Q-alpha} and
\eqref{eq:Sigma-k}), if \({\bf 
  x}=(x_1,\ldots,x_N)\not\in\Sigma^{\alpha}\), then
\(|x_k|\neq 0\) for all those \(k\in\{1,\ldots,N\}\) for which \(\alpha_k\neq 0\, (\in
\N_0^3)\) (that is, for \(k\in Q_{\alpha}\)), and \(|x_k-x_j|\neq0\)
for the same \(k\), and all \(j\neq k\). 

Next, let
\begin{align}\label{def:G-alpha}\nonumber
  G_{\alpha}({\bf x})&=-\Big[\sum_{j\not\in
     Q_{\alpha}}\tfrac{Z}{2}\Delta_{{\bf x}}(\sqrt{|x_j|^2+1}) 
  -\sum_{j,k\not\in Q_{\alpha},j<k}\tfrac14\Delta_{{\bf x}}(\sqrt{|x_j-x_k|^2+1}) \Big]
  \\&=V({\bf x})-V_{\alpha}({\bf x}) - \Delta_{{\bf x}}
  F_{\alpha}({\bf x})\,.
\end{align}
Since \(|\Delta_{x}(\sqrt{|x|^2+1})|\le 3\) for all
\(x\in\R^3\), there exists  \(C=C(N,Z)>0\) such that
\begin{align}\label{eq:bounds-G-alpha}
  |G_{\alpha}({\bf x})|\le C
  \quad \text{ for all } {\bf x}\in\R^{3N}\,.
\end{align}
Also, \(\partial^{\beta}_{\bf x}G_{\alpha}\equiv0\) for all
\(\beta\in\N_0^{3N}\) satisfying \(0<\beta\le\alpha\), by the same
argument as above.
Therefore, with 
\begin{align}\label{def:K-alpha}
   K_{\alpha}({\bf x}):= G_{\alpha}({\bf x})-|\nabla_{\bf x} F_{\alpha}({\bx })|^{2}\,,
\end{align}
using \eqref{eq:bounds-F-alpha} and \eqref{eq:bounds-G-alpha},
there exists \(C=C(N,Z)>0\) such that
\begin{align}\label{eq:bounds-K-alpha}
  |K_{\alpha}({\bf x})|\le C
  \quad \text{ for all } {\bf x}\in\R^{3N}\,,
\end{align}
and 
\begin{align}\label{eq:der-K-zero}
  \partial^{\beta}_{\bf x}K_{\alpha}\equiv0 \text{ for all }
  \beta\in\N_0^{3N} \text{ satisfying }0<\beta\le\alpha\,.
\end{align}
Define 
\begin{align}\label{def:phi-alpha}
  \psi_{\alpha}:={\rm e}^{-F_{\alpha}}\psi\,,
\end{align}
then, using
  \eqref{eq:eigen-bis}, \(\psi_{\alpha}\) satisfies
\begin{align}\label{eq:phi-alpha}
  -\Delta_{\bf x}\psi_{\alpha}-2\nabla_{\bf x}
  F_{\alpha}\cdot\nabla_{\bf x}\psi_{\alpha}+\big(V_{\alpha}+K_{\alpha}-E\big)\psi_{\alpha}=0\,. 
\end{align}
Define rescaled functions by
\begin{align}\label{def:rescaled-alpha}
  \psi_{\alpha}^{\lambda_{\alpha}}({\bf y})&:=\psi_{\alpha}({\bf
    x}^0+\lambda_{\alpha}{\bf y})\,, \\ \label{def:rescaled-alpha-2}
   V_{\alpha}^{\lambda_{\alpha}}({\bf y})&:=\lambda_{\alpha} V_{\alpha}({\bf
    x}^0+\lambda_{\alpha}{\bf y})\,, \\ \label{def:rescaled-alpha-3}
  H_{\alpha}^{\lambda_{\alpha}}({\bf y})&:=\big(\nabla_{\bf x} F_{\alpha}\big)({\bf
    x}^0+\lambda_{\alpha}{\bf y})\,,\\ \label{def:rescaled-alpha-4}
  K_{\alpha}^{\lambda_{\alpha}}({\bf y})&:=K_{\alpha}({\bf x}^{0}+\lambda_{\alpha}{\bf
    y})
\end{align}
for \({\bf y}=(y_1,\ldots,y_N)\in B_{3N}(0,1)\),
\(y_i\in\R^{3}\). Then, by \eqref{eq:bounds-F-alpha} and 
\eqref{eq:bounds-K-alpha}, 
\begin{align}\label{est:alpha-fcts}
  |K_{\alpha}^{\lambda_{\alpha}}({\bf y})|,
  |H_{\alpha}^{\lambda_{\alpha}}({\bf y})| \le C=C(N,Z)
\end{align}
for all \({\bf y}\in B_{3N}(0,1)\), and
\begin{align}\label{eq:der-K-H-zero}
  \partial^{\beta}_{\bf y}K_{\alpha}^{\lambda_{\alpha}}=
  \partial^{\beta}_{\bf y}H_{\alpha}^{\lambda_{\alpha}}
  \equiv0 \text{ for all }
  \beta\in\N_0^{3N} \text{ satisfying }0<\beta\le\alpha\,.
\end{align}
We have that \(V_{\alpha}^{\lambda_{\alpha}}\in
C^{\infty}(B_{3N}(0,1))\), since, as noted above, \(V_{\alpha}\in
C^{\infty}(B_{3N}({\bf x}^{0},\lambda_{\alpha}))\). 
Furthermore, by the chain rule, for all
\(\gamma\in\N_{0}^{3N}\),
\begin{align}\label{eq:der-V-alpha}
  (\partial_{{\bf y}}^{\gamma}V_{\alpha}^{\lambda_{\alpha}})({\bf y}) =
  \lambda_{\alpha}^{|\gamma|+1}(\partial_{{\bf x}}^{\gamma}V_{\alpha})({\bf
   x}^{0}+\lambda_{\alpha}{\bf y})\,.
\end{align}
Note that, for all
\(\gamma=(\gamma_1,\ldots,\gamma_N)\in\N_{0}^{3N}\),
\(\gamma_i\in\N_0^3\),  
\begin{align}\label{est:V-alpha}\nonumber
  \big|\partial_{\bf x}^{\gamma}V_{\alpha}({\bf x})\big|&\le 
  \sum_{j\in Q_{\alpha}}
  \frac{Z\sqrt{2}\gamma_j!}{|x_j|}\Big(\frac{8}{|x_j|}\Big)^{|\gamma_j|}
  \\&\ +\sum_{j\in Q_{\alpha}}
  \sum_{k=1,k\neq j}^{N}
  \frac{\sqrt{2}(\gamma_j!+\gamma_k!)}{|x_j-x_k|}
  \Big(\frac{8}{|x_j-x_k|}\Big)^{|\gamma_j|+|\gamma_k|}\,.
\end{align}
(The exact value of the constant is immaterial; it can be found in
\cite[Lemma~C.3, (C.7)]{Ana-rel-HF}.) 
By the definition of \(Q_{\alpha}\), of
\(\Sigma^{\alpha}=\Sigma^{Q_{\alpha}}\), and of
\(\lambda_{\alpha}\), we have that 
\begin{align}\label{est:d}
  \lambda_{\alpha}\le d({\bf x}^{0},\Sigma^{\alpha})\le
  \left\{\begin{array}{ll}
         |x_j^{0}|&\text{for all } j\in Q_{\alpha}\,,\\
         \tfrac{1}{\sqrt2}|x_j^{0}-x_k^{0}|&\text{for all } j\in Q_{\alpha}, k\neq j\,.
         \end{array}\right.
\end{align}
Note that \({\bf x}^{0}+\lambda_{\alpha}{\bf y}=(x_{1}^{0}+\lambda_{\alpha}
y_{1},\ldots,x_{N}^{0}+\lambda_{\alpha} y_{N})\). 
Now, let \(R\in(0,1)\), and \({\bf y}\in B_{3N}(0,R)\). Then
\(|y_j|^2+|y_k|^2\le |{\bf y}|^2<R^2\) for all 
\(j,k\in\{1,\ldots,N\}\), and so \(|y_j|+|y_k|<\sqrt2 R\).

Hence, for all \({\bf y}\in B_{3N}(0,R)\), \(R\in(0,1)\), 
\begin{align*}
  |x_{j}^{0}&+\lambda_{\alpha} y_{j}|\ge |x_{j}^{0}|-\lambda_{\alpha}|y_{j}|\ge
  (1-R)\lambda_{\alpha} \quad \text{for all } j\in Q_{\alpha}\,,\\
  |(x_{j}^{0}&+\lambda_{\alpha} y_{j})-(x_{k}^{0}+\lambda_{\alpha} y_{k})|
  \ge |x_{j}^{0}-x_{k}^{0}|-\lambda_{\alpha}|y_{j}-y_{k}|
  \\&\ge \lambda_{\alpha}\big(\sqrt2-(|y_{j}|+|y_{k}|)\big)
  \ge \sqrt2 (1-R)\lambda_{\alpha}\quad \text{for all } j\in Q_{\alpha}, k\neq j\,.
\end{align*} 
Using this, \eqref{eq:der-V-alpha} and \eqref{est:V-alpha} imply that
for all \(\gamma=(\gamma_1,\ldots,\gamma_N)\in\N_{0}^{3N}\), and
all \({\bf y}\in B_{3N}(0,R)\), \(R\in(0,1)\),
\begin{align}\label{eq:der-V--alpha-final}
  \big|(\partial_{{\bf y}}^{\gamma}V_{\alpha}^{\lambda_{\alpha}})(&{\bf y}) \big|
  \le \lambda_{\alpha}^{|\gamma|+1}
  \Big[\frac{\sqrt{2}Z}{(1-R)}\sum_{j\in Q_{\alpha}}\gamma_{j}!
  \Big(\frac{8}{1-R}\Big)^{|\gamma_j|}\lambda_{\alpha}^{-|\gamma_j|-1}
  \\&+\frac{\sqrt2}{1-R}\sum_{j\in Q_{\alpha}}
  \sum_{k=1,k\neq j}^{N}(\gamma_{j}!+\gamma_{k}!)
  \Big(\frac{4\sqrt{2}}{1-R}\Big)^{|\gamma_j|+|\gamma_{k}|}
   \lambda_{\alpha}^{-|\gamma_j|-|\gamma_k|-1}\Big]
   \nonumber
   \\&\le C_{\gamma}(R)\,,
  \nonumber
\end{align}
with
\begin{align}\label{const-irrel}
 C_{\gamma}(R)=C_{\gamma}(R,N,Z)=\frac{\sqrt2}{1-R}
   N\Big(\frac{8}{1-R}\Big)^{|\gamma|}\gamma!(Z+2N-1)\,.
\end{align}
Here we used that \(\gamma_{j}!\le \gamma!\), 
\(|\gamma|=\sum_{j=1}^{N}|\gamma_{j}|\), and that \(\lambda_{\alpha}\le 1\).

The estimate \eqref{eq:der-V--alpha-final} is the essential property
of the potential \(V\) for the proof to work. It is also
satisfied for the potential \(W\) given in \eqref{eq:general-V}; see
Remark~\ref{rem:2}. 

It follows from \eqref{eq:phi-alpha} that
\(\psi_{\alpha}^{\lambda_{\alpha}}\), defined in \eqref{def:rescaled-alpha}, satisfies 
\begin{align}\label{eq:phi-alpha-lambda}\nonumber
  (\Delta_{\bf y}\psi_{\alpha}^{\lambda_{\alpha}})({\bf y}) &=
  \lambda_{\alpha}^2(\Delta_{\bf x}\psi_{\alpha})({\bf x}^{0}+\lambda_{\alpha}{\bf y})
  \\&= -2\lambda_{\alpha} H_{\alpha}^{\lambda_{\alpha}}({\bf y})\cdot(\nabla_{\bf
    y}\psi_{\alpha}^{\lambda_{\alpha}})({\bf y})
   \nonumber\\&\qquad  
  +\big[\lambda_{\alpha}
  V_{\alpha}^{\lambda_{\alpha}}({\bf y})+\lambda_{\alpha}^2(K_{\alpha}^{\lambda_{\alpha}}({\bf
    y})-E)\big]\psi_{\alpha}^{\lambda_{\alpha}}({\bf y})\,,
\end{align}
that is, with \(P=P({\bf y},\partial_{\bf y})\) the operator
\begin{align}\label{def:P}
  P=-\Delta_{\bf y}&-2\lambda_{\alpha} H_{\alpha}^{\lambda_{\alpha}}({\bf y})\cdot\nabla_{\bf
    y}\\&
  +\big[\lambda_{\alpha}
  V_{\alpha}^{\lambda_{\alpha}}({\bf y})+\lambda_{\alpha}^2(K_{\alpha}^{\lambda_{\alpha}}({\bf
    y})-E)\big]
  \nonumber
\end{align}
we have
\begin{align}\label{eq:P-phi-alpha-lambda}
  (P\psi_{\alpha}^{\lambda_{\alpha}})({\bf y})=0 \ \text{ in } B_{3N}(0,1)\,.
\end{align}
Note that for all \(R\in(0,1)\), by \eqref{est:alpha-fcts} and
\eqref{eq:der-V--alpha-final} 
(and \(\lambda_{\alpha}\le1\)),
the coefficients of \(P\) are all in 
\(L^{\infty}(B_{3N}(0,R))\), with norms bounded by some
\(C=C(R,N,Z,E)\); recall also \eqref{eq:der-K-H-zero} and
\eqref{eq:der-V--alpha-final}. 

\begin{pf*}{Proof of Theorem~\ref{thm:main-one} for \(p=\infty\)}
We will use \eqref{eq:P-phi-alpha-lambda} and standard elliptic
regularity (in the form of Theorem~\ref{thm-GT}
in Appendix~\ref{sec:app1} below)
to prove the following lemma, from which the case \(p=\infty\) of
Theorem~\ref{thm:main-one} follows. 
\begin{lem}\label{lem:a-priori-lambda}
 For all \(\beta\in\N_{0}^{3N}\), with \(0<\beta\le\alpha\), all
 \(\theta\in(0,1)\), and all
 \(R,r>0\), \(0<r<R<1\), there exists \(C=C(r,R,\beta,\theta, E,N,Z)\)
 such that
\begin{align}\label{eq:a-priori-alpha}
  \|\partial_{\bf
    y}^{\beta}&\psi_{\alpha}^{\lambda_{\alpha}}\|_{C^{1,\theta}(B_{3N}(0,r))}\\&\le
  C\big(\lambda_{\alpha}\|\psi_{\alpha}^{\lambda_{\alpha}}\|_{L^{\infty}(B_{3N}(0,R))}
  +\|\nabla_{{\bf y}}(\psi_{\alpha}^{\lambda_{\alpha}})\|_{L^{\infty}(B_{3N}(0,R))}\big)\,. 
  \nonumber
\end{align}
\end{lem}
We first prove that Theorem~\ref{thm:main-one} follows from
Lemma~\ref{lem:a-priori-lambda}. 
In particular, \eqref{eq:a-priori-alpha} holds for
\(\beta=\alpha\). Note that for all \(\gamma\in\N_{0}^{3N}\),
\begin{align}\label{eq:rel-der-psi-alpha}
  (\partial_{\bf y}^{\gamma}\psi_{\alpha}^{\lambda_{\alpha}})({\bf y})
  =\lambda_{\alpha}^{|\gamma|}(\partial_{\bf
    x}^{\gamma}\psi_{\alpha})({\bf x}^{0}+\lambda_{\alpha}{\bf y})\,,
\end{align}
so for all \({\bf y}\in B_{3N}(0,r)\), using \eqref{eq:a-priori-alpha},
\begin{align}\label{eq:change-coor-back}
  \big|(\partial_{\bf x}^{\alpha}&\psi_{\alpha})({\bf x}^{0}+\lambda_{\alpha}{\bf y})\big|
  =\lambda^{-|\alpha|}_{\alpha}\big|(\partial_{\bf
    y}^{\alpha}\psi_{\alpha}^{\lambda_{\alpha}})({\bf y})\big|
  \\\nonumber&\le
  C\lambda_{\alpha}^{-|\alpha|}\big(\lambda_{\alpha}\|\psi_{\alpha}^{\lambda_{\alpha}}\|_{L^{\infty}(B_{3N}(0,R))}
  +\|\nabla_{{\bf y}}(\psi_{\alpha}^{\lambda_{\alpha}})\|_{L^{\infty}(B_{3N}(0,R))}\big)
  \\\nonumber&=C\lambda_{\alpha}^{-|\alpha|+1}
  \big(\|\psi_{\alpha}\|_{L^{\infty}(B_{3N}({\bf x}^{0},R\lambda_{\alpha}))}
  +\|\nabla_{{\bf x}}\psi_{\alpha}\|_{L^{\infty}(B_{3N}({\bf x}^{0},R\lambda_{\alpha}))}\big)\,.
\end{align} 
The last equality follows from \eqref{eq:rel-der-psi-alpha}, used on
\(\nabla_{{\bf y}}(\psi_{\alpha}^{\lambda_{\alpha}})\). 

Now (see \eqref{def:phi-alpha}), \(\psi={\rm e}^{F_{\alpha}}\psi_{\alpha}\), with \(\partial_{\bf
    x}^{\alpha}F_{\alpha}\equiv0\), and \(\|F_{\alpha}\|_{L^{\infty}(\R^{3N})}\), 
\(\|\nabla_{{\bf x}}F_{\alpha}\|_{L^{\infty}(\R^{3N})}\le
  C(N,Z)\) (see \eqref{eq:bounds-F-alpha}). Hence,
  \eqref{eq:change-coor-back} gives that, for all \({\bf
    y}\in B_{3N}(0,r)\), 
\begin{align}\label{eq:put-back-F}
  \big|(&\partial_{\bf x}^{\alpha}\psi)({\bf
    x}^{0}+\lambda_{\alpha}{\bf y})\big|
  =\big|({\rm e}^{F_{\alpha}}\partial_{\bf x}^{\alpha}\psi_{\alpha})({\bf
   x}^{0}+\lambda_{\alpha}{\bf y})\big|
  \\\nonumber&\le C \lambda_{\alpha}^{-|\alpha|+1}\big(\|\psi_{\alpha}
  \|_{L^{\infty}(B_{3N}({\bf x}^{0},R\lambda_{\alpha}))}
  +\|\nabla_{{\bf x}}\psi_{\alpha}\|_{L^{\infty}(B_{3N}({\bf x}^{0},R\lambda_{\alpha}))}\big)
  \\\nonumber&=C \lambda_{\alpha}^{-|\alpha|+1}\big(\|{\rm e}^{-F_{\alpha}}\psi
  \|_{L^{\infty}(B_{3N}({\bf x}^{0},R\lambda_{\alpha}))}
  +\|\nabla_{{\bf x}}({\rm e}^{-F_{\alpha}}\psi)\|_{L^{\infty}(B_{3N}({\bf x}^{0},R\lambda_{\alpha}))}\big)
  \\\nonumber
  &\le C \lambda_{\alpha}^{1-|\alpha|}
  \big(\|\psi\|_{L^{\infty}(B_{3N}({\bf x}_{0},R\lambda_{\alpha}))}
  +\|\nabla_{{\bf x}}\psi\|_{L^{\infty}(B_{3N}({\bf x}_{0},R\lambda_{\alpha}))}\big)\,.
   \nonumber
\end{align} 

Hence, the above proves that for all \(\alpha\in\N_0^{3N}\),
\(|\alpha|\ge1\), and all \(0<r<R<1\) there exists 
\(C_{\alpha}(r,R)=C(\alpha,r,R,N,Z,E)\) such that for all \({\bf
  x}^{0}=(x_1^{0},\ldots,x_N^{0})\in\R^{3N}\setminus\Sigma^{\alpha}\),
\begin{align}\label{eq:last-estimate}
  \|\partial^{\alpha}&\psi\|_{L^{\infty}(B_{3N}({\bf x}^{0},r\lambda_{\alpha})}
  \\&\le C_{\alpha}(r,R)
   \lambda_{\alpha}^{1-|\alpha|} \big(\|\psi\|_{L^{\infty}(B_{3N}({\bf
      x}^{0},R\lambda_{\alpha}))}
  +\|\nabla\psi\|_{L^{\infty}(B_{3N}({\bf
      x}^{0},R\lambda_{\alpha}))}\big)\,.
  \nonumber
\end{align}
Recall (see \eqref{def:lambda-alpha}) that 
\(\lambda_{\alpha}=\min\{1,d_{\alpha}({\bf x}^{0},\Sigma)\}\). This
then proves \eqref{eq:resultPSI-Q-alpha}, and therefore
Theorem~\ref{thm:main-one}. 

It remains to prove Lemma~\ref{lem:a-priori-lambda}.
This will be done by induction in \(|\beta|\). 

First, using \eqref{est:alpha-fcts}, \eqref{eq:der-V--alpha-final}
(with \(\gamma=0\)), and \(\lambda_{\alpha}\le1\),
it follows from \eqref{eq:P-phi-alpha-lambda} and Theorem~\ref{thm-GT}
in Appendix~\ref{sec:app1} below that, for all
\(0<r<R<1\) and all \(\theta\in(0,1)\), we have
\begin{align}\label{eq:c-alpha}
  \psi_{\alpha}^{\lambda_{\alpha}} \in C^{1,\theta}_{\rm loc}(B_{3N}(0,1)) \ \text{
    for all } \theta\in(0,1)\,,
\end{align}
and
\begin{align}\label{eq:c-alpha-bis}
  \|\psi_{\alpha}^{\lambda_{\alpha}}\|_{C^{1,\theta}(B_{3N}(0,r))}\le
  C \|\psi_{\alpha}^{\lambda_{\alpha}}\|_{L^{\infty}(B_{3N}(0,R))}
\end{align}
for some constant \(C=C(r,R,\theta,N,Z,E)\).

The induction base:

Let \(\beta\in\N^{3N}\), with \(0<\beta\le\alpha\) and \(|\beta|=1\). Define
\(\varphi_{\alpha,\beta}^{\lambda_{\alpha}}:=\lambda_{\alpha}^{-1}\partial_{\bf
  y}^{\beta}\psi_{\alpha}^{\lambda_{\alpha}}\). Differentiating the
equation \eqref{eq:P-phi-alpha-lambda} for
\(\psi_{\alpha}^{\lambda_{\alpha}}\), then multiplying with
\(\lambda_{\alpha}^{-1}\), we get that 
\begin{align}\label{eq:first-der-psi-alpha}
  (P\varphi_{\alpha,\beta}^{\lambda_{\alpha}})({\bf y})&=
  g_{\alpha,\beta}^{\lambda_{\alpha}}({\bf y})\,,\\
  g_{\alpha,\beta}^{\lambda_{\alpha}}({\bf y})&=
  -\big[\partial_{\bf y}^{\beta}V_{\alpha}^{\lambda_{\alpha}}({\bf
    y})\big]
   \psi_{\alpha}^{\lambda_{\alpha}}({\bf y})\,.
\end{align}
Here we also used \eqref{eq:der-K-H-zero}. 

From \eqref{eq:der-V--alpha-final} (with \(\gamma=\beta\)) and
\eqref{eq:c-alpha} it follows that, for all \(R\in(0,1)\),  
\(g_{\alpha,\beta}^{\lambda_{\alpha}}\in L^{\infty}(B_{3N}(0,R))\), 
and that
\begin{align}\label{eq:est-g-alpha}
  \|g_{\alpha,\beta}^{\lambda_{\alpha}}\|_{L^{\infty}(B_{3N}(0,R))}
  \le C_{\beta}\|\psi_{\alpha}^{\lambda_{\alpha}}\|_{L^{\infty}(B_{3N}(0,R))}\,.
\end{align}
It therefore follows from Theorem~\ref{thm-GT} that
\(\varphi_{\alpha,\beta}^{\lambda_{\alpha}}\in C^{1,\theta}_{\rm loc}(B_{3N}(0,1))\)
for all \(\theta\in(0,1)\), and that, for all \(0<r<R<1\), 
\begin{align}\label{eq:induc-step-phi-alpha}
 \|\varphi_{\alpha,\beta}^{\lambda_{\alpha}}&\|_{C^{1,\theta}(B_{3N}(0,r))}
  \\\nonumber&\le
  C\big(\|g_{\alpha,\beta}^{\lambda_{\alpha}}\|_{L^{\infty}(B_{3N}(0,R))}
  +\|\varphi_{\alpha,\beta}^{\lambda_{\alpha}}\|_{L^{\infty}(B_{3N}(0,R))}\big)
  \\\nonumber&\le C\big(\|\psi_{\alpha}^{\lambda_{\alpha}}\|_{L^{\infty}(B_{3N}(0,R))}
  +\|\varphi_{\alpha,\beta}^{\lambda_{\alpha}}\|_{L^{\infty}(B_{3N}(0,R))}\big)\,,
\end{align}
for some constant \(C=C(r,R,\theta, \beta, E,N,Z)\).

It follows from \eqref{eq:induc-step-phi-alpha} and the fact that
\(\varphi_{\alpha,\beta}^{\lambda_{\alpha}}=\lambda_{\alpha}^{-1}\partial_{\bf 
  y}^{\beta}\psi_{\alpha}^{\lambda_{\alpha}}\), \(|\beta|=1\), 
that, for all
\(0<r<R<1\), all \(\theta\in(0,1)\), and all \(\beta\in\N_{0}^{3N}\) with \(0<\beta\le\alpha\),
  \(|\beta|=1\), 
\begin{align}\label{eq:first-ind-step-final}
  \|\partial_{\bf y}^{\beta}
  &\psi_{\alpha}^{\lambda_{\alpha}}\|_{C^{1,\theta}(B_{3N}(0,r))}
  \\&\le
  C\big(\lambda_{\alpha}\|\psi_{\alpha}^{\lambda_{\alpha}}\|_{L^{\infty}(B_{3N}(0,R))}
  +\|\nabla_{{\bf
      y}}\psi_{\alpha}^{\lambda_{\alpha}}\|_{L^{\infty}(B_{3N}(0,R))}\big)\,,
  \nonumber
\end{align}
for some \(C=C(r,R,\theta,\beta, E,N,Z)\). This is
\eqref{eq:a-priori-alpha}. 

The induction step:

Let now \(j\in\N\), \(1\le j\le|\alpha|\), and assume
\eqref{eq:a-priori-alpha}
holds for  all \(\beta\in\N_{0}^{3N}\), with \(0<\beta\le\alpha\) and
\(|\beta|\le j\), all
 \(\theta\in(0,1)\), and all
 \(0<r<R<1\), with a constant \(C=C(r,R,\beta,\theta,
 E,N,Z)\). 

Let \(\beta\le\alpha\), with \(|\beta|=j+1\), and 
let (as before)
\(\varphi_{\alpha,\beta}^{\lambda_{\alpha}}:=\lambda_{\alpha}^{-1}\partial^{\beta}_{\bf
  y}\psi_{\alpha}^{\lambda_{\alpha}} \). Differentiating the equation
\eqref{eq:P-phi-alpha-lambda},  then multiplying with
\(\lambda_{\alpha}^{-1}\), we get that 
\begin{align}\label{eq:diff-eq-alpha}
  P(\varphi_{\alpha,\beta}^{\lambda_{\alpha}})({\bf y})&=g_{\alpha,\beta}^{\lambda_{\alpha}}\,,\\
 g_{\alpha,\beta}^{\lambda_{\alpha}}&=
 -\sum_{\gamma\le\beta,
   |\gamma|\ge1}\binom{\beta}{\gamma}\big[(\partial_{\bf
   y}^{\gamma}V_{\alpha}^{\lambda_{\alpha}})({\bf y})\big](\partial_{\bf
   y}^{\beta-\gamma}\psi_{\alpha}^{\lambda_{\alpha}})({\bf y})\,.\nonumber
\end{align}
Again, we also used \eqref{eq:der-K-H-zero}. 

From \eqref{eq:der-V--alpha-final} and the induction hypothesis
(that is, \eqref{eq:a-priori-alpha} for \(\beta-\gamma\le\alpha\) with
\(|\beta-\gamma|\le|\beta|-1=j\)) 
it follows that
\(g_{\alpha,\beta}^{\lambda_{\alpha}}\in L^{\infty}(B_{3N}(0,\tilde{r}))\) for all \(\tilde{r}\in(0,1)\), 
and that, for all \(0<r<R<1\), 
\begin{align}\label{eq:est-g-alpha-bis}
  \|g_{\alpha,\beta}^{\lambda_{\alpha}}&\|_{L^{\infty}(B_{3N}(0,(r+R)/2))}
  \le\!\!\sum_{\gamma\le\beta,
   |\gamma|\ge1}C_{\gamma,\beta}\|\partial_{\bf
   y}^{\beta-\gamma}\psi_{\alpha}^{\lambda_{\alpha}}\|_{L^{\infty}(B_{3N}(0,R))}
  \\&\!\!\le
  C_{\beta}\big(\lambda_{\alpha}\|\psi_{\alpha}^{\lambda_{\alpha}}\|_{L^{\infty}(B_{3N}(0,R))}
  +\|\nabla_{{\bf y}}(\psi_{\alpha}^{\lambda_{\alpha}})\|_{L^{\infty}(B_{3N}(0,R))}\big)\,. 
  \nonumber
\end{align}
It therefore follows from Theorem~\ref{thm-GT} (applied to
\eqref{eq:diff-eq-alpha}) that  
\(\varphi_{\alpha,\beta}^{\lambda_{\alpha}}\in C^{1,\theta}_{\rm loc}(B_{3N}(0,1))\)
for all \(\theta\in(0,1)\), and that, for all \(0<r<R<1\), 
\begin{align}\label{eq:induc-step-phi-alpha-bis}
 \|\varphi_{\alpha,\beta}^{\lambda_{\alpha}}&\|_{C^{1,\theta}(B_{3N}(0,r))}
  \\\nonumber&\le
  C\big(\|g_{\alpha,\beta}^{\lambda_{\alpha}}\|_{L^{\infty}(B_{3N}(0,(r+R)/2))}
  +\|\varphi_{\alpha,\beta}^{\lambda_{\alpha}}\|_{L^{\infty}(B_{3N}(0,(r+R)/2))}\big)
  \\\nonumber&\le C\big(\lambda_{\alpha}\|\psi_{\alpha}^{\lambda_{\alpha}}\|_{L^{\infty}(B_{3N}(0,R))}
  +\|\nabla_{{\bf y}}(\psi_{\alpha}^{\lambda_{\alpha}})\|_{L^{\infty}(B_{3N}(0,R))}
  \\&\qquad\qquad\qquad\qquad\qquad\qquad
  +\|\varphi_{\alpha,\beta}^{\lambda_{\alpha}}\|_{L^{\infty}(B_{3N}(0,(r+R)/2))}\big)\,,
  \nonumber
\end{align}
for some constant \(C=C(r,R,\beta,\theta,E,N,Z)\).

Now write
\(\beta=\beta_j+e_j\), \(|e_j|=1, |\beta_j|=j\) (so
\(\beta_j\le\alpha\)). Recall that
\(\varphi_{\alpha,\gamma}^{\lambda_{\alpha}}=\lambda_{\alpha}^{-1}\partial_{\bf  
  y}^{\gamma}\psi_{\alpha}^{\lambda_{\alpha}}\),
\(\gamma\in\N_{0}^{3N}\). 
Then, by the induction hypothesis
(used on \(\beta_j\le\alpha\), \(|\beta_j|=j\)), and the definition of
the \(C^{1,\theta}\)-norm, 
\begin{align}\label{eq:other-est-nduc-hyp}
  \|\varphi_{\alpha,\beta}^{\lambda_{\alpha}}&\|_{L^{\infty}(B_{3N}(0,(r+R)/2))}
  \\\nonumber&
  =\|\partial_{\bf y}^{e_j}\varphi_{\alpha,\beta_j}^{\lambda_{\alpha}}\|_{L^{\infty}(B_{3N}(0,(r+R)/2))}
  \le \|\varphi_{\alpha,\beta_j}^{\lambda_{\alpha}}\|_{C^{1,\theta}(B_{3N}(0,(r+R)/2))}
  \\\nonumber&=\lambda_{\alpha}^{-1}\|\partial_{\bf
    y}^{\beta_j}\psi_{\alpha}^{\lambda_{\alpha}}\|_{C^{1,\theta}(B_{3N}(0,(r+R)/2))} 
  \\\nonumber&\le
  C\big(\|\psi_{\alpha}^{\lambda_{\alpha}}\|_{L^{\infty}(B_{3N}(0,R))} 
  +\lambda_{\alpha}^{-1}\|\nabla_{{\bf y}}(\psi_{\alpha}^{\lambda_{\alpha}})\|_{L^{\infty}(B_{3N}(0,R))}\big)\,.
\end{align}
It follows from \eqref{eq:induc-step-phi-alpha}, 
\eqref{eq:other-est-nduc-hyp}, and the fact that 
\(\varphi_{\alpha,\beta}^{\lambda_{\alpha}}=\lambda_{\alpha}^{-1}\partial^{\beta}_{\bf
  y}\psi_{\alpha}^{\lambda_{\alpha}} \) that, for all \(0<r<R<1\),
\begin{align}\label{eq:almost-done-induc}
  \|\partial_{\bf
    y}^{\beta}&\psi_{\alpha}^{\lambda_{\alpha}}\|_{C^{1,\theta}(B_{3N}(0,r))} 
  =\lambda_{\alpha}\|\varphi_{\alpha,\beta}^{\lambda_{\alpha}}\|_{C^{1,\theta}(B_{3N}(0,r))} 
  \\\nonumber&\le
  C\lambda_{\alpha}\big\{\lambda_{\alpha}\|\psi_{\alpha}^{\lambda_{\alpha}}\|_{L^{\infty}(B_{3N}(0,R))}
  +\|\nabla_{{\bf y}}(\psi_{\alpha}^{\lambda_{\alpha}})\|_{L^{\infty}(B_{3N}(0,R))}
  \\\nonumber&\qquad\qquad+\|\psi_{\alpha}^{\lambda_{\alpha}}\|_{L^{\infty}(B_{3N}(0,R))}
  +\lambda_{\alpha}^{-1}\|\nabla_{{\bf y}}(\psi_{\alpha}^{\lambda_{\alpha}})\|_{L^{\infty}(B_{3N}(0,R))}\big\}\,.
\end{align}
Using that \(\lambda_{\alpha}\le1\), this proves that \eqref{eq:a-priori-alpha} holds for
all \(\beta\in\N_{0}^{3N}\), with \(0<\beta\le\alpha\) and \(|\beta|=j+1\), all
 \(\theta\in(0,1)\), and all
 \(R,r>0\), \(0<r<R<1\), and some constant \(C=C(r,R,\beta,\theta,N,Z,E)\).
The lemma now follows by induction over \(j\).
\end{pf*}
This finishes the proof of Theorem~\ref{thm:main-one} in the case \(p=\infty\).\qed
\begin{pf*}{Proof of Theorem~\ref{thm:main-one} for \(p\in(1,\infty)\)}
Again, \eqref{eq:P-phi-alpha-lambda} and standard elliptic
regularity (this time in the form of Theorem~\ref{thm-GT-two}
in Appendix~\ref{sec:app1} below)
give the following lemma, from which the case \(p\in(1,\infty)\) of
Theorem~\ref{thm:main-one} follows. This lemma is the substitute
for Lemma~\ref{lem:a-priori-lambda} in the case \(p\in(1,\infty)\).
\begin{lem}\label{lem:a-priori-lambda-L-2}
For all \(p\in(1,\infty)\), all \(\beta\in\N_{0}^{3N}\), with
\(0<\beta\le\alpha\), and all 
\(R,r>0\), \(0<r<R<1\), there exists \(C=C(p,r,R,\beta,E,N,Z)\)
such that
\begin{align}\label{eq:a-priori-alpha-L-2}
  \|\partial_{\bf y}^{\beta}
  &\psi_{\alpha}^{\lambda_{\alpha}}\|_{W^{2,p}(B_{3N}(0,r))}\\&\le
  C\big(\lambda_{\alpha}\|\psi_{\alpha}^{\lambda_{\alpha}}\|_{L^{p}(B_{3N}(0,R))}
  +\|\nabla_{{\bf y}}(\psi_{\alpha}^{\lambda_{\alpha}})\|_{L^{p}(B_{3N}(0,R))}\big)\,. 
  \nonumber
\end{align}
\end{lem}
The proof of Lemma~\ref{lem:a-priori-lambda-L-2} follows that of
Lemma~\ref{lem:a-priori-lambda} verbatim, except for substituting
`Theorem~\ref{thm-GT-two}' for `Theorem~\ref{thm-GT}',
`\(W^{2,p}_{{\rm loc}}\)' for `\(C^{1,\theta}_{{\rm loc}}\)',
`\(W^{2,p}(B_{3N}(0,\cdot))\)' for
`\(C^{1,\theta}(B_{3N}(0,\cdot))\)' (and leaving out \(\theta\) everywhere),
and `\(L^{p}(B_{3N}(0,\cdot))\)' for `\(L^{\infty}(B_{3N}(0,\cdot))\)'. 

Similarly, the proof that Theorem~\ref{thm:main-one} follows
from Lemma~\ref{lem:a-priori-lambda-L-2} in the case
\(p\in(1,\infty)\) mimics the one that
Theorem~\ref{thm:main-one} for \(p=\infty\) follows 
from Lemma~\ref{lem:a-priori-lambda} (substituting `\(L^p\)' for
`\(L^{\infty}\)'), and is left to the reader.
\end{pf*}
\section{Proof of Corollaries~\ref{cor:main} and ~\ref{cor:exp}}
\label{sec:proof-cor}
\begin{proof}[Proof of Corollary~\ref{cor:main}]
The inequality \eqref{eq:pointwise-bis} follows from
\eqref{eq:pointwise} by using \eqref{eq:d-alpha}, and that
\(\Sigma\supseteq\Sigma^{\alpha}\). 

It is obviously enough to prove  \eqref{eq:pointwise} for all
\(R\in(0,1)\). 
Use that \({\bf x}\in B_{3N}({\bf x},t)\) for all \(t>0\), the
bound \eqref{eq:resultPSI-Q-alpha} (with 
\(R/4\) and \(R/2\) for \(R\in(0,1)\)), that \(\lambda_{\alpha}\le
1\),   and the {\it a priori} estimate for \(\nabla\psi\) in
Theorem~\ref{eq:AHP-apriori} 
in Appendix~\ref{sec:our a priori} below
(with \(R/2\) and \(R\)),
to get the inequalities
\begin{align}\label{pf:cor}\nonumber
  |\partial^{\alpha}&\psi({\bf x})|\le
  \|\partial^{\alpha}\psi\|_{L^{\infty}(B_{3N}({\bf
      x},R\lambda_{\alpha}({\bf x})/4))}
  \\&\nonumber \le C\,\lambda_{\alpha}({\bf x})^{1-|\alpha|}
  \big(\|\psi\|_{L^{\infty}(B_{3N}({\bf
      x},R/2))}
  +\|\nabla\psi\|_{L^{\infty}(B_{3N}({\bf
      x},R/2))}\big)\\
  &\le C\,\lambda_{\alpha}({\bf x})^{1-|\alpha|}
  \|\psi\|_{L^{\infty}(B_{3N}({\bf x},R))}\,.
\end{align}
\end{proof}
\begin{proof}[Proof of Corollary~\ref{cor:exp}]
Note first that \eqref{eq:resultPSIc} follows from \eqref{eq:resultPSIb} (in
the same way that  \eqref{eq:pointwise-bis} followed from
\eqref{eq:pointwise} in the proof of Corollary~\ref{cor:main}).
 
To prove \eqref{eq:resultPSIb} note first that \eqref{eq:exp-dec}
implies that
\begin{align}\label{eq:equiv-norm-exp-decay}
  \|\psi\|_{L^{\infty}(B_{3N}({\bf x},1/2))}\le 
  C_0{\rm e}^{c_{0}/2}\, {\rm e}^{-c_{0}|{\bf x}|}\quad\text{ for all } {\bf
    x}\in\mathbb R^{3N}\,.
\end{align}
Therefore \eqref{eq:resultPSIb} follows from \eqref{eq:pointwise} when
\(d_{\alpha}({\bf x},\Sigma)\le 1\). 
Secondly note that (since \( d_{\alpha}({\bf x},\Sigma)\le |{\bf
  x}|\), see \eqref{eq:d-alpha}) we have that 
\begin{align*}
  d_{\alpha}({\bf x},\Sigma)^{|\alpha|-1}{\rm
  e}^{-c_{0}|{\bf x}|/2} 
  \le |{\bf x}|^{|\alpha|-1}{\rm
  e}^{-c_{0}|{\bf x}|/2} \,,
\end{align*}
and the right side is uniformly bounded for \({\bf x}\in\R^{3N}\). This proves
that \eqref{eq:resultPSIb} also follows from \eqref{eq:pointwise} when
\(d_{\alpha}({\bf x},\Sigma)\ge1\).  
(Note that this also shows that we can take \(c_{\alpha}\) as close
to \(c_{0}\) as we like, at the expense of increasing
\(C_{\alpha}\). Similarly for \(\widetilde{C}_{\alpha},
\widetilde{c}_{\alpha}\).) This finishes the proof of
\eqref{eq:resultPSIb}. 

To prove \eqref{eq:Sob-a} note that (see \eqref{eq:d-alpha}), for all
\({\bf x}=(x_1,\ldots,x_N)\in\R^{3N}\), 
\begin{align}\label{eq:triv}
d_{\alpha}({\bf x},\Sigma)&=|x_j| \ \ \text{ for some } j\in\{1,\ldots,N\}
\intertext{or}
\label{eq:triv-2}
d_{\alpha}({\bf x},\Sigma)&=\tfrac{1}{\sqrt{2}}|x_j-x_k| \ \ \text{
  for some } j,k\in\{1,\ldots,N\}\,. 
\end{align}
Hence, for all \({\bf x}\in \R^{3N}\) and all \(s\in\R\),
\begin{align}\label{eq:bound-d-pot}
  d_{\alpha}({\bf x},\Sigma)^{s}\le
  \sum_{j=1}^N|x_j|^{s}
    +\sum_{1\le j<k\le N}\big(\tfrac{1}{\sqrt{2}}|x_j-x_k|\big)^{s}\,.
\end{align}
We use the notation of \eqref{notation: hat x} and
\eqref{notation:dx} and, for \(j,k\in\{1,\ldots,N\}\), define the
orthogonal transformation 
\((y_j,y_k)=(x_j-x_k,x_j+x_k)/\sqrt2\). We denote the new coordinates
in \(\R^{3N}\) by \({\bf y}\). Then
it follows from \eqref{eq:resultPSIb} and \eqref{eq:bound-d-pot} that,
for \(|\alpha|\ge1\),
\begin{align}\label{eq:play exp's}
  \int_{\R^{3N}\setminus \Sigma^{\alpha}}
  &\big|d_{\alpha}({\bf x},\Sigma)^{|\alpha|-a}
  \partial^{\alpha}\psi({\bf  x})\big|^{2}
  \,d{\bf x}
  \\\nonumber&\le C_{\alpha}\sum_{j=1}^{N}\Big(\int_{\R^3}|x_j|^{2-2a}
  {\rm e}^{-2c_{\alpha}|x_j|}\,dx_j\Big)
   \Big(\int_{\R^{3N-3}}
   {\rm e}^{-2c_{\alpha}|\hat{\bf x}_j|}
   \,d\hat{\bf x}_j\Big)
   \\\nonumber&\ +C_{\alpha}\sum_{1\le j<k\le N}\Big(\int_{\R^3}|y_{j}|^{2-2a}
   {\rm e}^{-2c_{\alpha}|y_{j}|}\,dy_{j}\Big)
   \Big(\int_{\R^{3N-3}}
   {\rm e}^{-2c_{\alpha}|\hat{\bf y}_{j}|}
   \,d\hat{\bf y}_j\Big)\,.
\end{align}
Now note that the right side is finite for all \(a<5/2\). This
finishes the proof of \eqref{eq:Sob-a}.

The same argument works for \(d({\bf x},\Sigma)\) (use
\eqref{eq:resultPSIc} instead of \eqref{eq:resultPSIb}).
\end{proof}

\section{Proof of Theorem~\ref{thm:Nistor}}
\label{sec:proof-Nistor}
\begin{proof}
Note that, with \(\lambda_{\alpha}({\bf x})=\min\{1,d({\bf
  x},\Sigma^{\alpha})\}\), we have, for all \({\bf x}, {\bf
  y}\in\R^{3N}\) and \(\epsilon\in(0,1)\),
\begin{align}\label{eq:slow metric}
  |{\bf x}-{\bf y}|\le \epsilon \lambda_{\alpha}({\bf x})
  \ \Rightarrow\ 
  (1-\epsilon) \lambda_{\alpha}({\bf x}) \le \lambda_{\alpha}({\bf y}) \le
  (1+\epsilon) \lambda_{\alpha}({\bf x})\,.
\end{align}
This follows from 
\begin{align}\label{eq:dist}
 \big|d({\bf x},\Sigma^{\alpha})-d({\bf y},\Sigma^{\alpha})\big|\le
 |{\bf x}-{\bf y}|\,.
\end{align}
Also, for all \({\bf z}\in\R^{3N}\), \(b>0\),
\begin{align}\label{eq:integral ball}
   1=C_{N}(b)\lambda_{\alpha}({\bf z})^{-3N}
   \int_{\R^{3N}}
   \1_{\{|{\bf z}-{\bf u}|\le  b\lambda_{\alpha}({\bf  z})\}}
   ({\bf u})\,d{\bf u}\,,
\end{align}
with \(C_N(b)=b^{3N}\Vol(B_{3N}(0,1)\)). 
Note that, as a consequence of \eqref{eq:slow metric}, for all \({\bf z},
{\bf u}\in\R^{3N}\), all \(k\in\R\), and all \(\epsilon\in(0,1)\), 
\begin{align}\label{eq:change charac}
  \lambda_{\alpha}^{k}({\bf z})\,
  \1_{\{|{\bf z}-{\bf u}|\le  \epsilon\lambda_{\alpha}({\bf
      z})\}}({\bf u})
  \le C(k,\epsilon)\lambda_{\alpha}^{k}({\bf u})\,
  \1_{\{|{\bf u}-{\bf z}|\le \frac{\epsilon}{1-\epsilon}\lambda_{\alpha}({\bf
      u})\}}({\bf z})\,.
\end{align}

In the following we suppress that constants depend on
\(N,\alpha\), and \(a\). Also, \(C\) might change from line to line.

Using \eqref{eq:integral ball}, then 
\eqref{eq:change charac} (both with \(({\bf z},{\bf u})=({\bf x},{\bf
  y})\), and with \(b=\epsilon=1/4\)), 
we get that
\begin{align}
\label{eq:Nistor - pf}\nonumber
  &\int_{\R^{3N}\setminus \Sigma^{\alpha}}\big|\big(\lambda_{\alpha}^{|\alpha|-a}\partial^{\alpha}\psi\big)({\bf
   x})\big|^2\,d{\bf x}
  \\\nonumber&=C\int_{\R^{3N}\setminus \Sigma^{\alpha}}
  \int_{\R^{3N}}
  \!\!\!\!\!\!
  \lambda_{\alpha}({\bf
    x})^{-3N}
  \1_{\{|{\bf x}-{\bf y}|\le \lambda_{\alpha}({\bf  x})/4\}}({\bf y})
  \lambda_{\alpha}({\bf
    x})^{2|\alpha|-2a}\big|\partial^{\alpha}\psi({\bf x})\big|^{2}\,d{\bf y}\,d{\bf x}
  \\\nonumber&\le C\int_{\R^{3N}}\int_{\R^{3N}\setminus \Sigma^{\alpha}}
  \!\!\!\!\!\!
  \lambda_{\alpha}({\bf y})^{-3N}
  \1_{\{|{\bf y}-{\bf x}|\le \lambda_{\alpha}({\bf  y})/3\}}({\bf x})
  \lambda_{\alpha}({\bf y})^{2|\alpha|-2a}
  \big|\partial^{\alpha}\psi({\bf x})\big|^{2}\,d{\bf x}\,d{\bf y}
  \\\nonumber&=C\int_{\R^{3N}}\lambda_{\alpha}({\bf
    y})^{-3N}\lambda_{\alpha}({\bf y})^{2|\alpha|-2a}
  \Big(\int_{B_{3N}({\bf y},\lambda_{\alpha}({\bf y})/3)}  
  \big|\partial^{\alpha}\psi({\bf x})\big|^{2}\,d{\bf x}\Big)\,d{\bf
    y}
  \\&=C\int_{\R^{3N}}\lambda_{\alpha}({\bf
    y})^{-3N}\lambda_{\alpha}({\bf y})^{2|\alpha|-2a}
  \big\|\partial^{\alpha}\psi\|^{2}_{L^{2}(B_{3N}({\bf
      y},\lambda_{\alpha}({\bf y})/3))} \,d{\bf y}\,.
\end{align}
We also used that \(B_{3N}({\bf y},\lambda_{\alpha}({\bf y})/3)\setminus\Sigma^{\alpha}=
B_{3N}({\bf y},\lambda_{\alpha}({\bf y})/3)\).

We now use the {\it a priori} estimate \eqref{eq:resultPSI-Q-alpha} 
in Theorem~\ref{thm:main-one}
(with \(p=2\) and \(r=1/3, R=2/3\)),
then \eqref{eq:change charac} (this time with \(({\bf   z},{\bf
  u})=({\bf y},{\bf x})\) and \(\epsilon=2/3\)), and finally \eqref{eq:integral ball} 
(again with \(({\bf z},{\bf u})=({\bf x},{\bf
  y})\), but with \(b=2\)), to get that 
\begin{align}
\label{eq:Nistor - pf - two}\nonumber
  &\int_{\R^{3N}\setminus\Sigma^{\alpha}}\big|\big(\lambda_{\alpha}^{|\alpha|-a}\partial^{\alpha}\psi\big)({\bf
   x})\big|^2\,d{\bf x}
  \\\nonumber&
  \le C\int_{\R^{3N}}\lambda_{\alpha}({\bf
    y})^{-3N}\lambda_{\alpha}({\bf y})^{2-2a}
  \big\{\|\psi\|^{2}_{L^{2}(B_{3N}({\bf 
      y},2\lambda_{\alpha}({\bf y})/3))} 
    \\\nonumber&\qquad\qquad\qquad\qquad\qquad\qquad\qquad
   +\|\nabla\psi\|^{2}_{L^{2}(B_{3N}({\bf
      y},2\lambda_{\alpha}({\bf y})/3))} \big\}
   \,d{\bf y}
  \\\nonumber&=
  C\int_{\R^{3N}}\lambda_{\alpha}({\bf
    y})^{-3N}\lambda_{\alpha}({\bf y})^{2-2a}\times
  \\\nonumber&
\qquad\qquad
  \times
\Big(\int_{\R^{3N}}
  \1_{\{|{\bf y}-{\bf x}|\le 2\lambda_{\alpha}({\bf  y})/3\}}({\bf x})
  \big(|\psi({\bf x})|^2+|\nabla\psi({\bf x})|^2\big)\,   \,d{\bf x}\Big)   \,d{\bf y}
  \\\nonumber&\le C \int_{\R^{3N}}\lambda_{\alpha}({\bf
    x})^{2-2a}\big(|\psi({\bf x})|^2+|\nabla\psi({\bf x})|^2\big)\times
  \\\nonumber&
\qquad\qquad\qquad\qquad
  \times
   \Big(\int_{\R^{3N}}\lambda_{\alpha}({\bf x})^{-3N}
  \1_{\{|{\bf x}-{\bf y}|\le 2\lambda_{\alpha}({\bf  x})\}}({\bf y})
  \,d{\bf y}\Big)\,d{\bf x}
  \\&=
  C\int_{\R^{3N}}\lambda_{\alpha}^{2-2a}({\bf x})
  \big(|\psi({\bf x})|^2+|\nabla\psi({\bf x})|^2\big)\,d{\bf x}\,.
\end{align}

Recall that \(\lambda_{\alpha}({\bf x})\le 1\) for all \({\bf
  x}\in\R^{3N}\). Hence, if \(a\le1\), it follows that
\begin{align}\label{eq:triv case}\nonumber
  \int_{\R^{3N}\setminus\Sigma^{\alpha}}&\big|\big(\lambda_{\alpha}^{|\alpha|-a}\partial^{\alpha}\psi\big)({\bf
   x})\big|^2\,d{\bf x}
  \le C\int_{\R^{3N}}
  \big(|\psi({\bf x})|^2+|\nabla\psi({\bf x})|^2\big)\,d{\bf x}
  \\&=C\|\psi\|_{W^{1,2}(\R^{3N})}^{2}\le
 C\|\psi\|_{W^{2,2}(\R^{3N})}^{2}<\infty\,,
\end{align}
since \(\psi\in W^{2,2}(\R^{3N})\), which proves \eqref{eq:Sob-two} in
this case.

If, on the other hand, \(a\in(1,5/2)\), we have that
\begin{align}\label{eq:compl case}
   &\int_{\R^{3N}\setminus\Sigma^{\alpha}}\big|\big(\lambda_{\alpha}^{|\alpha|-a}\partial^{\alpha}
   \psi\big)({\bf x})\big|^2\,d{\bf x}
   \\\nonumber&\le C\int_{\{d({\bf x},\Sigma^{\alpha})\le1\}}
   \!\!\!\!\!\!\!
   \lambda_{\alpha}^{2-2a}({\bf x})
   \big(|\psi({\bf x})|^2+|\nabla\psi({\bf x})|^2\big)\,d{\bf x}
   \\\nonumber&\, +C\int_{\{d({\bf x},\Sigma^{\alpha})>1\}}
   \!\!\!\!\!\!\!
   \lambda_{\alpha}^{2-2a}({\bf x})
   \big(|\psi({\bf x})|^2+|\nabla\psi({\bf x})|^2\big)\,d{\bf x}
   \\\nonumber&\le C\int_{\{d({\bf x},\Sigma^{\alpha})\le1\}} 
   \!\!\!\!\!\!\!
   d({\bf x},\Sigma^{\alpha})^{2-2a}
   \big(|\psi({\bf x})|^2+|\nabla\psi({\bf x})|^2\big)\,d{\bf x}
   +C\|\psi\|_{W^{2,2}(\R^{3N})}^{2}\,,
\end{align}
by the same argument as above. It therefore remains to estimate the
first term on the right side of \eqref{eq:compl case}. (At this point,
compare with \eqref{eq:play exp's}.) Using \eqref{eq:triv},
\eqref{eq:triv-2}, and \eqref{eq:bound-d-pot} we get that
\begin{align}\label{eq:Large_a}
   &\int_{\{d({\bf x},\Sigma^{\alpha})\le1\}} 
   \!\!\!\!\!\!\!
   d({\bf x},\Sigma^{\alpha})^{2-2a}
   \big(|\psi({\bf x})|^2+|\nabla\psi({\bf x})|^2\big)\,d{\bf x}
   \\&\le \sum_{j=1}^{N}\int_{\{|x_j|\le 1\}}|x_j|^{2-2a}
   \big(|\psi({\bf x})|^2+|\nabla\psi({\bf x})|^2\big)\,d{\bf x}
   \nonumber
   \\&\ +\sum_{1\le j<k\le N}^{N}\int_{\{|x_j-x_k|\le 1\}}\big(\tfrac{1}{\sqrt2}|x_j-x_k|\big)^{2-2a}
   \big(|\psi({\bf x})|^2+|\nabla\psi({\bf x})|^2\big)\,d{\bf x}\,.
   \nonumber
\end{align}

It remains to show that each summand on the right side in
\eqref{eq:Large_a} is finite for any \(a<\frac52\). 
All summands will be treated in the same
manner, so we just consider one of each of them. 

For fixed $\hat{\bf x}_{1} \in \R^{3N-3}$ we can estimate, since \(a<\frac{5}{2}\),
\begin{align}\label{eq:Nistor-5}
  &\int_{\{|x_1|\le 1\}}|x_1|^{2-2a}
  \big(|\psi(x_1,\hat{\bf x}_{1})|^2+|\nabla\psi(x_1,\hat{\bf
    x}_{1})|^2\big)\,dx_1\nonumber\\ 
  &\leq
   \big\{\|\psi\|_{L^{\infty}(B_{3N}((0,\hat{\bf
       x}_{1}),2))}^2+\|\nabla\psi\|_{L^{\infty}(B_{3N}((0,\hat{\bf
       x}_{1}),2))}^2\big\}
   \int_{\{|x_1|\le 1\}}|x_1|^{2-2a}\,dx_1 \nonumber \\
   &\leq
   C(a) \|\psi\|_{L^{2}(B_{3N}((0,\hat{\bf x}_{1}),4))}^2\,,
\end{align}
where we used the {\it a priori} estimate of Proposition~\ref{prop:our a
  priori} below and the finiteness of the integral to get the last
inequality. 

Therefore we get
\begin{align}\label{eq:Nistor-6}
  \int_{\{|x_1|\le 1\}}&|x_1|^{2-2a}
  \big(|\psi({\bf x})|^2+|\nabla\psi({\bf x})|^2\big)\,d{\bf x}
  \nonumber \\ 
  &\leq
  C \int_{\R^{3N-3}} \big( \int_{\R^{3N}} |\psi({\bf y})|^2 \1_{\{|{\bf
      y} - (0, \hat{\bf x}_{1})| \leq 4\}} \,d{\bf y}\big) 
  d\hat{\bf x}_{1} \nonumber \\
  &\leq
  C \int_{\R^{3N}} |\psi({\bf y})|^2  \Big(  \int_{\R^{3N-3}} 
  \1_{\{|\hat{\bf y}_{1} - \hat{\bf x}_{1}| \leq 4\}} 
  \,d\hat{\bf x}_{1} \Big) 
 \,d{\bf
    y}\nonumber \\ 
  &=C \| \psi\|^2_{L^2(\R^{3N})} < \infty\,.
\end{align}
The last inequality follows since the inner integral is independent of
\({\bf y}\).

Similarly, for fixed $\hat{\bf x}_{1,2} \in \R^{3N-6}$ (with \({\bf
  x}=(x_1,x_2,\hat{\bf x}_{1,2})\)), make the 
orthogonal transformation \((y_1,y_2)=(x_1-x_2,x_1+x_2)/\sqrt2\)
(see the argument leading to \eqref{eq:play exp's}).
Then we can estimate, since \(a<\frac{5}{2}\),
\begin{align}\label{eq:Nistor-7}
  &\int_{\{|x_1-x_2|\le 1\}}
  \!\!\!\!\!
  \big(\tfrac{1}{\sqrt2}|x_1-x_2|\big)^{2-2a}
  \big(|\psi({\bf x})|^2+|\nabla\psi({\bf
    x})|^2\big)\,dx_1\,dx_2\nonumber\\ 
  &=\int_{\R^3}\int_{\{|y_1|\le 1\}}
  |y_1|^{2-2a}
  \big\{|\psi(\tfrac{y_1+y_2}{\sqrt2},\tfrac{y_2-y_1}{\sqrt2},\hat{\bf
    x}_{1,2})|^2 \nonumber
  \\&\qquad\qquad\qquad\qquad\qquad\qquad\qquad
  +|\nabla\psi(\tfrac{y_1+y_2}{\sqrt2},\tfrac{y_2-y_1}{\sqrt2},\hat{\bf
    x}_{1,2})|^2\big\}\,dy_1\,dy_2\nonumber\\  
  &\leq \int_{\R^3}
  \Big(\int_{\{|y_1|\le 1\}}|y_1|^{2-2a}\,dy_1\Big)
   \big\{\|\psi\|_{L^{\infty}(B_{3N}((\tfrac{y_2}{\sqrt2},\tfrac{y_2}{\sqrt2},
  \hat{\bf x}_{1,2}),2))}^2
  \nonumber\\&\qquad\qquad\qquad\qquad\qquad\qquad\qquad
  +\|\nabla\psi\|_{L^{\infty}(B_{3N}((\tfrac{y_2}{\sqrt2},\tfrac{y_2}{\sqrt2},\hat{\bf
       x}_{1,2}),2))}^2\big\}
   \,dy_2 \nonumber \\
   &\leq
   C(a) \int_{\R^3}
   \|\psi\|_{L^{2}(B_{3N}((\tfrac{y_2}{\sqrt2},\tfrac{y_2}{\sqrt2},\hat{\bf x}_{1,2}),4))}^2\,dy_2\,,
\end{align}
where we again used the {\it a priori} estimate of Proposition~\ref{prop:our a
  priori} and the finiteness of the inner integral to get the last
inequality. 

Hence,
\begin{align}\label{eq:Nistor-8}\nonumber
  &\int_{\{|x_1-x_2|\le 1\}}
  \!\!\!\!\!
  \big(\tfrac{1}{\sqrt2}|x_1-x_2|\big)^{2-2a}
  \big(|\psi({\bf x})|^2+|\nabla\psi({\bf
    x})|^2\big)\,dx_1\,dx_2\,d\hat{\bf x}_{1,2}   
  \\&\le
  C \int_{\R^{3N}}|\psi({\bf z})|^2 \Big( \int_{\R^{3N-3}}  \1_{\{|{\bf
      z} - (\tfrac{y_2}{\sqrt2},\tfrac{y_2}{\sqrt2},\hat{\bf
      x}_{1,2})| \leq 4\}}
  dy_2\, d\hat{\bf x}_{1,2}\Big) \,d{\bf z}\nonumber
  \\&\le C \int_{\R^{3N}}|\psi({\bf z})|^2 \Big( \int_{\R^{3N-3}}
  \1_{\{|(z_2, \hat{\bf
      z}_{1,2}) - (\tfrac{y_2}{\sqrt2},\hat{\bf
      x}_{1,2})| \leq 4\}}
  dy_2\, d\hat{\bf x}_{1,2}\Big) \,d{\bf z}\nonumber\\
 &=C \| \psi\|^2_{L^2(\R^{3N})} < \infty\,,
\end{align}
where, again, the last inequality follows since the inner integral is
independent of \({\bf z}\).

This finishes the proof of \eqref{eq:Sob-two} (for
\(\lambda_{\alpha}\)). The proof of \eqref{eq:Sob-three} (for
\(\lambda\)) is completely analogous (replace \(\lambda_{\alpha}\) by
\(\lambda\), and use \eqref{eq:resultPSIa} from Theorem~\ref{thm:main-one}
instead of \eqref{eq:resultPSI-Q-alpha}, in the argument above).

This finishes the proof of Theorem~\ref{thm:Nistor}.
\end{proof}
\section{Proof of Proposition~\ref{prop:main-three}}
\label{sec:proof Q}
Assume, without restriction, that
\(Q=\{1,\ldots,M\}\subseteq\{1,\ldots,N\}\), \(M\le N\). 
Fix \({\bf
  x}^{0}=(x_1^{0},\ldots,x_N^0)\in\R^{3N}\setminus\Sigma_{Q}\) and
\begin{align}\label{def:lambda-alpha-Q}
  \lambda_{Q}:=\min\{1,d_{Q}({\bf
    x}^0,\Sigma)\}=\min\{1,d({\bf x}^0,\Sigma_{Q})\}\,. 
\end{align}
For \(\Sigma_{Q}\) and \(d({\bf x}^0,\Sigma_{Q})\), see
\eqref{def:Sigma-parallel}--\eqref{eq:dist-Q-parallel}. Recall that 
(in general) \(\Sigma_{Q}\neq\Sigma^{Q}\).
We proceed similarly to the proof of Theorem~\ref{thm:main-one} but
will exploit the structure of \(\Sigma_{Q}\).

Define, for \({\bf x}=(x_1,\ldots,x_N)\in\R^{3N}\),
\begin{align}\label{def:F-alpha-Q}\nonumber
  F_{Q}({\bf x})&=\sum_{j\not\in Q}
  \big(-\tfrac{Z}{2}|x_j|+\tfrac{Z}{2}\sqrt{|x_j|^2+1}\big) 
  \\\nonumber
  &\quad+\sum_{j,k\not\in Q,j<k}
  \big(\tfrac{1}{4}|x_j-x_k|-\tfrac{1}{4}\sqrt{|x_j-x_k|^2+1}\big) 
  \\&\quad+\sum_{j,k\in Q, j<k}
  \big(\tfrac{1}{4}|x_j-x_k|-
  \tfrac{1}{4}\sqrt{|x_j-x_k|^2+1}\big)\,. 
\end{align}
Note that there exists \(C=C(N,Z)>0\) such that
\begin{align}\label{eq:bounds-F-alpha-Q}
  |F_{Q}({\bf x})|\,,\, |\nabla_{\bf x} F_{Q}({\bf x})|\le C
  \quad \text{ for all } {\bf x}\in\R^{3N}\setminus\Sigma\,, 
\end{align}
and that \(\partial^{\beta}_{x_{Q}}F_{Q}\equiv0\) for all
\(\beta\in\N_0^{3}\), \(\beta\neq0\), by the
definition of \(x_{Q}\) and \(\partial^{\beta}_{x_{Q}}\) (see
\eqref{def:x_Q}--\eqref{def:d-x-Q}), since, 
such \(\beta\), and all \(i,j\),
\begin{align}\label{eq:parallel-diff-e-e}
   \partial_{x_i+x_j}^{\beta}|x_i-x_j|=0\,;
\end{align}
see also \eqref{parallel-der-zero}. Let 
\begin{align}\label{def:V-alpha-Q}
  V_{Q}({\bf x})=\sum_{j\in Q}
  -\frac{Z}{|x_j|}
  +\sum_{j\in Q}\sum_{k\not\in Q}\frac{1}{|x_j-x_k|}\,.
\end{align}
Note that
\(V_{Q}\in C^{\infty}(B_{3N}({\bf x}^{0},\lambda_{Q}))\),
since \({\bf x}^{0}\in\R^{3N}\setminus \Sigma_{Q}\), and
\(\lambda_{Q}\le d({\bf x}^{0},\Sigma_{Q})\) by
\eqref{def:lambda-alpha-Q}.  

Define
\begin{align}\label{eq:def_psi_Q}
  \psi_Q := e^{-F_Q} \psi\,.
\end{align}
Then, using \(H\psi=E\psi\), we get that \(\psi_Q\) satisfies the
equation 
\begin{align}
  -\Delta \psi_Q - 2 \nabla F_Q \cdot \nabla \psi_Q +
  (V_Q + K_Q - E) \psi_Q=0\,,
\end{align}
where
\begin{align}
  K_Q = - |\nabla F_Q|^2 -\Delta
  \Big\{&\sum_{j\not\in Q}\tfrac{Z}{2}\sqrt{|x_j|^2+1}
  -\sum_{j,k\not\in Q,j<k}
  \tfrac{1}{4}\sqrt{|x_j-x_k|^2+1} \nonumber \\
  &-\sum_{j,k\in Q,j<k}
  \tfrac{1}{4}\sqrt{|x_j-x_k|^2+1}\,\Big\} \,.
\end{align}
Notice that $K_Q$ is bounded on \(\R^{3N}\setminus\Sigma\), and that 
\(\partial^{\beta}_{x_{Q}}K_{Q}\equiv0\) for all
\(\beta\in\N_{0}^{3}\), \(\beta\neq0\), just as above for \(F_{Q}\).  

Define rescaled functions by
\begin{align}\label{def:rescaled-alpha-Q}
  \psi_{Q}^{\lambda_{Q}}({\bf y})&:=\psi_{Q}({\bf x}^0
  +\lambda_{Q}{\bf y})\,, 
  \\ V_{Q}^{\lambda_{Q}}({\bf y})&:=\lambda_{Q} 
  V_{Q}({\bf x}^0+\lambda_{Q}{\bf y})\,, \\
  H_{Q}^{\lambda_{Q}}({\bf y})&:=\big(\nabla_{\bf x} F_{Q}\big)
  ({\bf x}^0+\lambda_{Q}{\bf y})\,,\\
  K_{Q}^{\lambda_{Q}}({\bf y})&:=K_{Q}({\bf x}^{0}
  +\lambda_{Q}{\bf y})
\end{align}
for \({\bf y}=(y_1,\ldots,y_N)\in B_{3N}(0,1)\),
\(y_i\in\R^{3}\). Then, since $K_Q$ and $\nabla F_Q$ are bounded on 
\(\R^{3N}\setminus\Sigma\),   
\begin{align}\label{est:alpha-fcts-Q}
  |K_{Q}^{\lambda_{Q}}({\bf y})|,
  |H_{Q}^{\lambda_{Q}}({\bf y})| \le C,
\end{align}
for all \({\bf y}\in B_{3N}(0,1)\), and
\begin{align}\label{eq:der-K-H-zero-Q}
  \partial^{\beta}_{y_Q}K_{Q}^{\lambda_{Q}}=
  \partial^{\beta}_{y_Q}H_{Q}^{\lambda_{Q}}
  \equiv0 \ \text{ for all }
  \beta\in\N_0^{3}\,,\, \beta\neq0\,.
\end{align}
Here, \(y_{Q}\) and \(\partial^{\beta}_{y_{Q}}\) are defined as for
\(x_{Q}\). 
Also, \(V_{Q}^{\lambda_{Q}}\in
C^{\infty}(B_{3N}(0,1))\), and, by estimates and arguments as in 
\eqref{eq:der-V-alpha}--\eqref{eq:der-V--alpha-final},
\begin{align}\label{eq:der-V-alpha-Q}
  |(\partial_{{\bf y}}^{\gamma}V_{Q}^{\lambda_{Q}})({\bf y})| =
  |\lambda_{Q}^{|\gamma|+1}(\partial_{{\bf x}}^{\gamma}V_{Q})({\bf
   x}^{0}+\lambda_{Q}{\bf y})|\leq C_{\gamma}(R)\,,
\end{align}
for all $R<1$, ${\bf y} \in B_{3N}(0,R)$.

It follows that, in $B_{3N}(0,1)$,
\begin{align}\label{eq:P-phi-Q-lambda}
  \Big\{-\Delta_{\bf y}-2\lambda_{Q} &H_{Q}^{\lambda_{Q}}({\bf y})\cdot\nabla_{\bf
    y}\nonumber \\&
  +\big[\lambda_{Q}
  V_{Q}^{\lambda_{Q}}({\bf y})+\lambda_{Q}^2(K_{Q}^{\lambda_{Q}}({\bf
    y})-E)\big] \Big\}\psi_Q^{\lambda_Q} = 0\,.
\end{align}
Compare with \eqref{def:P}--\eqref{eq:P-phi-alpha-lambda}.

The proof of Proposition~\ref{prop:main-three} follows by successive
differentiation with respect to $y_Q$ of the equation
\eqref{eq:P-phi-Q-lambda} for $\psi_Q^{\lambda_Q}$, and from applying
elliptic regularity to the resulting equations.
We state the relevant results for $p=\infty$ and $p<\infty$ as
Lemmas~\ref{lem:a-priori-Q} and~\ref{lem:a-priori-Q-L-2} below. One
can compare with Lemmas~\ref{lem:a-priori-lambda}
and~\ref{lem:a-priori-lambda-L-2}.  

\begin{lem}\label{lem:a-priori-Q}
 For all \(\beta\in\N_{0}^{3}\), with \(\beta\neq0\), all
 \(\theta\in(0,1)\), and all
 \(R,r>0\), \(0<r<R<1\), there exists \(C=C(r,R,\beta,\theta,E,N,Z)\)
 such that
\begin{align}\label{eq:a-priori-Q}
  \|\partial_{y_Q}^{\beta}&\psi_{Q}^{\lambda_{Q}}\|_{C^{1,\theta}(B_{3N}(0,r))}\\&\le
  C\big(\lambda_{Q}\|\psi_{Q}^{\lambda_{Q}}\|_{L^{\infty}(B_{3N}(0,R))}
  +\|\nabla_{{\bf y}}(\psi_{Q}^{\lambda_{Q}})\|_{L^{\infty}(B_{3N}(0,R))}\big)\,. 
  \nonumber
\end{align}
\end{lem}

\begin{lem}\label{lem:a-priori-Q-L-2}
For all \(p\in(1,\infty)\), all \(\beta\in\N_{0}^{3}\), with
\(\beta\neq0\), and all 
\(R,r>0\), \(0<r<R<1\), there exists \(C=C(p,r,R,\beta,E,N,Z)\)
such that
\begin{align}\label{eq:a-priori-Q-L-2}
  \|\partial_{y_Q}^{\beta}
  &\psi_{Q}^{\lambda_{Q}}\|_{W^{2,p}(B_{3N}(0,r))}\\&\le
  C\big(\lambda_{Q}\|\psi_{Q}^{\lambda_{Q}}\|_{L^{p}(B_{3N}(0,R))}
  +\|\nabla_{{\bf y}}(\psi_{Q}^{\lambda_{Q}})\|_{L^{p}(B_{3N}(0,R))}\big)\,. 
  \nonumber
\end{align}
\end{lem}

Since the proofs of Lemmas~\ref{lem:a-priori-Q}
and~\ref{lem:a-priori-Q-L-2} are completely analogous to the proofs of
Lemmas~\ref{lem:a-priori-lambda} and~\ref{lem:a-priori-lambda-L-2}, we
omit them here. (Note the similarity, but also difference, between 
\eqref{eq:der-K-H-zero-Q} and \eqref{eq:der-K-H-zero}.)
It is also simple to verify that Proposition~\ref{prop:main-three}
follows from Lemmas~\ref{lem:a-priori-Q} and~\ref{lem:a-priori-Q-L-2}
and the definition of $\psi_Q^{\lambda_Q}$ (see
\eqref{def:rescaled-alpha-Q} and \eqref{eq:def_psi_Q}). (Compare with
the proof that Theorem~\ref{thm:main-one} follows from
Lemma~\ref{lem:a-priori-lambda}, situated after
Lemma~\ref{lem:a-priori-lambda}.) 

\section{Proof of Theorem~\ref{thm:main-rho}}
\label{sec:proof-rho}
\begin{pf} To prove (i), 
let \(\rho\) be as in the theorem.
Note that it suffices to prove the 
statement for each \(\rho_j\) in \eqref{rho}. The proof is the same
for each \(j\), and so we shall prove it for \(\rho_1\), which, by
abuse of notation, we shall denote \(\rho\). To ease notation, we
shall write \({\bf x}=(x_1,\ldots,x_N)=(x_1,\hat{\bf x}_{1})\) and
\(\rho=\rho(x_1)\).  

To prove \eqref{est:pointwise-rho},
let \(x_1\in\R^3\setminus\{0\}\) and \(R\in(0,1)\) (the case \(R\ge1\)
obviously  follows from this case). 
Let \(1=\sum_I\chi_I\) be the partition of unity (on \(\R^{3N}\)) from
Lemma~\ref{lem:partition} in Appendix~\ref{sec:partition} below. Then
\begin{align}\label{rho:split}
  \rho(x_1)&=\sum_I
  \int_{\R^{3N-3}}
  \!\!\!  \!\!\!  \!\!\!
  \chi_{I}(x_1,\hat{\bf x}_{1})|\psi(x_1,\hat{\bf x}_{1})|^2\,d\hat{\bf x}_{1}
  =\sum_I\rho_I(x_1)\,,
\end{align}
and so, for all \(\alpha\in\N_0^3\) with \(|\alpha|\ge1\),
\begin{align}\label{rho:split-der}
  (\partial^{\alpha}_{x_1}\rho)(x_1)=\sum_I(\partial_{x_1}^{\alpha}\rho_I)(x_1)
\end{align}
with
\begin{align}\label{rho:der-rho-I}
  (\partial^{\alpha}_{x_1}\rho_I)(x_1)=
  \partial^{\alpha}_{x_1}\Big(\int_{\R^{3N-3}}
  \!\!\!  \!\!\!  \!\!\!
  \chi_{I}(x_1,\hat{\bf x}_{1})|\psi(x_1,\hat{\bf x}_{1})|^2\,
  d\hat{\bf x}_{1}\Big)\,.
\end{align}
It then suffices to prove the estimate in \eqref{est:pointwise-rho} for
each \(\partial^{\alpha}_{x_1}\rho_I\), since the sum in
\eqref{rho:split-der} is finite. 

To this end, recall again  the definition of \(\chi_I\) from
\eqref{chi-I} in Lemma~\ref{lem:partition} below. Let
\(Q:=\{1\}\cup(\cup_{j=0}^{J-1}Q_{j})\subseteq\{1,\ldots,N\}\). 
Re-numbering, we may assume that \(Q=\{1,\ldots,M\}\), \(M\le
N\). In the integral in \eqref{rho:der-rho-I}, make the change of
variables
\begin{align}\label{change-var}
  y_j=x_j-x_1\,,\ j=2,\ldots,M\,.
\end{align}
Then (re-naming \(y_j\) to \(x_j\) again)
\begin{align}\label{after-var-change}
  \int_{\R^{3N-3}}
  \!\!\!  \!\!\!  \!\!\!
  &\chi_{I}(x_1,\hat{\bf x}_{1})|\psi(x_1,\hat{\bf x}_{1})|^2\,
  d\hat{\bf x}_{1}
  \\&\nonumber
  = \int_{\R^{3N-3}}
  \!\!\!  \!\!\!  \!\!\!
  \big(\chi_{I}|\psi|^2\big)(x_1,x_2+x_1,\ldots,x_M+x_1,x_{M+1},\ldots,x_N)
  \,d\hat{\bf x}_{1}
  \\&=
  \int_{\R^{3N-3}}
  \!\!\!  \!\!\!  \!\!\!
  \widetilde\chi_{I}({\bf x})
  \,|\psi|^2(x_1,x_1+x_2,\ldots,x_1+x_M,x_{M+1},\ldots,x_N)
  \,d\hat{\bf x}_{1}\,,
  \nonumber
\end{align}
with \(\widetilde\chi_{I}\) as in 
\eqref{chi-I-tilde} in Lemma~\ref{lem:der-chi's} below (see
\eqref{chi-I} for \(\chi_{I}\)). By Leibniz' rule
\begin{align}\label{der-rho_I}
  (&\partial^{\alpha}_{x_1}\rho_I)(x_1)=
  \sum_{\beta\le\alpha}\binom{\alpha}{\beta}
  \int_{\R^{3N-3}}
  \!\!\!  \!\!\!
  \big(\partial_{x_1}^{\beta}
  \widetilde\chi_{I}({\bf x})\big)
  \times\\&\qquad
  \times\Big(\partial_{x_1}^{\alpha-\beta}
  \big\{|\psi|^2(x_1,x_1+x_2,\ldots,x_1+x_M,x_{M+1},\ldots,x_N)\big\}\Big)
  \,d\hat{\bf x}_{1}\,.\nonumber
\end{align} 
Differentiating under the integral sign can be justified as in \cite[p.~97]{AHP}.
Again by Leibniz' rule, and the chain rule,
\begin{align}\label{der-change-to-x-Q}
  \partial_{x_1}^{\alpha-\beta}&\big\{|\psi|^2(x_1,x_1+x_2,\ldots,x_1+x_M,x_{M+1},\ldots,x_N)\big\}
  \\\nonumber&\!=\sum_{\sigma\le\alpha-\beta}\binom{\alpha-\beta}{\sigma}
  \partial_{x_1}^{\sigma}\big\{\overline{\psi}(x_1,x_1+x_2,\ldots,x_1+x_M,x_{M+1},\ldots,x_N)\big\}
  \\\nonumber&\quad\qquad\quad\ \times
  \partial_{x_1}^{\alpha-\beta-\sigma}\big\{\psi(x_1,x_1+x_2,\ldots,x_1+x_M,x_{M+1},\ldots,x_N)\big\}\,.
\end{align}
Note that, by the chain rule and \eqref{def:der-x_q}, for \(s=1,2,3\),
\begin{align}\label{der-chain-rule}
  \partial_{x_1}^{e_s}&\big\{\psi(x_1,x_1+x_2,\ldots,x_1+x_M,x_{M+1},\ldots,x_N)\big\}
  \\\nonumber&=\sum_{j=1}^{M}\big(\partial_{x_{j,s}}\psi\big)(x_1,x_1+x_2,\ldots,x_1+x_M,x_{M+1},\ldots,x_N)
  \\\nonumber&=\sqrt{M}\big(\partial_{x_Q}^{e_s}\psi\big)(x_1,x_1+x_2,\ldots,x_1+x_M,x_{M+1},\ldots,x_N)\,,
\end{align}
and, by iteration, for \(\sigma\in\N_{0}^{3}\), 
\begin{align}\label{der-end-formula}
  \partial_{x_1}^{\sigma}&\big\{\psi(x_1,x_1+x_2,\ldots,x_1+x_M,x_{M+1},\ldots,x_N)\big\}
  \\\nonumber&=M^{|\sigma|/2}\big(\partial_{x_Q}^{\sigma}\psi\big)(x_1,x_1+x_2,\ldots,x_1+x_M,x_{M+1},\ldots,x_N)\,.
\end{align}
Now apply \eqref{der-change-to-x-Q} and \eqref{der-end-formula}
in \eqref{der-rho_I} above. 
Then estimate \(\partial_{x_1}^{\beta}\widetilde\chi_{I}({\bf x})\)
using Lemma~\ref{lem:der-chi's} below. (For \(|\beta|=|\alpha|\ge1\),
use \eqref{est:der-chi's-BIS} or \eqref{est:der-chi's-BIS-2} with
\(n=1\); for \(\beta<\alpha\), use \eqref{est:der-chi's}.) 
Then re-change variables (\(y_j=x_j+x_1\,,\ j=2,\ldots,M\)) and
re-name them back to \(x_j\), to obtain (for some \(j\in\{2,\ldots,N\}\))
\begin{align}\label{big-step-forward}
  &\big|(\partial^{\alpha}_{x_1}\rho_I)(x_1)\big|
  \le  C|x_1|^{1-|\alpha|}\int_{\R^{3N-3}}\!\!\!\!\!\!\!
  |x_j|^{-1}
  |\psi({\bf x})|^2\,d\hat{\bf x}_{1}
  \\\nonumber&+
  C\sum_{\beta<\alpha}|x_1|^{-|\beta|}\Big(\!\!
  \sum_{\sigma\le\alpha-\beta}\!
  \int_{\R^{3N-3}}\!\!\!\!\!\!\!\!\1_{\supp\,\chi_{I}}({\bf x})
  \big|\big(\partial_{x_Q}^{\sigma}\psi\big)({\bf x})\big|\,
  \big|\big(\partial_{x_Q}^{\alpha-\beta-\sigma}\psi\big) 
  ({\bf x})\big|\,d\hat{\bf x}_{1}\Big)\,.
\end{align}
By Lemma~\ref{partition:control} below, and the fact that \(x_1\neq0\), it
follows that if \({\bf x}\in\supp\,\chi_{I}\), then \({\bf
  x}\in\R^{3N}\setminus\Sigma_{Q}\) (see
\eqref{def:Sigma-parallel}). 
Also note that (see
\eqref{eq:dist-Q-parallel}), for all 
\({\bf x}=(x_1,\ldots,x_N)\in\R^{3N}\), 
\begin{align*}
  d_{Q}({\bf x},\Sigma)&=|x_j| \ \ \text{ for some } j\in Q=\{1,\ldots,M\}
  \intertext{or}
  d_{Q}({\bf x},\Sigma)&=\tfrac{1}{\sqrt{2}}|x_j-x_k| \ \ \text{
  for some } j\in Q,k\in\{M+1,\ldots,N\}\,. 
\end{align*}
In both case, it follows from Lemma~\ref{partition:control} that if
\({\bf x}\in\supp\,\chi_{I}\), then 
\begin{align}\label{bound-d_Q}
  d_{Q}({\bf x},\Sigma)^{1-|\sigma|}\le C|x_1|^{1-|\sigma|}\,.
\end{align}
Also, if \(|x_1|\le1\), then \(d_{Q}({\bf x},\Sigma)\le |x_1|\le1\) (since \(1\in Q\) and
\(|x_1|\le 1\); see also \eqref{eq:dist-Q-parallel}),
so, from \eqref{bound-d_Q},
\begin{align}\label{est-l-Q}
  \lambda_{Q}({\bf x})^{1-|\sigma|}\le 
  C\,r(x_1)^{1-|\sigma|}\,.
\end{align}
(Recall that \(\lambda_{Q}({\bf x})=\min\big\{1, d_{Q}({\bf
  x},\Sigma)\big\}\) and \(r(x_1)=\min\{1,|x_1|\}\).) On the other hand, if
\(|x_1|\ge1\), then either 
\(\lambda_{Q}({\bf x})=d_{Q}({\bf x},\Sigma)\), so that, by
\eqref{bound-d_Q}, and assuming \(\sigma\ne0\), then
\begin{align}\label{est-l-Q-2}
   \lambda_{Q}({\bf x})^{1-|\sigma|}\le C\,|x_1|^{1-|\sigma|}\le C\le C\,r(x_1)^{1-|\sigma|}\,,
 \end{align}
or, \(\lambda_{Q}({\bf x})=1\), in which case 
\eqref{est-l-Q} holds trivially if \(\sigma\ne0\) (since \(C\ge1\) and \(1-|\sigma|\le0\)).

In conclusion, if \({\bf x}\in\supp\,\chi_{I}\), then \eqref{est-l-Q}
holds for all \(\sigma\ne0\).

Hence, applying the estimate \eqref{eq:resultPSI-Q-alpha-parallel} in 
Proposition~\ref{prop:main-three} (with \(p=\infty\),
\(r=R/8,R=R/4\) (!)) for each point \({\bf x}\in\R^{3N}\) for which the
integrand in the second integral in \eqref{big-step-forward} is
non-zero,  and then the {\it a priori} estimate in
Theorem~\ref{eq:AHP-apriori} below
(with \(R=R/4\) (!)) for \(\nabla\psi\),
we get that
\begin{align}\label{another-est}
  \big|(&\partial^{\alpha}_{x_1}\rho_I)(x_1)\big|
  \\\nonumber&\le
  C\,r(x_1)^{1-|\alpha|}\int_{\R^{3N-3}}\!\!\!\!\!\!\!
  |x_j|^{-1}|\psi({\bf x})|^2\,d\hat{\bf x}_{1} 
  \\\nonumber&
  +C\sum_{\beta<\alpha}r(x_1)^{-|\beta|}\Big\{
  r(x_1)^{1-|\alpha|-|\beta|} 
  \int_{\R^{3N-3}}\!\!\!\!\!\!\!
  |\psi({\bf x})|\,
  \|\psi\|_{L^{\infty}(B_{3N}({\bf x},R/2))}
  \,d\hat{\bf x}_{1}
  \nonumber\\&\qquad+
  \sum_{0<\sigma<\alpha-\beta}
  \!\!\!\!
  r(x_1)^{1-|\sigma|}r(x_1)^{1-(|\alpha|-|\beta|-|\sigma|)}
  \int_{\R^{3N-3}}\!\!\!\!\!\!\!
  \|\psi\|_{L^{\infty}(B_{3N}({\bf x},R/2))}^2
  \,d\hat{\bf
    x}_{1}\Big\}\,.
  \nonumber
\end{align}
(The term in the second line comes from \(\sigma=0\) and
\(\sigma=\alpha-\beta\).) 
At this point we can finish the proof using
Proposition~\ref{prop:AprioriEstimatesOnRho} below (with 
\(r=R/2, R=R\)) to get 
\begin{align}
  \big|(&\partial^{\alpha}_{x_1}\rho_I)(x_1)\big| \leq C\,
  r(x_1)^{1-|\alpha|} \int_{B_3(x_1,R)}\rho(y)\,dy\\ \nonumber
  &\le  C\,  r(x_1)^{1-|\alpha|}\|\rho\|_{L^{1}(\R^3)}=
  C\,  r(x_1)^{1-|\alpha|} \|\psi\|_{L^{2}(\R^{3N})}^2\,,
\end{align}
for all $x_1\in\R^{3}\setminus\{0\}$ (and \(\alpha\ne0\)). This finishes the proof of
\eqref{est:pointwise-rho}, and hence of (i).

To prove (ii), for \(p\in[1,\infty)\), let \(a\in[0,\tfrac{p+3}{p})\). Then, since
\(r(x)=\min\{1,|x|\}\), 
\begin{align}\label{eq:rho-Babuschka-1}\nonumber
   \int_{\R^{3}\setminus\{0\}}&\!\!\!\!\!\!\!\!
   \big[r(x)^{|\alpha|-a}\partial^{\alpha}\rho(x)\big]^{p}\,dx\\
   &=\int_{B_3(0,1)\setminus\{0\}}  \!\!\!\!\!\!\!\!\!\!\!\!\!\!\!\!
   \big[|x|^{|\alpha|-a}\partial^{\alpha}\rho(x)\big]^p\,
     dx
  +\int_{\R^{3}\setminus B_3(0,1)}  \!\!\!\!\!\!\!\!\!\!\!\!\!\!\!\!
  |\partial^{\alpha}\rho(x)|^p\,dx\,.
\end{align}
Now, by \eqref{est:pointwise-rho} (for \(R=1\)), 
\begin{align}\label{eq:rho-Babuschka-2}\nonumber
  &\int_{B_{3}(0,1)\setminus\{0\}}
  \!\!\!\!\!\!\!\!\!\!\!\!\!\!\!\!
  \big[|x|^{|\alpha|-a}\partial^{\alpha}\rho(x)\big]^p\,dx 
  \le
   C\int_{B_{3}(0,1)}
   \!\!\!
   \Big(|x|^{1-a}\int_{B_{3}(x,1)}
  \!\!\!\!\!\!\!\!
  \rho(y)\,dy\Big)^p\,dx
   \\&\le C\|\rho\|_{L^{1}(\R^3)}^{p}\int_{B_{3}(0,1)}
  \!\!\!\!\!\!\!\!
  |x|^{p(1-a)}\,dx =
   C_{\alpha}(a,p)   \|\psi\|_{L^{2}(\R^{3N})}^{2p}<\infty\,, 
\end{align}
since \(a<\tfrac{p+3}{p}\) and by the definition \eqref{rho} of \(\rho\).

Furthermore, by \eqref{est:pointwise-rho}, for all
\(x\in\R^3\setminus B_3(0,1)\), 
\begin{align}\label{eq:rho-Babuschka-4}
  |\partial^{\alpha}\rho(x)|^{p-1}\le
  C\Big(\int_{B_3(x,1)}
  \!\!\!\!\!\!\!\!
  \rho(y)\,dy\Big)^{p-1}
  \le C\|\psi\|_{L^{2}(\R^{3N})}^{2(p-1)}\,,
\end{align}
again, by the defintion of \(\rho\).
Hence, again by \eqref{est:pointwise-rho} and Fubini,
\begin{align}\label{eq:rho-Babuschka-5}
  &\int_{\R^{3}\setminus B_3(0,1)}
  \!\!\!\!\!\!\!\!
  |\partial^{\alpha}\rho(x)|^p\,dx
  \le C\|\psi\|_{L^{2}(\R^{3N})}^{2(p-1)}
  \int_{\R^3\setminus B_3(0,1)}
  \!\!\!\!\!\!\!\!
  |\partial^{\alpha}\rho(x)|\,dx 
  \\\nonumber
  &\le
  C\|\psi\|_{L^{2}(\R^{3N})}^{2(p-1)}\int_{\R^3}\Big(\int_{B_3(x,1)}
  \!\!\!\!\!\!\!\!
  \rho(y)\,dy\Big)\,dx
  \\&=C\|\psi\|_{L^{2}(\R^{3N})}^{2(p-1)}\int_{\R^3}\rho(y)\Big(\int_{\R^3}
  \1_{\{|x-y|\le 1\}}
  \,dx\Big)\,dy
  =C_{\alpha}(p) \|\psi\|_{L^{2}(\R^{3N})}^{2p}\,.
  \nonumber
\end{align}
It then follows from \eqref{eq:rho-Babuschka-1},
\eqref{eq:rho-Babuschka-2}, and \eqref{eq:rho-Babuschka-5} that
\begin{align}\label{eq:rho-Babuschka-6}
  \int_{\R^{3}\setminus\{0\}}
  \!\!\!\!\!\!\!\!\!
  \big[r(x)^{|\alpha|-a}\partial^{\alpha}\rho(x)\big]^{p}\,dx
  \le C_{\alpha}(a,p)\|\psi\|_{L^2(\R^{3N})}^{2p}<\infty\,.
\end{align}
This proves (ii).

To prove (iii), note that, for \(|x_1|\le1\), the estimate
\eqref{eq:resultRHO} follows from 
\eqref{est:pointwise-rho}. 
For \(|x_1|>1\), it follows from 
\cite[Theorem~1 (1.10)]{rho-smooth} (using that, for all \(\delta>0\)
and all \(\alpha\in\N_{0}^{3}\) with \(|\alpha|\ge1\), the function
\({\rm e}^{-\delta|x_1|}|x_1|^{|\alpha|-1}\) is uniformly bounded for
\(|x_1|\ge1\)). 

This finishes the proof of Theorem~\ref{thm:main-rho}.
\end{pf}
\appendix
\section{Some new a priori estimates}
\label{sec:our a priori}
In this appendix we state and prove a few results related to the 
{\it a priori} estimate proved in \cite[Theorem~1.2]{AHP} (see also
the discussion in \cite[(19.17)]{Simon-on-Kato}).

We start by recalling that estimate.

\begin{thm}[{\cite[Theorem~1.2]{AHP}}]\label{eq:AHP-apriori}
Let $\psi$ be as in \eqref{eigen}. For all $R \in (0,\infty)$, there
exists $C>0$ such that 
\begin{align}\label{eq:old-AHP}
  \sup_{{\mathbf y} \in B({\mathbf x},R)} | \nabla \psi({\mathbf y})|
  \leq  C \sup_{{\mathbf y} \in B({\mathbf x},2R)} | \psi({\mathbf
    y})| 
\end{align}
for all ${\mathbf x} \in {\mathbb R}^{3N}$.
\end{thm}
The proof of \eqref{eq:old-AHP} is based on an `Ansatz' (see
also \eqref{def:phi-alpha-bis} below) for
the solution of the eigenvalue equation, and then on
using elliptic regularity on the resulting equation.
The objective of this Appendix is the following strengthening of
Theorem~\ref{eq:AHP-apriori}: 
\begin{prop}\label{prop:our a priori}
Let \(H\) be the operator in \eqref{H}. For all \(0<r<R\) and
\(E\in\C\) there exists \(C=C(r,R,E)\) (and also 
depending on \(N,Z\)) such that if \(H\psi=E\psi\), \(\psi\in
W_{{\rm loc}}^{2,2}(\R^{3N})\), then
\begin{align}\label{eq:our a priori}
  \|\psi\|_{L^{\infty}(B_{3N}({\bf x}_0,r))}+\|\nabla \psi\|_{L^{\infty}(B_{3N}({\bf x}_0,r))} \le C
  \|\psi\|_{L^2(B_{3N}({\bf x}_{0},R))} \, 
\end{align}
for all ${\bf x}_{0}\in\R^{3N}$.
\end{prop}
\begin{proof}
Define, for \({\bf x}=(x_1,\ldots,x_N)\in\R^{3N}\),
\begin{align}\label{def:tilde-F}\nonumber
  \widetilde{F}({\bf x})&=\sum_{j=1}^{N}
  \big(
  -\tfrac{Z}{2}|x_j|+\tfrac{Z}{2}\sqrt{|x_j|^2+1}\big)
  \\&\quad+\sum_{1\le j<k\le
      N}\big(\tfrac{1}{4}|x_j-x_k|-\tfrac{1}{4}\sqrt{|x_j-x_k|^2+1}\big)\,. 
\end{align}
Note that there exists \(C=C(N,Z)>0\) such that (for \(\Sigma\), see
\eqref{eq:Sigma}) 
\begin{align}\label{eq:bounds-tilde-F}
  |\widetilde{F}({\bf x})|\,,\, |\nabla_{\bf x} \widetilde{F}({\bf x})|\le C
  \quad \text{ for all } {\bf x}\in\R^{3N}\setminus\Sigma\,. 
\end{align}
Next, let (for \(V\), see \eqref{H}) 
\begin{align}\label{def:tilde-G}\nonumber
  \widetilde{G}({\bf x})&=-\Big[\sum_{j=1}^{N}
  \tfrac{Z}{2}\Delta_{{\bf x}}(\sqrt{|x_j|^2+1}) 
  -\sum_{1\le j<k\le N}\tfrac14\Delta_{{\bf x}}(\sqrt{|x_j-x_k|^2+1}) \Big]
  \\&=V({\bf x})- 
  \Delta_{{\bf x}}\widetilde{F}({\bf x})\,.
\end{align}
Since \(|\Delta_{x}(\sqrt{|x|^2+1})|\le 3\) for all
\(x\in\R^3\), there exists  \(C=C(N,Z)>0\) such that
\begin{align}\label{eq:bounds-tilde-G}
  |\widetilde{G}({\bf x})|\le C
  \quad \text{ for all } {\bf x}\in\R^{3N}\,.
\end{align}
Therefore, with 
\begin{align}\label{def:tilde-K}
  \widetilde{K}({\bf x}):= \widetilde{G}({\bf x})-|\nabla_{\bf x}
  \widetilde{F}({\bx })|^{2}-E\,, 
\end{align}
using \eqref{eq:bounds-tilde-F} and \eqref{eq:bounds-tilde-G},
there exists \(C=C(N,Z,E)>0\) such that
\begin{align}\label{eq:bounds-tilde-K}
  |\widetilde{K}({\bf x})|\le C
  \quad \text{ for all } {\bf x}\in\R^{3N}\setminus\Sigma\,.
\end{align}
Define 
\begin{align}\label{def:phi-alpha-bis}
  \widetilde{\psi}:={\rm e}^{-\widetilde{F}}\psi\,,
\end{align}
then, (using that \(H\psi=E\psi\)), \(\widetilde\psi\) satisfies the equation
\begin{align}\label{eq:phi-alpha-bis}
  -\Delta_{\bf x}\widetilde\psi-2\nabla_{\bf x}
  \widetilde{F}\cdot\nabla_{\bf x}\widetilde\psi+\widetilde{K}\widetilde\psi=0\,, 
\end{align}
with 
\begin{align}\label{eq:L-infty coeff}
   \nabla_{\bf x}\widetilde{F}, \widetilde{K}\in L^{\infty}(\R^{3N})\,.
\end{align}
Note that since \(\psi\in W^{2,2}_{\rm loc}(\R^{3N})\), we have that
\(\widetilde{\psi}\in W^{2,2}_{\rm loc}(\R^{3N})\). This follows from
\eqref{eq:bounds-tilde-F} and Hardy's inequality (note that any second
order derivative of \(\widetilde{F}\) behaves like \(|x_j|^{-1}\) and
\(|x_j-x_k|^{-1}\)). 

It follows from Theorem~\ref{thm-GT} in Appendix~\ref{sec:app1}
that \(\widetilde{\psi}\in
C^{1,\theta}_{\rm loc}(\R^{3N})\) for all \(\theta\in(0,1)\). In
particular, since this, \eqref{eq:bounds-tilde-F},
\eqref{eq:bounds-tilde-K}, and
\eqref{eq:phi-alpha-bis} then implies that  
\begin{align}\label{eq:phi-alpha-bis-2}
 -\Delta_{\bf x}\widetilde\psi=2\nabla_{\bf x}
 \widetilde{F}\cdot\nabla_{\bf x}
 \widetilde\psi-\widetilde{K}\widetilde\psi\in 
 L^{p}_{\rm loc}(\R^{3N}) \text{ for all } p\in[1,\infty]\,, 
\end{align}
it follows from Theorem~\ref{thm-Grisvard} 
that \(\widetilde{\psi}\in
W^{2,p}_{\rm loc}(\R^{3N})\) for all \(p\in[2,\infty)\).

It now follows from Theorem~\ref{thm-GT-two} (used on
\eqref{eq:phi-alpha-bis}, with \(p=2\)) that for 
all \(\widetilde{R}, \widetilde{r}>0\) there
exists a constant \(C=C(\widetilde{r},\widetilde{R})\) (depending also on  
\(N,Z,E\) through \eqref{eq:bounds-tilde-F} and
\eqref{eq:bounds-tilde-K}) such that, for all \({\bf
 x}_{0}\in\R^{3N}\),   
\begin{align}\label{eq:first-one}
 \|\widetilde\psi\|_{W^{2,2}(B_{3N}({\bf x}_0,\widetilde{r}))} \le C
 \|\widetilde\psi\|_{L^2(B_{3N}({\bf x}_{0},\widetilde{R}))} \,.
\end{align}

Hence, by Theorem~\ref{thm:Sobolev-embedding}({\rm i}) (Sobolev
embedding; with \(k=2\),
\(n=3N\), and \(p=p_1=2\),
\(q=p_1^*=6N/(3N-4)>2+\frac{8}{3N}=p_1+\frac{8}{3N}\)), and then 
\eqref{eq:first-one}, there exists a
constant \(C=C(\widetilde{r},\widetilde{R})\) such that
\begin{align}\label{eq:second-one}
  \|\widetilde\psi\|_{L^{p_1^*}(B_{3N}({\bf x}_{0},\widetilde{r}))}
  \le C  \|\widetilde\psi\|_{L^2(B_{3N}({\bf x}_{0},\widetilde{R}))} <\infty\,.
\end{align}
Now, using Theorem~\ref{thm-GT-two} again, but this time with \(p=p_2=p_1^*\),
and then \eqref{eq:second-one}, 
we therefore get that, for all \(\hat{r}\in (0,\widetilde{r})\), there
exists a constant \(C=C(\hat{r},\widetilde{r},\widetilde{R})>0\), such that
\begin{align}\label{eq:third-one}
  \|\widetilde\psi\|_{W^{2,p_2}(B_{3N}({\bf x}_0,\hat{r}))} \le C
  \|\widetilde\psi\|_{L^2(B_{3N}({\bf x}_{0},\widetilde{R}))} \,,
\end{align}
with \(p_2>p_1+\tfrac{8}{3N}=2+\tfrac{8}{3N}\).
(Of course the constant \(C\) changes every time.) 

We repeat 
this: Sobolev embedding, in the form of
Theorem~\ref{thm:Sobolev-embedding}({\rm i}) (always with \(k=2\),
\(n=3N\); next time with
\(W^{2,p_2}\) and \(L^{p_2^*}\), \(p_2^*>p_2+\tfrac{8}{3N}\)), and then
Theorem~\ref{thm-GT-two} (with \(p=p_3=p_2^*\)) as long as
\(2p_i<3N\). Note that \(2=p_1<p_2<p_3<\cdots\), with
\begin{align*}
  p_{i+1}=p_i^*=\frac{3Np_i}{3N-2p_i}>p_i+\frac{2p_i^2}{3N}>p_i+\frac{8}{3N}\,,\,
  i=1,2,\ldots \,.
\end{align*}
Hence, we reach \(p_M\) satisfying \(2p_M<3N<2p_{M}^*\) in maximally  
\((3N-2)/(8/3N)+1=(9N^2-6N+8)/8\) steps (that is, \(M\) is smaller
equal this number). As above, the radius of the
smaller ball decreases each time (above, from \(\widetilde{r}\) to
\(\hat{r}\)). However, splitting the original difference
\((R-r)/2=R-(R+r)/2\) in
\(M+1\) equally large parts (we use Theorem~\ref{thm-GT-two} \(M+1\)
times), we get: For all \(0<r<R\) there exists a 
constant \(C=C(r,R)>0\) such that
\begin{align}\label{eq:fourth-one}
  \|\widetilde\psi\|_{W^{2,{p_M}^*}(B_{3N}({\bf x}_0,(r+R)/2))} \le C
  \|\widetilde\psi\|_{L^2(B_{3N}({\bf x}_{0},R))} \,,
\end{align}
with \(2p_M<3N<2{p_M}^*\).

Now use Theorem~\ref{thm:Sobolev-embedding}({\rm ii}) (Morrey's
Theorem): With \(k=2, p={p_M}^*, n=3N\) (so \(kp>n\)), to get, for
some \(\theta\in(0,1)\),
\begin{align}\label{eq:fifth-one}
  \|\widetilde\psi\|_{C^{\theta}(\overline{B_{3N}({\bf x}_{0},(r+R)/2)})} 
  \le C
  \|\widetilde\psi\|_{W^{2,{p_M}^*}(B_{3N}({\bf x}_0,(r+R)/2))}\,. 
\end{align}
Using \eqref{eq:fourth-one}, and that
\(\|\widetilde{\psi}\|_{L^{\infty}}\le\|\widetilde{\psi}\|_{C^{\theta}}\),
  this implies that, for all \(0<r<R\), 
\begin{align}\label{eq:sixth-one}
  \|\widetilde{\psi}\|_{L^{\infty}(B_{3N}({\bf x}_{0},(r+R)/2))}
  \le C \|\widetilde\psi\|_{L^2(B_{3N}({\bf x}_{0},R))} \,,
\end{align}
for some \(C=C(r,R)>0\). 

Hence, using \eqref{eq:phi-alpha-bis}--\eqref{eq:L-infty coeff},
Theorem~\ref{thm-GT} (used on \eqref{eq:phi-alpha-bis}), and
\eqref{eq:sixth-one} give that, for all \(\theta\in(0,1)\), 
\begin{align}\label{eq:seventh-one}\nonumber
  \|\widetilde{\psi}\|_{C^{1,\theta}(B_{3N}({\bf x}_{0},r))}
  &\le C\|\widetilde{\psi}\|_{L^{\infty}(B_{3N}({\bf x}_{0},(r+R)/2))}
  \\&\leq C \|\widetilde\psi\|_{L^2(B_{3N}({\bf x}_{0},R))} \,.
\end{align}

Hence (since
\(\|\widetilde{\psi}\|_{L^{\infty}}+\|\nabla\widetilde{\psi}\|_{L^{\infty}}\le 
\|\widetilde{\psi}\|_{C^{1,\theta}}\)), \eqref{eq:our a priori}
follows, but with \(\widetilde\psi\) 
instead of \(\psi\). It remains to recall that \(\psi={\rm
 e}^{\widetilde{F}}\widetilde\psi\) 
(see \eqref{def:phi-alpha-bis}) with \(\widetilde{F}\) (globally)
Lipschitz (see also \eqref{eq:bounds-tilde-F}), to arrive at
\eqref{eq:our a priori} for \(\psi\).
\end{proof}
As a consequence of Propostion~\ref{prop:our a priori}, we get the
following, which is of independent interest:
\begin{prop}\label{prop:AprioriEstimatesOnRho}
For \(N\ge2\), let \(H\) be the operator in
\eqref{H}.
Then for all \(0<r<R\) and all \(E\in\R\) there exists a constant
$C=C(r,R,E)>0$ such that if \(H\psi=E\psi\), \(\psi\in
W^{2,2}(\R^{3N})\), and if \(\rho\) is the associated
one-electron density as in \eqref{rho}, then,
for all $x_1 \in {\mathbb R}^3$, 
\begin{align}\label{eq:IntegralLinfty_rho_apriori}
  \int_{\R^{3N-3}}\| \psi\|_{L^{\infty}(B_{3N}((x_1, \hat{\bf
      x}_{1}),r))}^2 \,d\hat{\bf x}_{1} &\leq
  C \int_{B_{3}(x_{1},R)}  \!\!\!\!\!\!\!\!\!\!\!\!
  \rho(y_1)\,dy_1
  \,,\\ 
  \label{eq:IntegralLinfty_nabla_rho_apriori}
 \int_{\R^{3N-3}}\| \nabla\psi\|_{L^{\infty}(B_{3N}((x_1, \hat{\bf
      x}_{1}),r))}^2 \,d\hat{\bf x}_{1} &\leq
  C \int_{B_{3}(x_{1},R)} \!\!\!\!\!\!\!\!\!\!\!\!
  \rho(y_1)\,dy_1
  \,,\\ 
  \label{eq:apriori_rho}
  \rho(x_1)=\int_{\R^{3N-3}} | \psi(x_1, \hat{\bf x}_{1})|^2
  \,d\hat{\bf x}_{1}  
  &\leq C \int_{B_{3}(x_{1},R)} \!\!\!\!\!\!\!\!\!\!\!\!
  \rho(y_1)\,dy_1\,,
\end{align}
and
\begin{align}\label{eq:apriori_rho_derivative}
  \big|\nabla\rho(x_1)\big|=\big|\int_{\R^{3N-3}} \nabla_{x_1}(| \psi(x_1, \hat{\bf x}_{1})|^2)
  \,d\hat{\bf x}_{1}\big| \leq C  \int_{B_{3}(x_{1},R)} \!\!\!\!\!\!\!\!\!\!\!\!
  \rho(y_1)\,dy_1\,,
\end{align}
in the sense that, for all \(v\in\R^{3}\), the directional derivative
(exists and) satisfies
\begin{align}\label{eq:apriori_rho_derivative-BIS}
  \big|v\cdot\nabla\rho(x_1)\big|
  \le C |v|  \int_{B_{3}(x_{1},R)} \!\!\!\!\!\!\!\!\!\!\!\!
  \rho(y_1)\,dy_1\,.
\end{align}
Furthermore, for all $b\in [0,3)$ and \(R>0\) there exists
$C=C(b,R,E)>0$ such that  
\begin{align}\label{eq:apriori_rho_coulomb}
  \int_{\R^{3N-3}} |x_2|^{-b} | \psi(x_1, \hat{\bf x}_{1})|^2
  \,d\hat{\bf x}_{1} \leq C \int_{B_{3}(x_{1},R)}
  \!\!\!\!\!\!\!\!\!\!\!\!
  \rho(y_1)\,dy_1  \,. 
\end{align}
\end{prop} 
\begin{remark}\label{rem:rho-psi-est}
Note that, for all \(x_1\in\R^{3}\), \(R>0\), 
\begin{align}\label{eq:rho-psi-est}
  \int_{B_{3}(x_{1},R)} \!\!\!\!\!\!\!\!\!\!\!\!
  \rho(y_1)\,dy_1 \le \|\rho\|_{L^1(\R^3)}=
 \|\psi\|_{L^{2}(\R^{3N})}^2<\infty\,.
\end{align}
In particular, it follows from \eqref{eq:apriori_rho_derivative} that \(\rho\)
is globally Lipschitz: \(\rho\in C^{0,1}(\R^{3})\). This was already
known \cite[Theorem~1.11 (i)]{AHP}.
\end{remark}
\begin{proof}
We start by proving \eqref{eq:IntegralLinfty_rho_apriori} and
\eqref{eq:IntegralLinfty_nabla_rho_apriori} from which
the other estimates will follow in a simple manner. 
Using \eqref{eq:our a priori} and Fubini's Theorem,
\begin{align}\label{est:smart}
  \int_{\R^{3N-3}} &\| \psi\|_{L^{\infty}(B_{3N}((x_1, \hat{\bf
      x}_{1}),r))}^2 \,d\hat{\bf x}_{1} 
  \leq C\int_{\R^{3N-3}}
  \| \psi\|_{L^{2}(B_{3N}((x_1, \hat{\bf
      x}_{1}),R))}^2 \,d\hat{\bf x}_{1} \nonumber \\ 
  &=C \int_{\R^{3N}} |\psi({\bf y})|^2\Big( \int_{\{|{\bf
      y}-(x_1,\hat{\bf x}_{1})|\leq R\}}
  \!\!\!\!\!\!\!\!\!\!\!\!\!\!\!\!\!\!\!\!
  d\hat{\bf x}_{1}\qquad
 \Big)\,d{\bf y}.
\end{align}
Now, for all \({\bf y}=(y_1,\hat{\bf y}_{1})\in\R^{3N}\), 
\begin{align}\label{eq:char-fct-trick-1}
  \int_{\{|{\bf
      y}-(x_1,\hat{\bf x}_{1})|\leq R\}}
  \!\!\!\!\!\!\!\!\!\!\!\!\!\!\!\!\!\!\!\!
  d\hat{\bf x}_{1}
  \qquad\le \1_{\{|y_1-x_1|\le R\}} 
   \int_{\{|\hat{\bf y}_{1}-\hat{\bf x}_{1})|\leq R\}}
   \!\!\!\!\!\!\!\!\!\!\!\!\!\!\!\!\!\!\!\!
   d\hat{\bf x}_{1}\qquad\,,
\end{align}
and the last integral equals the volume of \(B_{3N-3}(0,R)\) for all
\(\hat{\bf y}_{1}\in\R^{3N-3}\). Inserting this in \eqref{est:smart}
and using the definition of \(\rho\) in \eqref{rho} finishes the proof of 
\eqref{eq:IntegralLinfty_rho_apriori}. The
proof of \eqref{eq:IntegralLinfty_nabla_rho_apriori} is similar.

To prove \eqref{eq:apriori_rho} notice that 
\begin{align*}
  \int_{\R^{3N-3}} | \psi(x_1, \hat{\bf x}_{1})|^2 \,d\hat{\bf x}_{1}
  \leq  
  \int_{\R^{3N-3}} \| \psi\|_{L^{\infty}(B_{3N}((x_1, \hat{\bf
    x}_{1}),R/2))}^2 \,d\hat{\bf x}_{1} 
\end{align*}
and use \eqref{eq:IntegralLinfty_rho_apriori} with \(r=R/2\).

To prove \eqref{eq:apriori_rho_derivative}
we differentiate and estimate, to get that
\begin{align}\label{the-last-one}
  |\nabla_{x_1}(| &\psi(x_1, \hat{\bf x}_{1})|^2)| \nonumber \\&\leq
  2\| \psi\|_{L^{\infty}(B_{3N}((x_1, \hat{\bf x}_{1}),R/2))} 
  \| \nabla \psi\|_{L^{\infty}(B_{3N}((x_1, \hat{\bf x}_{1}),R/2))}.
\end{align}
Here \eqref{the-last-one} should be understood in terms of directional
derivatives in the same way as in
\eqref{eq:apriori_rho_derivative-BIS}. From
\cite[Proposition~1.5]{AHP} we know that the 
directional derivatives of \(\psi\) exist.

At this point we can use \eqref{eq:our a priori} and finish the
estimate as above.

To prove \eqref{eq:apriori_rho_coulomb} it suffices, using
\eqref{eq:apriori_rho}, to estimate 
\begin{align*}
  \int_{\{|x_2| \leq R/4\}} |x_2|^{-b} | \psi(x_1, \hat{\bf x}_{1})|^2
  \,d\hat{\bf x}_{1}\,. 
\end{align*}
We argue in a similar fashion as above, with \(\hat{\bf x}_{1,2} =
(x_3,\ldots,x_N)\). Since \((x_1,x_2,\hat{\bf x}_{1,2})\in
B_{3N}((x_1, 0, \hat{\bf x}_{1,2}),R/2)\) for all \(|x_2|\le R/4\), we get
from Fubini's Theorem and
\eqref{eq:our a priori} that
\begin{align}\label{est:smart-2}
  &\int_{\{|x_2| \leq R/4\}} 
  |x_2|^{-b} | \psi(x_1, \hat{\bf
    x}_{1})|^2 \,d\hat{\bf x}_{1}\nonumber \\ 
  &\leq \int_{\R^{3N-6}}\Big(\int_{\{|x_2| \leq R/4\}} |x_2|^{-b}\,dx_2\Big) \|
  \psi\|^2_{L^{\infty}(B_{3N}((x_1, 0, \hat{\bf x}_{1,2}),R/2))}
  \,d\hat{\bf x}_{1,2} \nonumber \\ 
  &\leq C(b,R) \int_{\R^{3N-6}}  \| \psi
  \|_{L^{2}(B_{3N}((x_1, 0, \hat{\bf x}_{1,2}),R))}^2\,d\hat{\bf x}_{1,2}
  \nonumber \\ 
  &=C(b,R)\int_{\R^{3N}} |\psi({\bf y})|^2
  \Big(\int_{\{|{\bf y}-(x_1,0,\hat{\bf x}_{1,2})|\leq R\}} 
  \!\!\!\!\!\!\!\!\!\!\!\!\!\!\!\!\!\!\!\!\!\!
  d\hat{\bf x}_{1,2}
  \qquad\Big)\,d{\bf y}\,. 
\end{align}
Here we also used that $b\in [0,3)$.
Now, for all \({\bf y}=(y_1,y_2,\hat{\bf y}_{1,2})\in\R^{3N}\), 
\begin{align}\label{eq:char-fct-trick-1}
  \int_{\{|{\bf y}-(x_1,0,\hat{\bf x}_{1,2})|\leq R\}} 
  \!\!\!\!\!\!\!\!\!\!\!\!\!\!\!\!\!\!\!\!\!\!
  d\hat{\bf x}_{1,2}
  \qquad\le \1_{\{|y_1-x_1|\le R\}} 
   \int_{\{|\hat{\bf y}_{1,2}-\hat{\bf x}_{1,2})|\leq R\}}
   \!\!\!\!\!\!\!\!\!\!\!\!\!\!\!\!\!\!\!\!
   d\hat{\bf x}_{1,2}\qquad\,,
\end{align}
and the last integral equals the volume of \(B_{3N-6}(0,R)\) for all
\(\hat{\bf y}_{1,2}\in\R^{3N-6}\). Inserting this in \eqref{est:smart-2}
and using the definition of \(\rho\) in \eqref{rho} finishes the proof
of \eqref{eq:apriori_rho_coulomb}. 
\end{proof}

\section{A partition of unity}
\label{sec:partition}
In this appendix we gather various facts about a particular partition
of unity (on \(\R^{3N}\)), needed when studying the electron density
\(\rho\); see Section~\ref{sec:proof-rho}.

We denote by $C_b^{\infty}(\Omega)$
the set of all smooth functions on \(\Omega\) which are bounded
together with all their derivatives. 

Let \(\chi_{1}, \chi_{2}\in C_{b}^{\infty}(\R)\),
\(0\le\chi_{i}\le1\), \(i=1,2\), \(\chi_1, \chi_2\) both monotone,
with  
\begin{eqnarray}\label{partition:supports}
  \chi_{1}(t)=\left\{\begin{array}{ll}
   1\,, & t \le 1/4\,,\\
   0\,, & t \ge 3/4\,,
   \end{array}\right.
   \quad \text{ and }\quad
  \chi_{2}(t)=\left\{\begin{array}{cc}
   0\,, & t \le 1/4\,,\\
   1\,, & t \ge 3/4\,,
   \end{array}\right.
\end{eqnarray}
and 
\begin{align}\label{partition}
  \chi_{1}(t)+\chi_{2}(t)=1 \text{ for all }t\in\R\,.
\end{align}
The partition of unity depends on an index $I \in X$, where $X =
\cup_{J=0}^{N-1} X_J$, with the \(X_J\)'s to be described below. Here 
\begin{align*}
  X_{J=0} = \{ (0, \{2,\ldots,N\}, \emptyset)\},
\end{align*}
and the corresponding function in the partition of unity is (with 
\({\bf x}=(x_1,\ldots,x_N)\in\R^{3N}\)),
\begin{align}\label{first-fct}
  \chi_{(0, \{2,\ldots,N\}, \emptyset)}({\bf x})=
  \prod_{j\in
     \{2,\ldots,N\}}\chi_{1}\left(\frac{|x_j|}{|x_1|}\right).
\end{align}
For $J\geq1$, $X_J$ consists of all elements of the form
$(J,P_J,Q_{J-1},\ldots,Q_0)$ with  
 \(Q_0,Q_1,\ldots,Q_{J-1},P_J\subset\{2,\ldots,N\}\) disjoint, and
\(P_J\cup\big(\cup_{s=0}^{J-1}Q_s\big)=\{2,\ldots,N\}\) (possibly
\(P_J=\emptyset\) or \(Q_s=\emptyset, s\ge1\)).
The corresponding function is (with \(\prod_{j\in\emptyset} = 1\))
\begin{align}\label{chi-I}
  &\chi_I({\bf x})=\chi_{(J,P_J, Q_{J-1},\ldots,Q_0)}({\bf x})
   \\\notag
   &=\left[\prod_{j\in
     P_{J}}\chi_{1}\left(\frac{4^J|x_j|}{|x_1|}\right)\right]
  \left[\prod_{j\in
      Q_{J-1}}\chi_{2}\left(\frac{4^{J-1}|x_j|}{|x_1|}\right)\chi_{1}\left(\frac{4^{J-2}|x_j|}{|x_1|}\right)\right] 
   \times
   \\\notag
   &\times\cdots\times \left[\prod_{j\in
      Q_{s}}\chi_{2}\left(\frac{4^{s}|x_j|}{|x_1|}\right)\chi_{1}\left(\frac{4^{s-1}|x_j|}{|x_1|}\right)\right] 
  \times\cdots\times\\\notag
  &\times\left[\prod_{j\in
      Q_{1}}\chi_{2}\left(\frac{4^{1}|x_j|}{|x_1|}\right)\chi_{1}\left(\frac{4^{0}|x_j|}{|x_1|}\right)\right] 
   \left[\prod_{j\in
     Q_{0}}\chi_{2}\left(\frac{4^0|x_j|}{|x_1|}\right)\right]\,.
\end{align}
\begin{lemma}\label{lem:partition}
 Let \(\chi_{1}\) and \(\chi_{2}\) be as above (see
 \eqref{partition:supports}--\eqref{partition}), then (as functions of 
   \({\bf x}=(x_1,\ldots,x_N)\in\R^{3N}\)), 
\begin{align}\label{partition-lem}
  1 = \sum_{I}\chi_{I}\,,
\end{align}
where the sum is over a subset of $X$.
\end{lemma}
\begin{pf}
To ease notation, let, for \({\bf
  x}=(x_1,\ldots,x_N)\in\R^{3N}\), 
\begin{align}\label{partition:notation}
  \chi_{i,j}^{s}({\bf x})=\chi_{i}\Big(\frac{4^{s}|x_j|}{|x_1|}\Big)\,,
  i=1,2\,,\, j=2,\ldots,N\,,s=0,1,2,\ldots\,.
\end{align}
Note that, by \eqref{partition:supports}, for all \(j\) and \(s=1,2,\ldots\),
\begin{align}\label{partition: product-prop}
  \chi_{1,j}^{s}\chi_{1,j}^{s-1}=\chi_{1,j}^{s}\,.
\end{align}
Using \eqref{partition} we have (again, with \(\prod_{j\in\emptyset} = 1\))
\begin{align}\label{part:first insert}
  1=
  \prod_{j=2}^{N}\Big[\chi_{1,j}^{0}+\chi_{2,j}^{0}\Big]
  =\sum_{p_0\cup q_0=\{2,\ldots,N\},p_0\cap q_0=\emptyset}
  \Big[\prod_{j\in p_0} \chi_{1,j}^{0} \Big]
  \Big[\prod_{j\in q_0}\chi_{2,j}^{0}\Big]\,.
\end{align}

The term in \eqref{part:first insert} with \(q_0=\emptyset\) equals
\begin{align}\label{partition:first p}
   \prod_{j\in \{2,\ldots,N\}}\chi_{1,j}^{0} = 
   \chi_{(0,\{2,\ldots,N\},\emptyset)}\,.
\end{align}

The term in \eqref{part:first insert} 
with \(p_0=\emptyset\) equals
\begin{align}\label{partition:first q}
  \prod_{j\in \{2,\ldots,N\}}\chi_{2,j}^{0} = 
   \chi_{(1,\emptyset,\{2,\ldots,N\})}\,.
\end{align}
 
For all other terms \(\chi_{p_0,q_0}=\Big[\prod_{j\in p_0} \chi_{1,j}^{0} \Big]
  \Big[\prod_{j\in q_0}\chi_{2,j}^{0}\Big]\) in \eqref{part:first insert} we
have \(q_0\neq\emptyset\neq p_0\), and so
\(0< \#p_0 < \# \{2,\ldots,N\}=N-1\). In each of these terms, insert a
factor of (recall \eqref{partition:notation} and \eqref{partition})
\begin{align}\label{part:second insert}
  1 &=  \prod_{j\in p_0}\Big[\chi_{1,j}^{1}+\chi_{2,j}^{1}\Big]\,,
\end{align}
and multiply out, to get
\begin{align}\label{partition: mess}
  \chi_{p_0,q_0}&=
  \Big[\prod_{j\in p_0}\chi_{1,j}^{0}\Big] \cdot1\cdot
  \Big[\prod_{j\in q_0}\chi_{2,j}^{0}\Big]
  \\&=\sum_{q_1\cup p_1=p_0, q_1\cap p_1=\emptyset}
  \Big[\prod_{j\in p_0}\chi_{1,j}^{0}\Big]
  \Big[\prod_{j\in p_1}\chi_{1,j}^{1}\Big]
  \Big[\prod_{j\in q_1}\chi_{2,j}^{1}\Big]
  \Big[\prod_{j\in q_0}\chi_{2,j}^{0}\Big]\,.
  \notag
\end{align}
By \eqref{partition: product-prop}, 
\(\chi_{1,j}^{0}\chi_{1,j}^{1}=\chi_{1,j}^{1}\)
for all \(j\in p_1\subseteq p_0\), and so,
since \(p_0=q_1\cup p_1\), each of the terms in the sum in 
\eqref{partition: mess} is of the form
\begin{align}\label{partition: generic 1}
  \chi_{p_1,q_1,q_0}=\Big[\prod_{j\in p_1}\chi_{1,j}^{1}\Big]
  \Big[\prod_{j\in q_1}\chi_{2,j}^{1}\chi_{1,j}^{0}\Big]
  \Big[\prod_{j\in q_0}\chi_{2,j}^{0}\Big]\,.
 \end{align}
As before, the term with \(p_1=\emptyset\) (that is, \(q_1=p_0\)) equals
\begin{align}\label{partition:second q}
  \Big[\prod_{j\in q_1}\chi_{2,j}^{1}\chi_{1,j}^{0}\Big]
  \Big[\prod_{j\in q_0}\chi_{2,j}^{0}\Big]
  = \chi_{(2,\emptyset,q_1,q_0)}
\end{align} 
and the term with \(q_1=\emptyset\) (that is, \(p_1=p_0\)) equals 
\begin{align}\label{partition:second p}
  \Big[\prod_{j\in p_1}\chi_{1,j}^{1}\Big]
  \Big[\prod_{j\in q_0}\chi_{2,j}^{0}\Big]
    = \chi_{(1,p_1,q_0)}\,.
\end{align}

For the rest of the terms in \eqref{partition: generic 1}, we have
\(q_1\neq\emptyset\neq p_1\), and so
\(0<\# p_1<\# p_0 < N-1\), that is, \(0<\# p_1<N-2\). For each of
these terms \(\chi_{p_1,q_1,q_0}\) 
(with \(p_1\cup q_1\cup q_0=\{2,\ldots,N\}\), \(p_1,q_1,q_0\)
disjoint), insert a factor of   
\begin{align}\label{partition: third factor}
  1 = \prod_{j\in p_1}\Big[\chi_{1,j}^{2}+\chi_{2,j}^{2}\Big]\,,
\end{align}
 and proceed as above, using \eqref{partition: product-prop} with
 \(s=2\), to write
 \(\chi_{p_1,q_1,q_0}\) as a sum (over \(p_2,q_2\) with \(p_2\cup
 q_2=p_1\), \(p_2\cap q_2=\emptyset\)) of terms of the form
 \begin{align}\label{partition: generic 2}
   \chi_{p_2,q_2,q_1,q_0}=\Big[\prod_{j\in p_2}\chi_{1,j}^{2}\Big]
   \Big[\prod_{j\in q_2}\chi_{2,j}^{2}\chi_{1,j}^{1}\Big]
   \Big[\prod_{j\in q_1}\chi_{2,j}^{1}\chi_{1,j}^{0}\Big]
   \Big[\prod_{j\in q_0}\chi_{2,j}^{0}\Big]\,.
 \end{align}
Again, the terms with $p_2 = \emptyset$ or $q_2 = \emptyset$ have (see
\eqref{chi-I}) the correct form
(namely, with \(J=3\), \(P_3=p_2\), \(Q_i=q_i, i=0,1,2\),
\(I=(3,\emptyset,Q_2,Q_1,Q_0)\in 
X_{3}\), and,
respectively, with \(J=2\), \(P_2=p_2, Q_i=q_i, i=0,1\), \(I=(2,
P_2,Q_1,Q_0)\in X_{2}\)). 
Furthermore, for all other terms
\(\chi_{p_2,q_2,q_1,q_0}\) in \eqref{partition: generic 2}, we have
\(q_2\neq \emptyset\neq p_2\), hence, 
 \(0<\# p_2 < \# p_1<N-2\), that is, \(0<\# p_2 < N-3\). 
Continuing like this, we get a sum of terms of the form in
\eqref{chi-I}, with the size of \(p_j\) diminishing at each step,
until \(\# p_k=1\) (which occurs for \(k=N-3\)). Then the above two
possibilities---\(p_k=\emptyset\) or \(q_k=\emptyset\)---are the only
two, and we are done. \end{pf}
The localization functions \(\chi_{I}\) above are constructed 
in order to have the following lemma, bounding certain terms in the
Coulomb-potential by \(|x_1|^{-1}\), on the support of
\(\chi_{I}\). 
\begin{lemma}\label{partition:control}
Let \(\chi_{1}\) and \(\chi_{2}\) be as in
\eqref{partition:supports}--\eqref{partition}, and define \(\chi_{I}\)
as in \eqref{chi-I}.

 Then there exists a constant \(C=C(N)>0\) such that for all \({\bf
   x}=(x_1,\ldots,x_N)\in\supp\,\chi_I\):
 \begin{align}\label{partition:est-control-one}
   |x_j|^{-1}&\le C |x_1|^{-1} 
   \ \ \text{ for all } j\in \cup_{j=0}^{J-1}Q_{j}\,,\\
   \label{partition:est-control-two} 
   |x_1-x_j|^{-1}&\le C |x_1|^{-1}
   \ \ \text{ for all } j\in \big(\cup_{j=1}^{J-1}Q_{j}\big) \cup P_{J}\,,\\
   |x_j-x_k|^{-1}&\le C |x_1|^{-1}
   \ \ \text{ for all } j\in P_J, k\in \cup_{j=0}^{J-1}Q_{j}\,.
   \label{partition:est-control-three}
 \end{align}
\end{lemma}
\begin{pf}
To prove \eqref{partition:est-control-one} note that, since
\(\chi_I({\bf x})\neq0\), for all the stated \(j\)'s we have
\(\chi_2(4^s|x_j|/|x_1|)\neq0\) for some \(s\in\{1,\ldots,J-1\}\),
\(J\le N\). Hence, by \eqref{partition:supports}, 
\begin{align*}
  |x_j|\ge \frac{1}{4}\frac{1}{4^s}|x_1|\ge \frac{1}{4}\frac{1}{4^N}|x_1|=c_{N}|x_1|\,,
\end{align*}
which proves \eqref{partition:est-control-one}.

To prove \eqref{partition:est-control-two}, note that, for
\(j\in P_J\), we have
\(\chi_1(4^J|x_j|/|x_1|)\neq0\). Hence, by \eqref{partition:supports},
\(|x_j|\le \tfrac34|x_1|\), and so \(|x_1-x_j|\ge \tfrac14|x_1|\) for
these \(j\).

On the other hand, for \(j\in Q_1\cup\ldots\cup Q_{J-1}\),
\(\chi_1(4^{s-1}|x_j|/|x_1|)\neq0\) for some \(s\in\{1,\ldots,J-1\}\) ,
\(J\le N\). Hence, by \eqref{partition:supports}, \(|x_j|\le
\tfrac{3}{4}\tfrac{1}{4^{s-1}}|x_1|\le \tfrac34|x_1|\), and so
\eqref{partition:est-control-two} holds also for these \(j\)'s. 

Finally, to prove \eqref{partition:est-control-three}, note that for the
stated \(j\)'s, we have \(|x_j|\le \tfrac34\tfrac{1}{4^J}|x_1|\), and
for the stated \(k\)'s, we have, for some \(s\in\{1,\ldots,J-1\}\), 
\begin{align*}
  |x_k|\ge\frac14\frac{1}{4^{s}}|x_1|\ge \frac14\frac{1}{4^{J-1}}|x_1|\,.
\end{align*}
Therefore, 
\begin{align*}
  |x_j-x_k|&\ge |x_k|-|x_j|\ge
  \frac{1}{4}\frac{4}{4^J}|x_1|-\frac34\frac{1}{4^J}|x_1| 
  \\&=\frac{1}{4}\frac{1}{4^J}|x_1|\ge \frac{1}{4}\frac{1}{4^N}|x_1|=c_N|x_1|\,,
\end{align*}
which proves \eqref{partition:est-control-three}.
\end{pf}
\begin{remark}
This last argument is the reason why we need \(4^J\) in the \(\chi_1\)
in the \(P_J\)-factor, and at most \(4^{J-1}\) in the \(\chi_2\) in the \(Q_s\)-factors, in 
\eqref{chi-I}.
\end{remark}
The next lemma uses the previous one, to control derivatives with
respect to \(x_1\) of (a slightly changed version of) the 
localization functions \(\chi_{I}\).
\begin{lemma}\label{lem:der-chi's}
Let \(\chi_{1}\) and \(\chi_{2}\) be as in
\eqref{partition:supports}--\eqref{partition}, 
and let $\chi_I$ be as in Lemma~\ref{lem:partition}.
For  \({\bf x}=(x_1,\ldots,x_N)\in\R^{3N}\),  define
${\bf \tilde{x}} =(\tilde{x}_1,\ldots, \tilde{x}_N)$ with 
\begin{align*}
\tilde{x}_j=\begin{cases}x_j,& \text {if } j=1 \text{ or } j \in P_J,\\
x_1 + x_j,& \text{else}.
\end{cases}
\end{align*}
Define finally 
\begin{align}\label{chi-I-tilde}
  \widetilde\chi_I({\bf x})=\chi_I({\bf \tilde{x}}).
\end{align}

Then, for all \(\beta\in\N_{0}^{3}\) there exists a constant
\(C=C(I,\beta)\) such that, for all  \({\bf 
  x}=(x_1,\ldots,x_N)\in\R^{3N}\),  
\begin{align}\label{est:der-chi's}
  \big|(\partial_{x_1}^{\beta}\widetilde\chi_{I})({\bf x})\big|
  \le C\,|x_1|^{-|\beta|}\,.
\end{align}

Furthermore, if \(|\beta|\ge 1\), then there exists
\(j\in\{2,\ldots,N\}\) and a 
constant \(C=C(I,\beta,j)\) such that, for all  \({\bf 
  x}=(x_1,\ldots,x_N)\in\R^{3N}\) and all  \(n\in\N_{0}\), 
\begin{align}\label{est:der-chi's-BIS}
  \big|(\partial_{x_1}^{\beta}\widetilde\chi_{I})({\bf x})\big|
  \le C\,|x_j|^{-n}|x_1|^{n-|\beta|}\,
  \intertext{ or }
  \label{est:der-chi's-BIS-2}
  \big|(\partial_{x_1}^{\beta}\widetilde\chi_{I})({\bf x})\big|
  \le C\,|x_1+x_j|^{-n}|x_1|^{n-|\beta|}\,.
\end{align}
\end{lemma}
\begin{pf}
First note that, by Leibniz' rule, to prove \eqref{est:der-chi's} it
suffices to prove that for all \(\gamma\in\N_{0}^{3}\),  there 
exists a constant such that
\begin{align}\label{est:der-f's}
  \big|(\partial_{x_1}^{\gamma}f)(x_1)\big|
  \le C\,|x_1|^{-|\gamma|}\,
\end{align}
for \(f\) any of the functions
\begin{align}\label{functions}
 \chi_{1}\left(\frac{4^k|x_j|}{|x_1|}\right)\,,\ 
 \chi_{1}\left(\frac{4^{k}|x_1+x_j|}{|x_1|}\right)\,,\ 
  \chi_{2}\left(\frac{4^{k}|x_1+x_j|}{|x_1|}\right)\,,
\end{align}
(\(k\in\{0,\ldots,N\}, j\neq 1\)). By the choice of \(\chi_1\) and
\(\chi_2\), the bound \eqref{est:der-f's} is trivial for 
\(\gamma=0\) (with \(C=1\)). In particular, \eqref{est:der-chi's}
trivially holds if 
\(\beta=0\) (again, with \(C=1\)). 

Secondly, note that in each case, for any
\(\gamma\in\N_{0}^{3}\setminus\{0\}\), 
\begin{align}\label{Faa-di-Bruno}\nonumber
  \big(\partial_{x_1}^{\gamma}&f\big)(x_1)\\
  &=\sum_{\underset{\gamma_1+\cdots+\gamma_s=\gamma}{1\le m\le|\gamma|}}
  c_{m,\gamma}\chi_{i}^{(m)}(g(x_1))
  (\partial^{\gamma_1}_{x_1}g)(x_1)\cdot\ldots\cdot(\partial^{\gamma_s}_{x_1}g)(x_1)\,,
\end{align}
with \(i=1\) or \(2\), and \(g(x_1)\) either \(\frac{4^k|x_j|}{|x_1|}\) or 
\(\frac{4^k|x_1+x_j|}{|x_1|}\). On \(\supp (\chi_i^{(m)}\!\!\circ g)\)
(\(m\ge1\)) we
have, in all cases (see \eqref{partition:supports})
\begin{align}\label{bound g}
  \frac14\le g(x_1)\le \frac34\,.
\end{align}
Hence, if \(g(x_1)=\frac{4^k|x_j|}{|x_1|}\), then for any
\(\gamma\in\N_{0}^{3}\setminus\{0\}\), on \(\supp
(\chi_i^{(m)}\!\!\circ g)\), 
\begin{align}\label{est:first-der-g}
  \big|(\partial^{\gamma}_{x_1}g)(x_1)\big|
  \le c_{\gamma,k}|x_j|\,|x_1|^{-1-|\gamma|}\le
  \tilde{c}_{\gamma,k}|x_1|^{-|\gamma|}\,.
\end{align}
On the other hand, if \(g(x_1)=\frac{4^k|x_1+x_j|}{|x_1|}\), then
for any \(\gamma\in\N_{0}^{3}\setminus\{0\}\),
again on \(\supp (\chi_i^{(m)}\!\!\circ g)\), 
\begin{align}\label{est:second-der-g}\nonumber
  \big|(\partial^{\gamma}_{x_1}g)(x_1)\big|
  &=\Big| \sum_{\sigma\le\gamma}\binom{\gamma}{\sigma}
  \big(\partial^{\sigma}_{x_1}|x_1+x_j|\big)
  \big(\partial_{x_1}^{\gamma-\sigma}|x_1|^{-1}\big)\Big|\\
  &\le\sum_{\sigma\le\gamma} c_{\gamma,\sigma}|x_1+x_j|^{1-|\sigma|}
  |x_1|^{-1-|\gamma|+|\sigma|}
  \le \tilde{c}_{\gamma,k}|x_1|^{-|\gamma|}\,.
\end{align}
In both \eqref{est:first-der-g} and \eqref{est:second-der-g}, the
second inequality follows from \eqref{bound g}.

Hence, \eqref{Faa-di-Bruno}, \eqref{est:first-der-g},
\eqref{est:second-der-g}, and the fact that all derivatives of
\(\chi_1\) and \(\chi_2\) are globally bounded, imply that
\begin{align}\label{est:final-der-f}\nonumber
  \big| \big(\partial_{x_1}^{\gamma}&f\big)(x_1)\big|\\
  &\le \sum_{\underset{\gamma_1+\cdots+\gamma_s=\gamma}{1\le m\le|\gamma|}}
  \tilde{c}_{m,\gamma}|x_1|^{-|\gamma_1|}\cdot\ldots\cdot|x_1|^{-|\gamma_s|}
  =C|x_1|^{-|\gamma|}\,.
\end{align}
This finishes the proof of \eqref{est:der-f's}
in the case \(|\gamma|\ge1\), and hence the proof of \eqref{est:der-chi's}. 

To prove that \eqref{est:der-chi's-BIS}
or \eqref{est:der-chi's-BIS-2} hold when \(|\beta|\ge1\), notice that
in this case at least one of the functions in the product in
\eqref{chi-I-tilde} (that is, in \eqref{functions}) gets
differentiated (that is, \(|\gamma|\ge1\)). For this one, do as above, 
but use additionally \eqref{bound g} to get, for all \(n\in\N\), 
\begin{align}\label{extra-factors-g}
  \le \frac{|x_1|^n}{|x_j|^n}
  \quad \text{ or }\quad
  \le \frac{|x_1|^n}{|x_1+x_j|^n}\,.
\end{align}
(As before, on \(\supp (\chi_i^{(m)}\!\!\circ g)\)). Applying this in
\eqref{est:first-der-g} or \eqref{est:second-der-g} yields
\eqref{est:der-chi's-BIS} or \eqref{est:der-chi's-BIS-2}.
\end{pf}
\section{Needed {\it a priori} estimates}
\label{sec:app1}
In this section we collect needed results from the literature.

We start by Sobolev embedding.
\begin{thm}[{\cite[Theorem 6 p.~284]{Evans}, \cite[4.12 Theorem
    p.~85]{Adams}}]\label{thm:Sobolev-embedding} 
Let \(\Omega\subset\R^{n}\) be open and bounded, and let \(k\in\N
,p\ge1\). 
\begin{itemize}
\item[(i)]
Assume \(\Omega\) satisfies an
interior cone condition. Then, 
for any \(k, p\) with 
\(kp<n\), we have the continuous embedding
\begin{align}\label{sobolev}
  W^{k,p}(\Omega)\hookrightarrow L^q(\Omega)\, \
  \text{ for all } q\in[p,p^*], \text{with } p^*:=np/(n-kp)\,.
\end{align}
Moreover, there exists a constant \(C=C(k,p,n,\Omega)\) such that
\begin{align}\label{eq:Sob}
  \|u\|_{L^{q}(\Omega)}\le C\|u\|_{W^{k,p}(\Omega)} \ \text{ for all }
  u\in W^{k,p}(\Omega)\,.
\end{align}
\item[(ii)]
Assume \(\Omega\) is locally Lipschitz. Then for
\(kp>n\),  we have the continuous embedding
\begin{align}\label{morrey}
   W^{k,p}(\Omega)&\hookrightarrow C^{k-1-[n/p],\theta}(\overline{\Omega})\, \
   \text{ for all } \theta\in[0,\theta_0]\,, 
   \\
   \theta_0
   &=\begin{cases}[n/p]+1-(n/p) ,  \text{ if } n/p \text{ is not
         an integer}\,,\\ \text{ any positive number less than } 1 \text{
         if } n/p \text{ is an integer}\,.
\end{cases}\nonumber
\end{align}
Moreover, there exists a constant \(C=C(k,p,n,\theta,\Omega)\) such that
\begin{align}\label{eq:Morrey}
  \|u\|_{C^{k-1-[n/p],\theta}(\overline{\Omega})}\le C\|u\|_{W^{k,p}(\Omega)} \ \text{ for all }
  u\in W^{k,p}(\Omega)\,.
\end{align}
\end{itemize}
\end{thm}

Next, we list some results on elliptic regularity.

The following is adapted from \cite[Theorem~8.32]{GT}
by choosing \(a^{ij}=\delta_{ij}, b^i=f^i=0\), \(i,j=1,\ldots,n\).
\begin{thm}[{\cite[Theorem~8.32]{GT}}]\label{thm-GT}
Let \(\theta\in(0,1)\), and let \(u\in C^{1,\theta}(\Omega)\) be a
weak solution of  
\begin{align}\label{eq:GT}
  \big({}-\Delta+c(x)\cdot\nabla+d(x)\big)u = g
\end{align}
in a bounded domain \(\Omega\subset\R^{n}\), with \(c_i,d,g\in
L^{\infty}(\Omega)\), with
\begin{align}\label{bounds-coeff:GT}
  \max_{i=1,\ldots,n}\big\{\|c_i\|_{L^{\infty}(\Omega)}\big\}, \|d\|_{L^{\infty}(\Omega)}\le K\,.
\end{align}

Then for any subdomain
\(\Omega'\subset\subset\Omega\) we have
\begin{align}\label{est:GT}
   \|u\|_{C^{1,\theta}(\overline{\Omega'})}\le
   C\big(\|u\|_{L^{\infty}(\Omega)}+\|g\|_{L^{\infty}(\Omega)}\big)\,, 
\end{align}
for \(C=C(n,K,d')\) where \(d'=\dist(\Omega',\Omega)\). 
\end{thm}

The following is adapted from \cite[Theorem~9.11]{GT}
by choosing \(a^{ij}=\delta_{ij}\), \(i,j=1,\ldots,n\).
\begin{thm}[{\cite[Theorem~9.11]{GT}}]\label{thm-GT-two}
Let \(\Omega\) be an open set in \(\R^{n}\) and 
\(u\in W^{2,p}_{{\rm
    loc}}(\Omega)\cap L^{p}(\Omega)\), \(1<p<\infty\), a strong
solution of the equation
\begin{align}\label{eq:GT-two}
  \big({}-\Delta+b(x)\cdot\nabla+c(x)\big)u = f
\end{align}
in \(\Omega\) with \(b_i,c\in
L^{\infty}(\Omega)\), \(f\in L^p(\Omega)\), with
\begin{align}\label{bounds-coeff:GT-L-2}
  \max_{i=1,\ldots,n}\big\{\|b_i\|_{L^{\infty}(\Omega)}\big\}, \|c\|_{L^{\infty}(\Omega)}\le \Lambda\,.
\end{align}
Then for any subdomain
\(\Omega'\subset\subset\Omega\),
\begin{align}\label{est:GT-L-2}
   \|u\|_{W^{2,p}(\Omega')}\le C\big(\|u\|_{L^{p}(\Omega)}+\|f\|_{L^{p}(\Omega)}\big)\,,
\end{align}
where \(C\) depends on \(n,p,\Lambda, \Omega'\), and \(\Omega\). 
\end{thm}
\begin{thm}[{\cite[Lemma~2.4.1.4]{Grisvard}}]\label{thm-Grisvard}
  Let \(\Omega\) be an open and bounded set in \(\R^{n}\), let $2 \leq
  p < \infty$, and let \(u\in W^{2,2}(\Omega)\) be a strong 
  solution of the equation
  \begin{align}\label{eq:GT-two-bis}
    {}-\Delta u = f
  \end{align}
  in \(\Omega\) with \(f\in L^p(\Omega)\).
  Then $u \in W^{2,p}_{{\rm loc}}(\Omega)$.
\end{thm}
%

\begin{thebibliography}{10}

\bibitem{Adams}
Robert~A. Adams and John J.~F. Fournier, \emph{Sobolev spaces}, second ed.,
  Pure and Applied Mathematics (Amsterdam), vol. 140, Elsevier/Academic Press,
  Amsterdam, 2003.

\bibitem{H-O-Ahlrichs-Morgan-2}
Reinhart Ahlrichs, Maria Hoffmann-Ostenhof, Thomas Hoffmann-Ostenhof, and
  John~D. Morgan~III, \emph{Bounds on the decay of electron densities with
  screening}, Phys. Rev. A (3) \textbf{23} (1981), no.~5, 2106--2117.

\bibitem{NistorEtAl}
Bernd Ammann, Catarina Carvalho, and Victor Nistor, \emph{Regularity for
  {E}igenfunctions of {S}chr\"odinger {O}perators}, Lett. Math. Phys.
  \textbf{101} (2012), no.~1, 49--84.

\bibitem{Bingel}
Werner~A. Bingel, \emph{The {B}ehaviour of the {F}irst-{O}rder {D}ensity
  {M}atrix at the {C}oulomb {S}ingularities of the {S}chr\"odinger {E}quation},
  Z. Naturforschg. \textbf{18 a} (1963), 1249--1253.

\bibitem{Ana-rel-HF}
Anna Dall'Acqua, S{\o}ren Fournais, Thomas {\O}stergaard~S{\o}rensen, and
  Edgardo Stockmeyer, \emph{Real analyticity away from the nucleus of
  pseudorelativistic {H}artree-{F}ock orbitals}, Anal. PDE \textbf{5} (2012),
  no.~3, 657--691.

\bibitem{Evans}
Lawrence~C. Evans, \emph{Partial {D}ifferential {E}quations}, second ed.,
  Graduate Studies in Mathematics, vol.~19, American Mathematical Society,
  Providence, RI, 2010.

\bibitem{Flad-et-al-1}
Heinz-J\"urgen Flad, Wolfgang Hackbusch, and Reinhold Schneider, \emph{Best
  {$N$}-term approximation in electronic structure calculations. {I}.
  {O}ne-electron reduced density matrix}, M2AN Math. Model. Numer. Anal.
  \textbf{40} (2006), no.~1, 49--61.

\bibitem{Flad-3}
Heinz-J\"urgen Flad and Gohar Harutyunyan, \emph{Ellipticity of quantum
  mechanical {H}amiltonians in the edge algebra}, Discrete Contin. Dyn. Syst.
  (2011), no.~Dynamical systems, differential equations and applications. 8th
  AIMS Conference. Suppl. Vol. I, 420--429.

\bibitem{Flad-1}
Heinz-J\"urgen Flad, Gohar Harutyunyan, Reinhold Schneider, and Bert-Wolfgang
  Schulze, \emph{Explicit {G}reen operators for quantum mechanical
  {H}amiltonians. {I}. {T}he hydrogen atom}, Manuscripta Math. \textbf{135}
  (2011), no.~3-4, 497--519.

\bibitem{Flad-2}
Heinz-J\"urgen Flad, Gohar Harutyunyan, and Bert-Wolfgang Schulze,
  \emph{Explicit {G}reen operators for quantum mechanical {H}amiltonians. {II}.
  {E}dge type singularities of the helium atom}, ArXiv e-prints {\tt
  1801.07552} (2018), 50 pp.

\bibitem{Flad-et-al-2}
Heinz-J\"urgen Flad, Reinhold Schneider, and Bert-Wolfgang Schulze,
  \emph{Asymptotic regularity of solutions to {H}artree-{F}ock equations with
  {C}oulomb potential}, Math. Methods Appl. Sci. \textbf{31} (2008), no.~18,
  2172--2201.

\bibitem{rho-smooth}
S{\o}ren Fournais, Maria Hoffmann-Ostenhof, Thomas Hoffmann-Ostenhof, and
  Thomas~{\O}stergaard S{\o}rensen, \emph{The {E}lectron {D}ensity is {S}mooth
  {A}way from the {N}uclei}, Comm. Math. Phys. \textbf{228} (2002), no.~3,
  401--415.

\bibitem{Taxco}
\bysame, \emph{On the regularity of the density of electronic wavefunctions},
  Mathematical {R}esults in {Q}uantum {M}echanics ({T}axco, 2001), Contemp.
  Math., vol. 307, Amer. Math. Soc., Providence, RI, 2002, pp.~143--148.

\bibitem{analytic}
\bysame, \emph{Analyticity of the density of electronic wavefunctions}, Ark.
  Mat. \textbf{42} (2004), no.~1, 87--106.

\bibitem{monster}
\bysame, \emph{Sharp {R}egularity {R}esults for {C}oulombic {M}any-electron
  {W}ave {F}unctions}, Comm. Math. Phys. \textbf{255} (2005), no.~1, 183--227.

\bibitem{KS}
\bysame, \emph{Analytic {S}tructure of {M}any-{B}ody {C}oulombic {W}ave
  {F}unctions}, Commun. Math. Phys. \textbf{289} (2009), no.~1, 291--310.

\bibitem{ks-hf}
\bysame, \emph{Analytic structure of solutions to multiconfiguration
  equations}, J. Phys. A: Math. Theor. \textbf{42} (2009), 315208.

\bibitem{thirdder}
S{\o}ren Fournais, Maria Hoffmann-Ostenhof, and Thomas~{\O}stergaard
  S{\o}rensen, \emph{Third {D}erivative of the {O}ne-{E}lectron {D}ensity at
  the {N}ucleus}, Ann. Henri Poincar\'e \textbf{9} (2008), no.~7, 1387--1412.

\bibitem{non-iso}
S{\o}ren Fournais, Thomas~{\O}stergaard S{\o}rensen, Maria Hoffmann-Ostenhof,
  and Thomas Hoffmann-Ostenhof, \emph{Non-{I}sotropic {C}usp {C}onditions and
  {R}egularity of the {E}lectron {D}ensity of {M}olecules at the {N}uclei},
  Ann. Henri Poincar\'e \textbf{8} (2007), no.~4, 731--748.

\bibitem{froese-herbst}
Richard Froese and Ira Herbst, \emph{Exponential {B}ounds and {A}bsence of
  {P}ositive {E}igenvalues for {$N$}-{B}ody {S}chr\"odinger {O}perators}, Comm.
  Math. Phys. \textbf{87} (1982), no.~3, 429--447.

\bibitem{GT}
David Gilbarg and Neil~S. Trudinger, \emph{Elliptic {P}artial {D}ifferential
  {E}quations of {S}econd {O}rder}, Classics in Mathematics, Springer-Verlag,
  Berlin, 2001, Reprint of the 1998 edition.

\bibitem{Grisvard}
Pierre Grisvard, \emph{Elliptic {P}roblems in {N}onsmooth {D}omains},
  Monographs and Studies in Mathematics, vol.~24, Pitman (Advanced Publishing
  Program), Boston, MA, 1985.

\bibitem{H-O-2-schr-ineq}
Maria Hoffmann-Ostenhof and Thomas Hoffmann-Ostenhof, \emph{``{S}chr\"odinger
  inequalities'' and asymptotic behavior of the electron density of atoms and
  molecules}, Phys. Rev. A (3) \textbf{16} (1977), no.~5, 1782--1785.

\bibitem{HO-alone}
\bysame, \emph{Local properties of solutions of {S}chr\"odinger equations},
  Comm. Partial Differential Equations \textbf{17} (1992), no.~3-4, 491--522.

\bibitem{H-O-Ahlrichs-Morgan-1}
Maria Hoffmann-Ostenhof, Thomas Hoffmann-Ostenhof, Reinhart Ahlrichs, and
  John~D. Morgan~III, \emph{On the exponential fall off of wavefunctions and
  electron densities}, Mathematical problems in theoretical physics ({P}roc.
  {I}nternat. {C}onf. {M}ath. {P}hys., {L}ausanne, 1979), Lecture Notes in
  Phys., vol. 116, Springer, Berlin-New York, 1980, pp.~62--67.

\bibitem{HO-Nadira}
Maria Hoffmann-Ostenhof, Thomas Hoffmann-Ostenhof, and Nikolai Nadirashvili,
  \emph{Interior {H}\"older estimates for solutions of {S}chr\"odinger
  equations and the regularity of nodal sets}, Comm. Partial Differential
  Equations \textbf{20} (1995), no.~7-8, 1241--1273.

\bibitem{AHP}
Maria Hoffmann-Ostenhof, Thomas Hoffmann-Ostenhof, and Thomas~{\O}stergaard
  S{\o}rensen, \emph{Electron {W}avefunctions and {D}ensities for {A}toms},
  Ann. Henri Poincar\'e \textbf{2} (2001), no.~1, 77--100.

\bibitem{HO-Stremnitzer}
Maria Hoffmann-Ostenhof, Thomas Hoffmann-Ostenhof, and Hanns Stremnitzer,
  \emph{Local {P}roperties of {C}oulombic {W}ave {F}unctions}, Comm. Math.
  Phys. \textbf{163} (1994), no.~1, 185--215.

\bibitem{Maria-Seiler}
Maria Hoffmann-Ostenhof and Ruedi Seiler, \emph{Cusp conditions for
  eigenfunctions of {$n$}-electron systems}, Phys. Rev. A (3) \textbf{23}
  (1981), no.~1, 21--23.

\bibitem{Hormander}
Lars H{\"o}rmander, \emph{Linear {P}artial {D}ifferential {O}perators}, Third
  revised printing. Die Grundlehren der mathematischen Wissenschaften, Band
  116, Springer-Verlag New York Inc., New York, 1969.

\bibitem{Jastrow}
Robert Jastrow, \emph{Many-body {P}roblem with {S}trong {F}orces}, Phys. Rev.
  \textbf{98} (1955), 1479--1484.

\bibitem{Jecko}
Thierry Jecko, \emph{A {N}ew {P}roof of the {A}nalyticity of the {E}lectronic
  {D}ensity of {M}olecules}, Lett. Math. Phys. \textbf{93} (2010), no.~1,
  73--83.

\bibitem{Kato-s-a}
Tosio Kato, \emph{Fundamental {P}roperties of {H}amiltonian {O}perators of
  {S}chr\"odinger {T}ype}, Trans. Amer. Math. Soc. \textbf{70} (1951),
  195--211.

\bibitem{Kato-reg}
\bysame, \emph{On the {E}igenfunctions of {M}any-{P}article {S}ystems in
  {Q}uantum {M}echanics}, Comm. Pure Appl. Math. \textbf{10} (1957), 151--177.

\bibitem{Leray}
Jean Leray, \emph{Sur les {S}olutions de l'{E}quation de {S}chr\"odinger
  {A}tomique et le {C}as {P}articulier de deux {E}lectrons}, Trends and
  {A}pplications of {P}ure {M}athematics to {M}echanics ({P}alaiseau, 1983),
  Lecture Notes in Phys., vol. 195, Springer, Berlin, 1984, pp.~235--247.

\bibitem{Mezey}
Paul~G. Mezey, \emph{The holographic electron density theorem and quantum
  similarity measures}, Molecular Physics \textbf{96} (1999), no.~2, 169--178.

\bibitem{Simon_notes}
Barry Simon, \emph{Exponential decay of quantum wave functions}, {\tt
  http://www.math.caltech.edu/simon/Selecta/ExponentialDecay.pdf}, Online
  notes, part of {\it Barry Simon's Online Selecta} at {\tt
  http://www.math.caltech.edu/simon/selecta.html}.

\bibitem{Si-semi}
\bysame, \emph{Schr\"odinger semigroups}, Bull. Amer. Math. Soc. (N.S.)
  \textbf{7} (1982), no.~3, 447--526.

\bibitem{Simon-on-Kato}
\bysame, \emph{{Tosio Kato's Work on Non--Relativistic Quantum Mechanics}},
  ArXiv e-prints {\tt 1711.00528} (2017), 215 pp.

\bibitem{Steiner}
Erich Steiner, \emph{Charge {D}ensities in {A}toms}, J. Chem. Phys. \textbf{39}
  (1963), no.~9, 2365--2366.

\end{thebibliography}

\providecommand{\bysame}{\leavevmode\hbox to3em{\hrulefill}\thinspace}
\providecommand{\MR}{\relax\ifhmode\unskip\space\fi MR }
\providecommand{\MRhref}[2]{%
  \href{http://www.ams.org/mathscinet-getitem?mr=#1}{#2}
}
\providecommand{\href}[2]{#2}

\end{document}